\documentclass[12pt]{amsart}
\usepackage{amsmath, amsthm, amssymb, amsfonts}
\usepackage{paralist}

\theoremstyle{plain}
\newtheorem{theorem}{Theorem}[subsection]
\newtheorem{corollary}[theorem]{Corollary}
\newtheorem{lemma}[theorem]{Lemma}

\newtheorem{proposition}[theorem]{Proposition}

\theoremstyle{definition}
\newtheorem{example}[theorem]{Example}
\newtheorem{definition}[theorem]{Definition}
\newtheorem{remark}[theorem]{Remark}

\theoremstyle{remark}

\newtheorem{assumption}{Assumption}

\newcommand{\clg}[1][G]{\ensuremath{{\mathbb{#1}}}}
\newcommand{\clgd}[1][G]{\ensuremath{{\mathbb{#1}^{\vee}}}}
\newcommand{\cla}[1][g]{\ensuremath{{\mathfrak{#1}}}}
\newcommand{\uea}[1][g]{\ensuremath{U(\cla[#1])}}
\newcommand{\vm}[1][w]{\ensuremath{\text{M}_{#1}}}
\newcommand{\vmi}[1][w]{\ensuremath{\text{L}_{#1}}}
\newcommand{\multInt}[2]{\ensuremath{n_{#1,#2}}}
\newcommand{\MultInt}[2]{\ensuremath{N_{#1,#2}}}

\newcommand{\multhc}[2]{\ensuremath{m(#1,#2)}}
\newcommand{\Multhc}[2]{\ensuremath{M(#1,#2)}}

\newcommand{\realGroupDual}[1][G]{\ensuremath{{#1}^{\vee}}}
\newcommand{\stdRep}[1][\pi]{\ensuremath{#1}}
\newcommand{\irrRep}[1][\pi]{\ensuremath{\bar{#1}}}
\newcommand{\irrReps}[1][G]{\ensuremath{\Pi(\realGroup[#1])}}
\newcommand{\irrRepsTriv}[1][G]{\ensuremath{\Pi_{\text{triv}}(\realGroup[#1])}}
\newcommand{\clgDual}[1][G]{\ensuremath{\clg[G]^{\vee}}}

\newcommand{\complexGroup}[1][G]{\ensuremath{{\mathbb{#1}}}}
\newcommand{\realGroup}[1][G]{\ensuremath{{#1}}}
\newcommand{\realGroupSplit}[1][G]{\ensuremath{{#1}^{\text{s}}}}
\newcommand{\realGroupCover}[1][G]{\ensuremath{{\widetilde{\realGroup[#1]}}}}
\newcommand{\complexTorus}[1][H]{\ensuremath{{\complexGroup[#1]}}}
\newcommand{\realTorus}[1][H]{\ensuremath{{\realGroup[#1]}}}
\newcommand{\realTorusId}[1][H]{\ensuremath{{\realGroup[#1]_{0}}}}
\newcommand{\realTorusSplit}[1][H]{\ensuremath{\realGroup[#1]^{\text{s}}}}
\newcommand{\realTorusCover}[1][H]{\ensuremath{{\widetilde{\realTorus[#1]}}}}
\newcommand{\realTorusCoverId}[1][H]{\ensuremath{{(\widetilde{\realTorus[#1]})_{0}}}}
\newcommand{\realTorusSplitCover}[1][H]{\ensuremath{{\realGroupCover[#1]}^{\text{s}}}}

\newcommand{\realTorusCt}[1][\simpleRoot]{\ensuremath{{\realGroup[H]_{#1}}}}
\newcommand{\realTorusCtCover}[1][\simpleRoot]{\ensuremath{{\widetilde{\realGroup[H]}_{#1}}}}
\newcommand{\realTorusCtNci}[1][\simpleRoot]{\ensuremath{\realGroup[H]^{#1}}}
\newcommand{\realTorusCtCoverNci}[1][\simpleRoot]{\ensuremath{{\widetilde{\realGroup[H]}^{#1}}}}
\newcommand{\realTorusCtCoverInt}[1][\simpleRoot]{\ensuremath{{\widetilde{\realGroup[H]}^{#1}_{1}}}}
\newcommand{\realTorusSplitCt}[1][\simpleRoot]{\ensuremath{{\realTorusSplit_{#1}}}}
\newcommand{\realTorusSplitCtCover}[1][\simpleRoot]{\ensuremath{{\widetilde{\realGroup[H]}^{\text{s}}_{#1}}}}
\newcommand{\realTorusSplitCtCoverId}[1][\simpleRoot]{\ensuremath{{(\widetilde{\realGroup[H]}^{\text{s}}_{#1})_{0}}}}
\newcommand{\realTorusCtGp}[2][\simpleRoot]{\ensuremath{{\realGroup[#2]_{#1}}}}
\newcommand{\realTorusCtNciGp}[2][\simpleRoot]{\ensuremath{{\realGroup[#2]^{#1}}}}
\newcommand{\realTorusCtId}[1][\simpleRoot]{\ensuremath{{(\realGroup[H]_{#1})_{0}}}}

\newcommand{\realNormalizer}[1][G]{\ensuremath{{N_{\realGroup[#1]}}}}

\newcommand{\maxRealCompact}{\ensuremath{{\realGroup[K]}}}
\newcommand{\maxRealCompactCover}{\ensuremath{{\widetilde{\realGroup[K]}}}}
\newcommand{\complexSpinGroup}[1][n]{\ensuremath{{\text{Spin}(#1,\mathbb{C})}}}
\newcommand{\realSpinGroup}[2]{\ensuremath{{\text{Spin}(#1,#2)}}}
\newcommand{\realSpinGroupCover}[2]{\ensuremath{{\widetilde{\text{Spin}}(#1,#2)}}}

\newcommand{\projOp}{\ensuremath{{\pi}}}
\newcommand{\projOpInv}{\ensuremath{{\projOp^{-1}}}}
\newcommand{\realTorusCoverT}{\ensuremath{{\widetilde{\realTorus[T]}}}}
\newcommand{\realTorusCoverA}{\ensuremath{{\widetilde{\realTorus[A]}}}}
\newcommand{\realTorusCoverTCt}[1][\simpleRoot]{\ensuremath{{\widetilde{\realTorus[T]}^{#1}}}}
\newcommand{\realTorusCoverACt}[1][\simpleRoot]{\ensuremath{{\widetilde{\realTorus[A]}^{#1}}}}
\newcommand{\realTorusCoverTCtInt}[1][\simpleRoot]{\ensuremath{{\widetilde{\realTorus[T]}_{1}^{#1}}}}

\newcommand{\realGroupCoverDual}[1][G]{\ensuremath{{\widetilde{\realTorus[#1]}}}^{\vee}}

\newcommand{\realTorusESplitCover}[1][H]{\ensuremath{{\realGroupCover[#1]}^{\text{es}}}}

\newcommand{\realTorusESplitCtCover}[1][\simpleRoot]{\ensuremath{{\widetilde{\realGroup[H]}^{\text{es}}_{#1}}}}
\newcommand{\realTorusESplitCtCoverId}[1][\simpleRoot]{\ensuremath{{(\widetilde{\realGroup[H]}^{\text{es}}_{#1})_{0}}}}
\newcommand{\realTorusESplitCoverId}[1][H]{\ensuremath{{(\realGroupCover[#1]^{\text{es}})_{0}}}}

\newcommand{\complexLieAlgebra}[1][g]{\ensuremath{{\mathfrak{#1}}}}
\newcommand{\complexLieAlgebraDual}[1][g]{\ensuremath{{\complexLieAlgebra[#1]^{*}}}}
\newcommand{\complexLieAlgebraCt}[1][\simpleRoot]{\ensuremath{{\mathfrak{h}_{#1}}}}
\newcommand{\complexLieAlgebraCtNci}[1][\simpleRoot]{\ensuremath{{\mathfrak{h}^{#1}}}}
\newcommand{\splitLieAlgebra}[1][h]{\ensuremath{{\mathfrak{#1}^{\text{s}}}}}
\newcommand{\splitLieAlgebraCt}[1][\simpleRoot]{\ensuremath{{\mathfrak{h}^{\text{s}}_{#1}}}}

\newcommand{\realLieAlgebra}[1][g]{\ensuremath{{\mathfrak{#1}_{\mathbb{R}}}}}
\newcommand{\realLieAlgebraCt}[1][\simpleRoot]{\ensuremath{{\left(\mathfrak{h}_{\mathbb{R}}\right)_{#1}}}}
\newcommand{\realLieAlgebraCtNci}[1][\simpleRoot]{\ensuremath{{\left(\mathfrak{h}_{\mathbb{R}}\right)^{#1}}}}

\newcommand{\uaeCenter}[1][g]{\ensuremath{{\mathcal{Z}(\complexLieAlgebra[g])}}}
\newcommand{\fixedInfChar}{\ensuremath{\chi}}
\newcommand{\absLieAlgebra}{\ensuremath{{\complexLieAlgebra[h]^{\text{abs}}}}}
\newcommand{\absLieAlgebraDual}{\ensuremath{{(\absLieAlgebra)^{*}}}}

\newcommand{\eSplitLieAlgebra}[1][h]{\ensuremath{{\mathfrak{#1}^{\text{es}}}}}
\newcommand{\eSplitLieAlgebraCt}[1][\simpleRoot]{\ensuremath{{\mathfrak{h}^{\text{es}}_{#1}}}}
\newcommand{\eSplitLieAlgebraDual}[1][h]{\ensuremath{{(\mathfrak{#1}^{\text{es}})^{*}}}}

\newcommand{\hcCat}{\ensuremath{\mathcal{HC}(\complexLieAlgebra, \maxRealCompact)}}
\newcommand{\hcCatCover}{\ensuremath{\mathcal{HC}(\complexLieAlgebra, \maxRealCompactCover)}}
\newcommand{\hcCatGen}{\ensuremath{\mathcal{HC}(\complexLieAlgebra, \maxRealCompactCover)^{\text{gen}}}}
\newcommand{\hcCatInf}{\ensuremath{\mathcal{HC}(\complexLieAlgebra, \maxRealCompact)_{\fixedInfChar}}}
\newcommand{\hcCatInfGen}{\ensuremath{\mathcal{HC}(\complexLieAlgebra, \maxRealCompactCover)^{\text{gen}}_{\fixedInfChar}}}

\newcommand{\hcCatInfGG}{\ensuremath{\mathcal{K}\mathcal{HC}(\complexLieAlgebra, \maxRealCompact)_{\fixedInfChar}}}
\newcommand{\hcCatInfGenGG}[1][\fixedInfChar]{\ensuremath{\mathcal{K}\mathcal{HC}(\complexLieAlgebra, \maxRealCompactCover)^{\text{gen}}_{#1}}}

\newcommand{\hcBasisGen}[1][\fixedInfChar]{\ensuremath{\widetilde{\mathcal{D}}_{#1}}}

\newcommand{\hcBasisGenInvOrderInf}[2]{\ensuremath{\left|\widetilde{\mathcal{D}}^{#1}_{#2}\right|}}
\newcommand{\hcIso}{\ensuremath{\rho}}

\newcommand{\hcBasisGenInvOrderCent}[1][\inv]{\ensuremath{\left|\widetilde{\mathcal{D}}^{#1}_{\absDiff}(\chi)\right|}}
\newcommand{\hcBasisGenCentSpin}[2]{\ensuremath{\widetilde{\mathcal{D}}_{\absDiff}(\chi)(#1,#2)}}
\newcommand{\hcBasisGenInvCentSpin}[2]{\ensuremath{\widetilde{\mathcal{D}}^{\inv}_{\absDiff}(\chi)(#1,#2)}}
\newcommand{\hcBasisGenInvOrderCentSpin}[2]{\ensuremath{\left|\widetilde{\mathcal{D}}^{\inv}_{\absDiff}(\chi)(#1,#2)\right|}}
\newcommand{\hcBasisGenInvOrderCentSpinDual}[2]{\ensuremath{\left|\widetilde{\mathcal{D}}^{-\inv}_{\absDiff}(\chi^{\vee})(#1,#2)\right|}}
\newcommand{\hcBasisGenInvCentSpinDual}[2]{\ensuremath{\widetilde{\mathcal{D}}^{-\inv}_{\absDiff}(\chi^{\vee})(#1^{\vee},#2^{\vee})}}
\newcommand{\hcBasisGenOrderCentSpin}[2]{\ensuremath{\left|\widetilde{\mathcal{D}}_{\absDiff}(\chi)(#1,#2)\right|}}
\newcommand{\hcBasisGenSpin}[2]{\ensuremath{\widetilde{\mathcal{D}}_{\absDiff}(#1,#2)}}
\newcommand{\hcBasisGenInvCentSpinInv}[3]{\ensuremath{\widetilde{\mathcal{D}}^{#3}_{\absDiff}(\chi)(#1,#2)}}
\newcommand{\hcBasisGenCentSpinInf}[3]{\ensuremath{\widetilde{\mathcal{D}}_{#3}(\chi)(#1,#2)}}
\newcommand{\hcBasisGenCentSpinInfDual}[3]{\ensuremath{\widetilde{\mathcal{D}}_{#3}(\chi^{\vee})(#1^{\vee},#2^{\vee})}}
\newcommand{\hcBasisGenCentSpinCent}[3]{\ensuremath{\widetilde{\mathcal{D}}_{\absDiff}(#3)(#1,#2)}}
\newcommand{\hcBasisSpin}[2]{\ensuremath{\mathcal{D}_{\fixedInfChar}(#1,#2)}}

\newcommand{\hcBasisGenInvSpin}[2]{\ensuremath{\widetilde{\mathcal{D}}^{\inv}_{\absDiff}(#1,#2)}}
\newcommand{\hcBasisGenInvOrderSpin}[2]{\ensuremath{\left|\widetilde{\mathcal{D}}^{\inv}_{\absDiff}(#1,#2)\right|}}
\newcommand{\hcBasisGenCentSpinInfOrder}[3]{\ensuremath{\left|\widetilde{\mathcal{D}}_{#3}(\chi)(#1,#2)\right|}}
\newcommand{\hcBasisGenCentSpinInfOrderDual}[3]{\left|\ensuremath{\widetilde{\mathcal{D}}_{#3}(\chi^{\vee})(#1^{\vee},#2^{\vee})\right|}}

\newcommand{\vSpace}{\ensuremath{\absLieAlgebraDual}}
\newcommand{\rootSystem}{\ensuremath{\Delta}}
\newcommand{\coRootSystem}{\ensuremath{\rootSystem^{\vee}}}
\newcommand{\posRootSystem}{\ensuremath{{\rootSystem^{+}}}}
\newcommand{\simpleRoots}{\ensuremath{{\Pi}}}
\newcommand{\rootLattice}{\ensuremath{{L(\rootSystem)}}}
\newcommand{\coRootLattice}{\ensuremath{{L(\coRootSystem)}}}
\newcommand{\posCorootLattice}[1][\inv]{\ensuremath{{L(\coRootSystem)^{#1}_{+}}}}
\newcommand{\negCorootLattice}[1][\inv]{\ensuremath{{L(\coRootSystem)^{#1}_{-}}}}
\newcommand{\posQuotCorootLattice}[1][\inv]{\ensuremath{{\mGroup^{#1}_{+}}}}
\newcommand{\posnegQuotCorootLattice}[1][\inv]{\ensuremath{{\mGroup^{#1}_{\pm}}}}
\newcommand{\negQuotCorootLattice}[1][\inv]{\ensuremath{{\mGroup^{#1}_{-}}}}
\newcommand{\weightLattice}{\ensuremath{{X}}}

\newcommand{\quotWeightLatticeMap}{\ensuremath{{\pi}}}

\newcommand{\weylGroup}{\ensuremath{{W}}}

\newcommand{\algWeylGroup}[2]{\ensuremath{{W(#1,#2)}}}
\newcommand{\rootReflection}[1][\alpha]{\ensuremath{{s_{#1}}}}
\newcommand{\ei}[1][i]{\ensuremath{{e_{#1}}}}
\newcommand{\bit}[1][i]{\ensuremath{{\epsilon_{#1}}}}
\newcommand{\symGroup}[1][n]{\ensuremath{{S_{#1}}}}
\newcommand{\inv}{\ensuremath{{\theta}}}
\newcommand{\invSplit}{\ensuremath{{\theta^{\text{s}}}}}

\newcommand{\invSplitCt}[1][\simpleRoot]{\ensuremath{{\theta^{\text{s}}_{#1}}}}
\newcommand{\weylGroupInv}{\ensuremath{{\mathcal{I}}}}
\newcommand{\weylGroupInvMap}{\ensuremath{{\Psi}}}
\newcommand{\invConjClasses}{\ensuremath{{\weylGroupInv / \weylGroup}}}
\newcommand{\cartanWeylGroup}[2][\realGroup]{\ensuremath{{W(#1,#2)}}}

\newcommand{\cptImaginaryRoots}[1][\inv]{\ensuremath{{\rootSystem^{#1}_{ic}}}}
\newcommand{\imaginaryRoots}[1][\inv]{\ensuremath{{\rootSystem^{#1}_{i}}}}
\newcommand{\realRoots}[1][\inv]{\ensuremath{{\rootSystem^{#1}_{\mathbb{R}}}}}
\newcommand{\complexRoots}[1][\inv]{\ensuremath{{\rootSystem^{#1}_{\mathbb{C}}}}}
\newcommand{\complexOrthogonalRoots}[1][\inv]{\ensuremath{{\rootSystem^{#1}_{\mathbb{C}^{\bot}}}}}
\newcommand{\imaginaryWeylGroup}[1][\inv]{\ensuremath{{\weylGroup^{#1}_{i}}}}
\newcommand{\cptImaginaryWeylGroup}[1][\inv]{\ensuremath{{\weylGroup^{#1}_{ic}}}}
\newcommand{\realWeylGroup}[1][\inv]{\ensuremath{{\weylGroup^{#1}_{\mathbb{R}}}}}
\newcommand{\complexWeylGroup}[1][\inv]{\ensuremath{{\weylGroup^{#1}_{\mathbb{C}}}}}
\newcommand{\invWeylGroup}[1][\inv]{\ensuremath{{\weylGroup^{#1}}}}
\newcommand{\rGroup}[1][\inv]{\ensuremath{{\mathcal{R}}}}
\newcommand{\rootSpace}[1][\simpleRoot]{\ensuremath{\complexLieAlgebra_{#1}}}

\newcommand{\rootVector}[1][\simpleRoot]{\ensuremath{X_{#1}}}

\newcommand{\coroot}[1][\simpleRoot]{\ensuremath{h_{#1}}}

\newcommand{\imaginaryBits}[1][\inv]{\ensuremath{{S_{#1}}}}
\newcommand{\realBits}[1][\inv]{\ensuremath{{R_{#1}}}}
\newcommand{\complexBits}[1][\inv]{\ensuremath{{C_{#1}}}}
\newcommand{\numImaginaryBits}[1][\inv]{\ensuremath{{n_{s}^{#1}}}}
\newcommand{\numRealBits}[1][\inv]{\ensuremath{{n_{r}^{#1}}}}
\newcommand{\numComplexBits}[1][\inv]{\ensuremath{{n_{c}^{#1}}}}
\newcommand{\numImaginaryBitsSimple}{\ensuremath{{n_{s}}}}

\newcommand{\imaginaryBitsBit}[1][\inv]{\ensuremath{{\epsilon_{s}^{#1}}}}
\newcommand{\realBitsBit}[1][\inv]{\ensuremath{{\epsilon_{r}^{#1}}}}
\newcommand{\imaginaryBitsParityBit}[1][\inv]{\ensuremath{{\epsilon_{p}^{#1}}}}
\newcommand{\realBitsParityBit}[1][\inv]{\ensuremath{{\epsilon_{m}^{#1}}}}
\newcommand{\numCptBits}[1][\inv]{\ensuremath{{n_{+}^{#1}}}}
\newcommand{\numNcptBits}[1][\inv]{\ensuremath{{n_{\oplus}^{#1}}}}
\newcommand{\NcptBitsBit}[1][\inv]{\ensuremath{{\epsilon_{\oplus}^{#1}}}}
\newcommand{\numIntBits}[1][\inv]{\ensuremath{N^{#1}_{1}(\absDiff)}}
\newcommand{\numHalfIntBits}[1][\inv]{\ensuremath{N^{#1}_{\frac{1}{2}}(\absDiff)}}
\newcommand{\symBit}[1][\inv]{\ensuremath{\epsilon^{#1}_{\text{sym}}(\absDiff)}}
\newcommand{\numNonParBits}[1][\inv]{\ensuremath{{n_{-}^{#1}}}}
\newcommand{\numParBits}[1][\inv]{\ensuremath{{n_{\ominus}^{#1}}}}

\newcommand{\diagram}[1][\inv]{\ensuremath{D_{#1}}}

\newcommand{\rCrossn}[1][]{\ensuremath{{(\mathbb{R}^{\times})^{#1}}}}

\newcommand{\sOnen}[1][]{\ensuremath{{(S^{1})^{#1}}}}

\newcommand{\cCrossn}[1][]{\ensuremath{{(\mathbb{C}^{\times})^{#1}}}}

\newcommand{\absDiff}{\ensuremath{{\lambda}}}
\newcommand{\diff}{\ensuremath{{\phi}}}
\newcommand{\genDiff}{\ensuremath{{\phi}}}
\newcommand{\torusChar}{\ensuremath{{\Gamma}}}
\newcommand{\genTorusChar}{\ensuremath{\widetilde{\torusChar}}}
\newcommand{\genTorusCharCt}[1][\simpleRoot]{\ensuremath{\genTorusChar^{#1}}}
\newcommand{\genTorusCharCtInt}[1][\simpleRoot]{\ensuremath{\genTorusChar^{#1}_{1}}}
\newcommand{\triple}{\ensuremath{{(\realTorus, \diff, \torusChar)}}}
\newcommand{\tripleInf}[1][\absDiff]{\ensuremath{{(\realTorus, \diff, \torusChar)_{#1}}}}

\newcommand{\gtRep}{\ensuremath{\Upsilon}}
\newcommand{\hcRep}{\ensuremath{\upsilon}}
\newcommand{\genTriple}{\ensuremath{{(\realTorusCover, \genDiff, \genTorusChar)}}}

\newcommand{\numCorGenTriplesInf}[1][\genPairInf]{\ensuremath{{\left[{#1}\right]}}}

\newcommand{\numCorGenTriplesSplitCtInf}[1][\genPairSplitCtInf]{\ensuremath{{\left[{#1}\right]}}}
\newcommand{\numCorGenTriplesSplitInf}[1][\genPairSplitInf]{\ensuremath{{\left[{#1}\right]}}}
\newcommand{\genTripleInf}[1][\absDiff]{\ensuremath{{(\realTorusCover, \genDiff, \genTorusChar)_{#1}}}}
\newcommand{\genTripleInfDiffChar}[2]{\ensuremath{{(\realTorusCover, #1, #2)_{\absDiff}}}}
\newcommand{\genTripleInfDiffCharInf}[3]{\ensuremath{{(\realTorusCover, #1, #2)_{#3}}}}

\newcommand{\genTripleInfCtDiffChar}[3]{\ensuremath{{(\realTorusCtCoverNci[#3], #1, #2)_{\absDiff}}}}

\newcommand{\genTripleSplitCtInf}[1][\rSeq]{\ensuremath{{(\realTorusSplitCtCover[#1], \genDiff, \genTorusChar)_{\absDiff}}}}

\newcommand{\absTriple}[1][\absDiff]{\ensuremath{{(\inv, \grading, #1)}}}

\newcommand{\absPair}[1][\grading]{\ensuremath{{(\inv, #1)}}}

\newcommand{\absBg}[1][\absDiff]{\ensuremath{{(\inv, \grading, \realGrading, #1)}}}
\newcommand{\absBgDual}[1][\absDiff]{\ensuremath{{(-\inv, \realGrading, \grading, #1)}}}
\newcommand{\absBgDualInv}[1][\absDiff]{\ensuremath{{(#1, \realGrading, \grading, \absDiff)}}}
\newcommand{\weylOrbit}{\ensuremath{{\Lambda}}}

\newcommand{\pairDual}{\ensuremath{{(\realTorus', \diff')}}}

\newcommand{\pairInf}{\ensuremath{{(\realTorus, \diff)_{\absDiff}}}}
\newcommand{\pairInfi}[1][i]{\ensuremath{{(\realTorus_{#1}, \diff)_{\absDiff}}}}
\newcommand{\pairInfDiff}[1][\diff]{\ensuremath{{(\realTorus, #1)_{\absDiff}}}}
\newcommand{\genPair}{\ensuremath{{(\realTorusCover, \diff)}}}

\newcommand{\genPairSplitCtInf}[1][\rSeq]{\ensuremath{{(\realTorusSplitCtCover[#1], \diff)_{\absDiff}}}}

\newcommand{\genPairInf}[1][\absDiff]{\ensuremath{{(\realTorusCover, \diff)_{#1}}}}

\newcommand{\absTripleDual}[1][\absDiff]{\ensuremath{{(-\inv, \realGrading, #1)}}}
\newcommand{\genTripleInfDual}[1][\absDiff]{\ensuremath{{(\realTorusCover^{\vee}, \genDiff^{\vee}, \genTorusChar^{\vee})_{#1}}}}
\newcommand{\absPairDual}[1][\realGrading]{\ensuremath{{(-\inv, #1)}}}

\newcommand{\genTripleESplitInf}[1][\absDiff]{\ensuremath{{(\realTorusESplitCover, \genDiff, \genTorusChar)_{#1}}}}
\newcommand{\genTripleESplitCtInf}[1][\rSeq]{\ensuremath{{(\realTorusESplitCtCover[#1], \genDiff, \genTorusChar)_{\absDiff}}}}

\newcommand{\genPairESplitCtInf}[1][\rSeq]{\ensuremath{{(\realTorusESplitCtCover[#1], \diff)_{\absDiff}}}}

\newcommand{\genTripleESplitInfCtDiffChar}[3]{\ensuremath{{(\realTorusESplitCtCover[#3], #1, #2)_{\absDiff}}}}
\newcommand{\numCorGenTriplesESplitCtInf}[1][\genPairESplitCtInf]{\ensuremath{{\left[{#1}\right]}}}

\newcommand{\absConj}[1][\diff]{\ensuremath{i_{#1}}}
\newcommand{\absConjInv}[1][\diff]{\ensuremath{i_{#1}^{-1}}}
\newcommand{\absConjTwo}[2][\diff]{\ensuremath{i_{#2,#1}}}
\newcommand{\absConjTwoInv}[2][\diff]{\ensuremath{i_{#2,#1}^{-1}}}

\newcommand{\mGroup}{\ensuremath{{M}}}

\newcommand{\mGroupTorus}[1][\realTorus]{\ensuremath{{M(#1)}}}
\newcommand{\mGroupTorusCt}[1][\simpleRoot]{\ensuremath{{M(\realTorus)_{#1}}}}
\newcommand{\mGroupTorusGpCt}[2][\simpleRoot]{\ensuremath{{M(#2)_{#1}}}}

\newcommand{\mGroupTorusCover}[1][\realTorusCover]{\ensuremath{{\widetilde{\mGroup}(#1)}}}
\newcommand{\simpleRoot}{\ensuremath{{\alpha}}}
\newcommand{\simpleRoota}{\ensuremath{{\alpha}}}
\newcommand{\simpleRootb}{\ensuremath{{\beta}}}
\newcommand{\simpleRootc}{\ensuremath{{\gamma}}}
\newcommand{\simpleRootd}{\ensuremath{{\delta}}}
\newcommand{\simpleRoote}{\ensuremath{{\epsilon}}}
\newcommand{\simpleRooti}[1][i]{\ensuremath{\alpha_{#1}}}
\newcommand{\simpleRootai}[1][i]{\ensuremath{\alpha_{#1}}}
\newcommand{\simpleRootbi}[1][i]{\ensuremath{\beta_{#1}}}
\newcommand{\simpleRootci}[1][i]{\ensuremath{\gamma_{#1}}}
\newcommand{\simpleRootdi}[1][i]{\ensuremath{\delta_{#1}}}
\newcommand{\simpleRootei}[1][i]{\ensuremath{\epsilon_{#1}}}
\newcommand{\rootChar}[1][\simpleRoot]{\ensuremath{\tilde{#1}}}
\newcommand{\ma}[1][\simpleRoot]{\ensuremath{m_{#1}}}
\newcommand{\mac}[1][\simpleRoot]{\ensuremath{\tilde{m}_{#1}}}

\newcommand{\rootCheck}[1][\simpleRoot]{\ensuremath{{#1^{\vee}}}}
\newcommand{\oa}[1][\simpleRoot]{\ensuremath{\sigma_{#1}}}
\newcommand{\oac}[1][\simpleRoot]{\ensuremath{\tilde{\sigma}_{#1}}}

\newcommand{\genReps}[1][\realTorusCover]{\ensuremath{{\Pi_{g}(#1)}}}
\newcommand{\genRepsCenter}[1][\realTorusCover]{\ensuremath{{\Pi_{g}(Z(#1))}}}

\newcommand{\zTwo}{\ensuremath{\mathbb{Z}_{2}}}

\newcommand{\cinv}{\ensuremath{{\Theta}}}
\newcommand{\grading}{\ensuremath{{\varepsilon}}}
\newcommand{\realGrading}{\ensuremath{{\eta}}}
\newcommand{\sOneCt}[1][\simpleRoot]{\ensuremath{{B_{#1}}}}
\newcommand{\Ad}[1]{\ensuremath{{\text{Ad}(#1)}}}
\newcommand{\ad}[1]{\ensuremath{{\text{ad}(#1)}}}
\newcommand{\rSeq}{\ensuremath{{\mathfrak{c}}}}
\newcommand{\pSeq}[2]{\ensuremath{{\mathfrak{c}_{#1}^{#2}}}}
\newcommand{\ctLaOp}[1][\simpleRoot]{\ensuremath{\xi_{#1}}}
\newcommand{\ctLaOpNci}[1][\simpleRoot]{\ensuremath{\zeta_{#1}}}
\newcommand{\ctOp}[1][\simpleRoot]{\ensuremath{\mathcal{C}_{#1}}}
\newcommand{\ctOpNci}[1][\simpleRoot]{\ensuremath{\mathcal{C}^{#1}}}

\newcommand{\ctOpSeq}[2]{\ensuremath{\mathcal{C}_{#1}^{#2}}}
\newcommand{\ctOpSeqInv}[2]{\ensuremath{(\mathcal{C}_{#1}^{#2})^{-1}}}
\newcommand{\ctOpRoot}[1][\simpleRoot]{\ensuremath{\overline{#1}}}
\newcommand{\ctOpRooti}[1][i]{\ensuremath{\overline{\simpleRoot}_{#1}}}
\newcommand{\zLa}[1][\simpleRoot]{\ensuremath{Z_{#1}}}
\newcommand{\expg}{\ensuremath{\text{exp}_{\realGroup}}}
\newcommand{\expgc}{\ensuremath{\text{exp}_{\realGroupCover}}}
\newcommand{\identity}{\ensuremath{\text{I}}}
\newcommand{\stCenter}{\ensuremath{z}}
\newcommand{\psMap}{\ensuremath{\wp}}

\newcommand{\numFlips}[1][\inv]{\ensuremath{n_{cr}^{#1}}}
\newcommand{\psInvMap}{\ensuremath{\tau}}
\newcommand{\cmplxBit}{\ensuremath{\epsilon}}
\newcommand{\longhookrightarrow}{\ensuremath{\lhook\joinrel\relbar\joinrel\rightarrow}}
\newcommand{\qSpec}{\hat{q}}
\newcommand{\pc}[2]{\ensuremath{\text{pc}_{\absDiff}(#1,#2)}}
\newcommand{\pcCent}[2]{\ensuremath{\text{pc}_{\absDiff}^{\chi}(#1,#2)}}

\newcommand{\pcCentOrder}[2]{\ensuremath{\left|\pcCent{#1}{#2}\right|}}
\newcommand{\pcMap}{\ensuremath{\wp}}
\newcommand{\pcMapDual}{\ensuremath{\wp^{\vee}}}
\newcommand{\pcInvMap}{\ensuremath{\tau}}
\newcommand{\pcCentInf}[3]{\ensuremath{\text{pc}_{#3}^{\chi}(#1,#2)}}
\newcommand{\pcCenter}[3]{\ensuremath{\text{pc}_{\absDiff}^{#3}(#1,#2)}}

\newcommand{\weylElt}{\ensuremath{w}}
\newcommand{\weylEltDiff}[1][\diff]{\ensuremath{{\weylElt_{#1}}}}
\newcommand{\weylEltDiffInv}[1][\diff]{\ensuremath{{\weylElt_{#1}^{-1}}}}
\newcommand{\cross}{\ensuremath{{\times}}}
\newcommand{\korbit}[1][\diff]{\ensuremath{{[#1]}}}
\newcommand{\fiber}[1][\inv]{\ensuremath{{#1^{\dagger}_{\absDiff}}}}
\newcommand{\genFiber}[1][\inv]{\ensuremath{{\tilde{#1}^{\dagger}_{\absDiff}}}}
\newcommand{\fiberOrder}[1][\inv]{\ensuremath{{\left|{\fiber[#1]}\right|}}}
\newcommand{\genFiberOrder}[1][\inv]{\ensuremath{{\left|{\genFiber[#1]}\right|}}}
\newcommand{\absTripleFiber}{\ensuremath{{\absTriple}^{\dagger}}}
\newcommand{\absTripleFiberOrder}{\ensuremath{{\left|\absTripleFiber\right|}}}
\newcommand{\mca}[2][\korbit]{\ensuremath{{m^{#1}_{#2}}}}
\newcommand{\halfSumIm}[1][\diff]{\ensuremath{\rho_{\text{i}}^{#1}}}
\newcommand{\halfSumImCpt}[1][\diff]{\ensuremath{\rho_{\text{ic}}^{#1}}}
\newcommand{\halfSumReal}[1][\diff]{\ensuremath{\rho_{\text{r}}^{#1}}}
\newcommand{\famInfChar}[1][\absDiff]{\ensuremath{\mathcal{F}(#1)}}
\newcommand{\ecaElt}[2]{\ensuremath{\mu_{#1}(#2)}}

\newcommand{\stdMod}[1][\gamma]{\ensuremath{\text{std}(#1)}}
\newcommand{\irrMod}[1][\gamma]{\ensuremath{\text{irr}(\overline{#1})}}
\newcommand{\mult}[2]{\ensuremath{\text{m}(\overline{#1},#2)}}
\newcommand{\multMatrix}{\ensuremath{m}}
\newcommand{\Mult}[2]{\ensuremath{\text{M}(#1,\overline{#2})}}
\newcommand{\MultMatrix}{\ensuremath{M}}
\newcommand{\len}[1][\hcRep]{\ensuremath{\ell(#1)}}

\newlength{\parLen}
\newenvironment{steplist}{\bigskip\begin{asparadesc}\setlength\itemsep{\bigskipamount}\setlength{\parindent}{\parLen}}{\end{asparadesc}}

\begin{document}

\begin{abstract}
Let $\complexGroup = \complexSpinGroup[4n+1]$ be the connected, simply connected complex Lie group of type $B_{2n}$ and let $\realGroup = \realSpinGroup{p}{q}$ $(p+q=4n+1)$ denote a (connected) real form. If $q \notin \left\{0,1\right\}$, $\realGroup$ has fundamental group $\mathbb{Z}_{2}$ and we denote the corresponding nonalgebraic double cover by $\realGroupCover = \realSpinGroupCover{p}{q}$. The main purpose of this paper is to describe a symmetry in the set genuine parameters for the various $\realGroupCover$ at certain half-integral infinitesimal characters. This symmetry is used to establish a duality of the corresponding generalized Hecke modules and ultimately results in a character multiplicity duality for the genuine characters of $\realGroupCover$.
\end{abstract}

\title{Vogan Duality for $\realSpinGroupCover{p}{q}$}
\author{Scott Crofts}
\date{\today}
\maketitle


\section{Introduction}
\label{secIntro}
\subsection{Duality for Verma Modules}
Before stating the main result for which we are aiming (Theorem \ref{theoremMainFirst}), we begin by recalling a familiar case. Let $\clg$ be a simple complex algebraic group with Lie algebra $\cla$ and write $\uea$ for its universal enveloping algebra. Choose a Cartan subalgebra $\cla[h] \subset \cla$ and let $\rootSystem = \rootSystem(\cla, \cla[h])$ be the corresponding root system. Write $\weylGroup = \weylGroup(\cla, \cla[h])$ for the Weyl group of $\rootSystem$ and fix a Borel subalgebra $\cla[b] \supset \cla[h]$. Then $\cla[b]$ induces a choice $\posRootSystem = \posRootSystem(\cla[g], \cla[h])$ of positive roots for $\rootSystem$ and we set $\rho$ to be the half sum of the elements of $\posRootSystem$. Finally, for $w \in \weylGroup$ let
\begin{eqnarray*}
\vm & = & \uea \underset{\uea[b]}{\otimes} \mathbb{C}_{w\rho - \rho}
\end{eqnarray*}
be the Verma module of highest weight $w\rho - \rho$ and write $\vmi$ for its unique irreducible quotient.

A fundamental problem in the representation theory of Verma modules is to determine the composition factors (with multiplicities) of $\vm$. Each such factor is known to be of the form $\vmi[y]$, for some $y \in \weylGroup$. On the level of formal characters we seek a decomposition of the form
\begin{eqnarray*}
\text{ch}(\vm) & = & \sum_{y \in \weylGroup} \multInt{y}{w}\text{ ch}(\vmi[y]) \hspace{0.25in} \multInt{y}{w} \in \left\{0,1,2,\ldots\right\}
\end{eqnarray*}
where it remains to compute the numbers $\multInt{y}{w}$. A related problem asks for a type of inverse decomposition
\begin{eqnarray*}
\text{ch}(\vmi) & = & \sum_{y \in \weylGroup} \MultInt{y}{w}\text{ ch}(\vm[y]) \hspace{0.25in} \MultInt{y}{w} \in \mathbb{Z}
\end{eqnarray*}
of an irreducible module in terms of Verma modules. A conjectural algorithm for computing the multiplicities $\multInt{y}{w}$ and $\MultInt{y}{w}$ was first given by Kazhdan and Lusztig in \cite{KL1}. Their method uses the combinatorics of Hecke algebras to inductively build polynomials (now known to be) related to singularities of the corresponding Schubert varieties. The celebrated Kazhdan-Lusztig conjecture asserts that evaluating these polynomials at one gives the desired multiplicities. It was later proven correct in the work of Brylinski-Kashiwara and Beilinson-Bernstein \cite{BK1}, \cite{BB1}.

A striking feature of the Kazhdan-Lusztig theory is the existence of a type symmetry in the multiplicities $\multInt{y}{w}$ and $\MultInt{y}{w}$. Let $w_{0}$ denote the longest element of $\weylGroup$ and define the \emph{duality map}
\begin{eqnarray*}
\weylGroupInvMap : \weylGroup & \longrightarrow & \weylGroup \\
w & \longmapsto & w_{0}w.
\end{eqnarray*}
Then $\weylGroupInvMap$ is an involution of $\weylGroup$ that implements the familiar up-down symmetry of the corresponding Coxeter graph. Moreover, $\weylGroupInvMap$ defines a type of dual Hecke module whose combinatorics are formally reversed. On the level of multiplicities this dualization induces the equality
\begin{eqnarray}
\label{cmDuality}
\MultInt{y}{w} & = & \pm\multInt{\weylGroupInvMap(w)}{\weylGroupInvMap(y)}
\end{eqnarray}
for all $y,w \in \weylGroup$. Geometrically, this suggests the representation theory of Verma modules is dual to the singular structure of certain Schubert varieties. This observation is central to \cite{ABV}, with duality playing a key role in a reformulated version of the local Langlands conjecture. 


\subsection{Duality for Real Groups}
There is an analogous result for a large class of real Lie groups. Fix a real form $\realGroup \subset \clg$ and let $\irrReps$ denote the set of equivalence classes of irreducible admissible representations of $\realGroup$. For $\irrRep \in \irrReps$, write $\stdRep$ for a standard representation (Section \ref{ssHCModIntro}) corresponding to $\irrRep$. For simplicity (and by analogy with the case above), we restrict our attention to the finite set $\irrRepsTriv$ of representations with trivial infinitesimal character. If $\irrRep \in \irrRepsTriv$, we again seek to understand the composition factors of $\stdRep$ as
\begin{eqnarray*}
\stdRep = \sum_{\irrRep[\eta] \in \irrRepsTriv} \multhc{\irrRep[\eta]}{\stdRep}\irrRep[\eta] \hspace{0.25in} \multhc{\irrRep[\eta]}{\stdRep} \in \left\{0,1,2,\ldots\right\}
\end{eqnarray*}
with the sum interpreted in an appropriate Grothendieck group. Similarly, the inverse problem asks for an expression (again in a Grothendieck group) for $\irrRep$ in terms of standard representations
\begin{eqnarray*}
\irrRep = \sum_{\irrRep[\eta] \in \irrRepsTriv} \Multhc{\stdRep[\eta]}{\irrRep}\stdRep[\eta] \hspace{0.25in} \Multhc{\stdRep[\eta]}{\irrRep} \in \mathbb{Z}.
\end{eqnarray*}
Once again, it remains to compute the numbers $\multhc{\irrRep[\eta]}{\stdRep}$ and $\Multhc{\stdRep[\eta]}{\irrRep}$.

This problem was solved by Vogan for reductive linear Lie groups (\cite{IC1},\cite{IC2},\cite{IC3}) in an extension of the results cited above. In this setting, $\irrRepsTriv$ is no longer parameterized by elements of a Weyl group, but rather a collection of geometric parameters (roughly, irreducible equivariant local systems on the flag manifold). These parameters form a natural basis for a Hecke module whose combinatorics embody a generalized Kazhdan-Lusztig algorithm. The resulting polynomials (sometimes called Kazhdan-Lusztig-Vogan polynomials) again give the desired multiplicities through evaluation at one.

In a remarkable final paper \cite{IC4}, Vogan describes a generalization of the duality in (\ref{cmDuality}). If we write
\begin{eqnarray*}
\irrRep[\pi]_{1} \sim \irrRep[\pi]_{2} & \Longleftrightarrow & \multhc{\irrRep_{1}}{\stdRep_{2}} \ne 0
\end{eqnarray*}
then $\sim$ generates an equivalence relation whose equivalence classes are called \emph{blocks}. Given a block $B = \left\{\irrRep_{1},\ldots,\irrRep_{n}\right\} \subset \irrRepsTriv$, Vogan constructs a real form $\realGroupDual \subset \clgd$ of the complex dual group, a block $B^{\vee} = \left\{\irrRep[\eta]_{1},\ldots,\irrRep[\eta]_{n}\right\} \subset \irrReps[\realGroupDual]$, and a bijection $\weylGroupInvMap : B \to B^{\vee}$. As before, the map $\weylGroupInvMap$ commutes with the combinatorics of the Hecke module for $B$ allowing one to define a dual Hecke module (\cite{IC4}, Definition 13.3) isomorphic to the one for $B^{\vee}$. On the level of multiplicities this implies
\begin{eqnarray}
\label{cmDuality2}
\Multhc{\stdRep[\pi]_{i}}{\irrRep[\pi]_{j}} & = & \pm \multhc{\weylGroupInvMap(\irrRep_{j})}{\weylGroupInvMap(\stdRep_{i})}
\end{eqnarray}
(\cite{IC4}, Theorem 1.15) as desired.

Theorem \ref{theoremMainFirst} below is a version of (\ref{cmDuality2}) for the nonalgebraic universal covers of even rank spin groups in type B (see Definition \ref{defGCover} for a precise description of the groups being considered). Blocks are replaced with the (computationally easier) notion of a central character (Definition \ref{defCC}) and the dual groups are constructed explicitly (Definition \ref{defDualGroup}). Moreover, unlike the type B linear groups in Vogan's theory, the groups of Definition \ref{defGCover} are preserved by the duality map.

Theorem \ref{theoremMainFirst} adds to the work of a number of people. The Kazhdan-Lusztig-Vogan algorithm has been generalized for a large class of nonlinear groups by Renard and Trapa \cite{RT3}. Their method builds extended Hecke algebra structures to track nonintegral wall crossings and computes KLV-polynomials for a (potentially large) number of infinitesimal characters simultaneously.  As in \cite{IC4}, it is ultimately a duality of these extended structures (Theorem \ref{theoremHeckeModuleDuality}) that is used to establish (\ref{cmDuality2}). Along these lines, Adams and Trapa have established a duality theory for nonalgebraic double covers whose corresponding root system is simply laced \cite{TA} (for technical reasons their results include $\text{G}_{2}$). Finally, Renard and Trapa build a general duality theory for the metaplectic group (the nontrivial double cover the symplectic group) in \cite{RT1} by directly extending ideas in \cite{IC4}.

\subsection{Overview}

As described above, proving Theorem \ref{theoremMainFirst} requires an explicit description of the main ingredients in a nonlinear Kazhdan-Lusztig-Vogan algorithm. Fortunately the authors of \cite{RT3} outline a general method for proving duality theories of the kind above. Just as before, what is needed is a version of the map $\weylGroupInvMap$ that intertwines the relevant combinatorial operations. The primary focus of this paper is a specialization of this approach to nonlinear, even rank type B. Although conceptually simple, the details quickly become quite technical and ultimately resist the usual generalizations of \cite{IC4} that were successful in the other cases.

Section \ref{secPreliminaries} introduces notation and explicitly describes the collection of nonalgebraic groups we are considering (Definition \ref{defGCover}). We also recall the Langlands classification in a form that is convenient for our purposes, along with a precise statement of character-multiplicity duality (Theorem \ref{theoremMainFirst}).

Section \ref{secWeylGroupInvolutions} describes the set of in involutions in a Weyl group of type B (denoted $\weylGroupInv$) and defines 
\[
\weylGroupInvMap : \weylGroupInv \to \weylGroupInv
\]
(Definition \ref{defWeylGroupInvMap}) at this level. Although elementary, the inherent symmetries induced by $\weylGroupInvMap$ are suggestive of a deeper theory (Corollary \ref{corNumInvConjugate}).

Sections \ref{secCartanSubgroups} and \ref{secGenTriples} determine the structure of Cartan subgroups for both linear and nonlinear groups of type B. In the nonlinear case, Cartan subgroups need not be abelian (Proposition \ref{propMCStruct}) and methods are developed to handle this (Proposition \ref{propCompH}). We also describe how to recover the real form from basic data associated to a particular Cartan subgroup (Proposition \ref{propRealForm}) and make precise the specific infinitesimal characters in which we are interested (Section \ref{ssInfChar}).

Section \ref{secKorbits} describes a combinatorial approach to the theory of $\maxRealCompact$-orbits. Typically $\maxRealCompact$-orbits refer to geometric partitions of the flag manifold given by the action of a complexified maximal compact subgroup. While this perspective is essential in all known proofs of the KLV-algorithm, only the combinatorics are important for computational purposes. Therefore we ignore the underlying geometry altogether and focus on a combinatorial description.

Section \ref{secCentralMa} is the final and most technical structure theory chapter. Here we develop the notions of abstract bigradings (Definition \ref{defAbsBigrading}) and central characters (Definition \ref{defCC}). These objects are used to define the dual group (Definition \ref{defDualGroup}) and describe a type of `numerical duality' in the space of nonlinear parameters (Theorem \ref{theoremNumericalDuality}). This is another precursor to the map $\weylGroupInvMap$ and is almost sufficient to give a complete definition. What ultimately remains is to distinguish a collection of representations that are grouped naturally in packets of order two (Corollary \ref{corRepFiberCent}). The solution to this problem is technical and postponed until the next section.

Section \ref{secAlmostCentralMa} outlines the main operations of the (nonlinear) KLV-algorithm (Definitions \ref{defCrossActionTriples}, \ref{defExtendedCrossActionTriples}, \ref{defCTTriples}) and contains a complete definition of the map $\weylGroupInvMap$. The key to this definition is the notion of a principal class (Definition \ref{defPrincipalClass}) and the corresponding principal class map $\pcMap$ (Definition \ref{defPsMap}). The map $\pcMap$ assigns a principal class to representations grouped in 2-packets, and this assignment is bijective (Theorem \ref{theoremPsMapBij}). In particular, principal classes distinguish the remaining representations and allow us to unambiguously define $\weylGroupInvMap$. We then show $\weylGroupInvMap$ has the desired properties (Theorems \ref{theoremDualityCommutesWithCA}, \ref{theoremDualityCommutesWithCT}) and the main duality theorems (Theorems \ref{theoremHeckeModuleDuality}, \ref{theoremMainSecond}) follow formally.

Initially, it may help to consider only the special case where $\realGroupCover = \realSpinGroupCover{2n+1}{2n}$ is split. In fact, if $\fixedInfChar$ is a symmetric infinitesimal character (i.e.~the integral Weyl group is of type $\text{B}_{n} \times \text{B}_{n}$), then $\realGroupCoverDual = \realGroupCover$ and principal classes are simply genuine principal series representations. The general theory was very much developed with this case in mind. Additionally, the material in each section is described from (and motivated by) a computational perspective. The author has implemented most of what follows in software and verified the main results in low rank via computer. 

\subsection{Future Directions}

The results discussed here beg the question of odd rank (type B) and we expect to return to this is a future paper. Although there are many similarities, the techniques required are sufficiently different and require separate discussion. In particular, the treatment involves disconnected nonlinear covering groups. 

Combined with the results of \cite{RT1} and \cite{TA}, the only remaining simple case is F$_4$. The author has verified computationally a duality theory exists here and expects to address this in a future paper as well. With a complete (albeit ad hoc) duality theory of simple nonlinear double covers at hand, the possibility of finding a uniform approach (of the kind described for linear groups by Vogan) exists. Moreover, it is reasonable to expect the duality for type B to take a leading role since this is the only classical case that encompasses all of the nonlinear phenomenon of \cite{RT3}.

More broadly, we remark that some flavor of duality plays a key role in other important results in the field. A future possibility then is to interpret nonlinear duality as a bridge for extending existing theory to nonlinear groups. As mentioned above, duality is central to the results of \cite{ABV} suggesting a uniform duality theory might allow an extension of the Langlands formalism to nonlinear groups. Foundations for this approach are described by Adams and Trapa in \cite{TA}. Another example is the theory of character lifting. Adams and Herb describe character lifting for nonlinear simply laced groups in \cite{AH}. Their work possesses formal properties closely related to the duality of \cite{TA} suggesting one could use a nonlinear duality theory to create a general notion of character lifting from linear to nonlinear groups.

\subsection{Acknowledgments}

The author thanks Peter Trapa for his patience and expertise in support of the work discussed here. It is also a pleasure to thank Jeffrey Adams and David Vogan for many helpful conversations along the way.


\section{Preliminaries}
\label{secPreliminaries}
\subsection{Notation}
We begin with the following notational definition.

\begin{definition}
\label{defG}
Throughout these notes let $\complexGroup = \complexSpinGroup[2n+1]$ be the connected, simply connected complex Lie group of type $B_{n}$. Up to equivalence $\complexGroup$ has $n+1$ (connected) real forms, denoted $\realSpinGroup{p}{q}$, with $p+q = 2n+1$ and $p > q$ (by convention).
\end{definition}

If $\realGroup = \realSpinGroup{p}{q}$, Lie algebras for $\complexGroup$ and $\realGroup$ will be denoted by $\complexLieAlgebra$ and $\realLieAlgebra$ respectively with analogous notation used for subgroups and subalgebras. If $\cinv$ is a Cartan involution for $\realGroup$, write $\maxRealCompact = \realGroup^{\cinv}$ for the corresponding maximal compact subgroup. Given a $\cinv$-stable Cartan subgroup $\realGroup[H] \subset \realGroup$, let $\rootSystem(\complexLieAlgebra[g], \complexLieAlgebra[h])$ be the induced root system and $\algWeylGroup{\complexLieAlgebra}{\complexLieAlgebra[h]}$ the algebraic Weyl group. Finally, the Killing form $(\cdot,\cdot)$  is a nondegenerate inner product on $\complexLieAlgebra[h]$ allowing us to identify $\complexLieAlgebra[h] \leftrightarrow \complexLieAlgebra[h]^{\ast}$. For some calculations it will be convenient to use $(\cdot,\cdot)$ to view roots and coroots as living in the same vector space.


\subsection{The Langlands Classification and $\realSpinGroup{p}{q}$}
\label{ssHCModIntro}

We begin by recalling the Langlands classification for algebraic groups. The discussion in this section is valid whenever $\realGroup$ is a real reductive linear Lie group with abelian Cartan subgroups. 

Let $\hcCat$ denote the category of Harish-Chandra modules for $\realGroup$ and suppose $\uaeCenter$ is the center of its universal enveloping algebra. Each irreducible object $X \in \hcCat$ has a corresponding infinitesimal character $\fixedInfChar$ of $\uaeCenter$ and we will only consider $X$ for which $\fixedInfChar$ is nonsingular. For any maximal torus $\complexTorus \subset \complexGroup$, the Harish-Chandra isomorphism
\[
\hcIso : \uaeCenter \to S(\complexLieAlgebra[h])^{\weylGroup(\complexLieAlgebra,\complexLieAlgebra[h])}
\]
allows us to identify infinitesimal characters with $\algWeylGroup{\complexLieAlgebra}{\complexLieAlgebra[h]}$ orbits in $\complexLieAlgebra[h]^{*}$.

Write $\hcCatInf$ for the full subcategory of $\hcCat$ consisting of modules with infinitesimal character $\fixedInfChar$ and let $\hcCatInfGG$ denote its Grothendieck group. The irreducible objects in $\hcCatInf$ (or equivalently their distribution characters) form a natural basis of $\hcCatInfGG$ that we wish to understand. Typically this is achieved with the help of a second basis given by equivalence classes of \emph{standard modules} (or equivalently their distribution characters) \cite{VGr}. Loosely, standard modules are representations induced from discrete series on cuspidal parabolic subgroups of $\realGroup$ and for our purposes we may assume these objects are known. Each standard module in $\hcCatInf$ has a canonical irreducible subquotient creating a one-to-one correspondence between the standard and irreducible modules (on the level of distribution characters).

Let $\realGroup[H] \subset \realGroup$ be a $\cinv$-stable Cartan subgroup and write $\realGroup[H] = \realGroup[T]\realGroup[A]$ for its decomposition into compact and vector pieces. The main ingredients of the Langlands classification are given in the following definition.

\begin{definition} 
\label{defTriple}
Suppose $\realGroup = \realSpinGroup{p}{q}$ and let $\hcBasisSpin{p}{q}$ be the set of $\maxRealCompact$-conjugacy classes of triples $\triple$, where $\realTorus = \realGroup[T]\realGroup[A] \subset \realGroup$ is a $\cinv$-stable Cartan subgroup, $\diff$ is in the $\algWeylGroup{\complexLieAlgebra}{\complexLieAlgebra[h]}$-orbit determined by $\fixedInfChar$ (viewed in $\complexLieAlgebra[h]^{*}$), and $\torusChar$ is a character of $\realGroup[T]$ whose differential is determined by $\diff$. Specifically, we must have
\begin{eqnarray*}
\text{d}\torusChar & = & \diff|_{\complexLieAlgebra[t]} + \halfSumIm - 2\halfSumImCpt
\end{eqnarray*}
where $\halfSumIm$ and $\halfSumImCpt$ are the half sums of positive (with respect to $\diff$) imaginary and compact imaginary roots. We refer to (\cite{VPc}, Chapter 3) for more details. Although technically redundant, we will occasionally write $\tripleInf[\fixedInfChar]$ for a representative triple in $\hcBasisSpin{p}{q}$ when we wish to emphasize the infinitesimal character $\fixedInfChar$. 
\end{definition}

\begin{theorem}[Langlands - see \cite{IC3}]
The set $\hcBasisSpin{p}{q}$ parameterizes the simple objects in $\hcCatInf$.
\end{theorem}

In particular, $\hcBasisSpin{p}{q}$ is a finite set that can be parameterized in terms of certain structure theoretic information for $\realGroup$. A large portion of what follows is dedicated to describing this parameterization in type B. The set $\hcBasisSpin{p}{q}$ will be referred to as the set of \emph{linear parameters} for $\realGroup$ at $\fixedInfChar$.


\subsection{The Langlands Classification and $\realSpinGroupCover{p}{q}$}
\label{ssGenHcMod}
Recall $\realGroup = \realSpinGroup{p}{q}$ is a real form of $\complexGroup$ with $p+q=2n+1$ and $p>q$. Although $\complexGroup$ is simply connected, if $q \notin \left\{1,0\right\}$, $\realGroup$ has fundamental group $\mathbb{Z}_2$. In particular, the (nonlinear) universal cover $\realGroupCover$ is a central extension of $\realGroup$ and we have a short exact sequence
\[
1 \to \left\{\pm 1\right\} \to \realGroupCover \to \realGroup \to 1.
\]
Let $\projOp : \realGroupCover \to \realGroup$ be the projection map and follow the usual convention that preimages of subgroups under $\projOp$ are denoted by adding a tilde. 

\begin{definition}
\label{defGCover}
If $q \notin \left\{0,1\right\}$, write $\realGroupCover = \realSpinGroupCover{p}{q}$ for the nonlinear universal cover of $\realGroup$. To simplify notation, set $\realGroupCover = \realSpinGroupCover{p}{q} = \realSpinGroup{p}{q}$ whenever $q \in \left\{0,1\right\}$ and take the map $\projOp$ to be trivial.
\end{definition}

\begin{definition}
\label{defCSGCover}
A \emph{Cartan subgroup} $\realTorusCover \subset \realGroupCover$ is defined to be the centralizer in $\realGroupCover$ of a Cartan subalgebra $\complexLieAlgebra[h] \subset \complexLieAlgebra$.
\end{definition}

 As the notation of Definition \ref{defCSGCover} suggests, $\realTorusCover = \projOpInv(\realTorus)$ whenever $\realTorus$ is a Cartan subgroup of $\realGroup$. In particular, the real Weyl group
\[
\cartanWeylGroup[\realGroupCover]{\realTorusCover} = \realNormalizer[\realGroupCover](\realTorusCover) / Z_{\realGroupCover}(\realTorusCover)
\]
is naturally isomorphic to $\cartanWeylGroup{\realTorus}$. These facts allow us to reduce many questions about Cartan subgroups in $\realGroupCover$ to equivalent ones about Cartan subgroups in $\realGroup$.

\begin{definition}
\label{defGenModule}
Suppose $\realGroupCover = \realSpinGroupCover{p}{q}$ and $\genTriple$ is a triple for $\realGroupCover$ as in Definition \ref{defTriple}. A module (respectively triple) is said to be \emph{genuine} if the action of $-1$ (respectively $\genTorusChar(-1)$) is nontrivial. 
\end{definition}

Write $\hcCatCover$ for the category of Harish-Chandra modules for $\realGroupCover$ and let $\hcCatGen \subset \hcCatCover$ denote the full subcategory of genuine modules in $\hcCatCover$. The following proposition implies the irreducible genuine objects in $\hcCatGen$ are parameterized in the same fashion as for $\realGroup$.

\begin{proposition}[\cite{RT3}, Proposition 6.1]
The genuine irreducible objects in $\hcCatGen$ are parameterized by $\maxRealCompactCover$-conjugacy classes of genuine triples $\genTriple$.
\end{proposition}

\begin{definition}
\label{defHcBasisGen}
In the setting of Definition \ref{defGenModule}, let $\hcBasisGen(p,q)$ denote the set of $\maxRealCompactCover$-conjugacy classes of genuine triples $\genTriple$ where the orbit of $\genDiff \in \complexLieAlgebra[h]^{*}$ is determined by $\fixedInfChar$. As before, $\hcBasisGen(p,q)$ will be referred to as the set of \emph{genuine parameters} for $\realGroupCover$ at infinitesimal character $\fixedInfChar$. If $q \in \left\{0,1\right\}$, set $\hcBasisGen(p,q) = \hcBasisSpin{p}{q}$.
\end{definition}

Our first goal will be to understand $\hcBasisGen(p,q)$ (as a set) at certain half-integral infinitesimal characters. We then turn to understanding the ingredients for the Kazhdan-Lusztig algorithm and definition of the map $\weylGroupInvMap$.
 

\subsection{Character Multiplicity Duality}
\label{ssIntroCMDuality}
We come now to an important definition.

\begin{definition}
\label{defCMDuality}
If $\gamma \in \hcBasisGen(p,q)$ is a genuine parameter for $\realGroupCover = \realSpinGroupCover{p}{q}$, denote by $\stdMod[\gamma]$ and $\irrMod[\gamma]$ the corresponding standard and irreducible modules in $\hcCatInfGen$, respectively. Write $\mult{\gamma}{\delta} \in \mathbb{N}$ for the number of times $\irrMod$ appears as a subquotient of $\stdMod[\delta]$. In $\hcCatInfGenGG$ we have
\[
\stdMod[\delta] = \sum_{\gamma \in \hcBasisGen}\mult{\gamma}{\delta}\irrMod.
\]
Similarly write $\Mult{\gamma}{\delta} \in \mathbb{Z}$ for the multiplicity of $\stdMod$ in $\irrMod[\delta]$ and
\[
\irrMod[\delta] = \sum_{\gamma \in \hcBasisGen}\Mult{\gamma}{\delta}\stdMod
\]
in $\hcCatInfGenGG$. 
\end{definition}

The uniqueness of the above expressions implies
\[
\sum_{\pi \in \hcBasisGen}\Mult{\gamma}{\pi}\mult{\pi}{\delta} = \left\{
\begin{array}{ll}
1 & \gamma = \delta \\
0 & \gamma \ne \delta
\end{array}
\right.
\]
so that the matrices $\multMatrix$ and $\MultMatrix$ are inverses. The integers $\Mult{\gamma}{\delta}$ are thus of fundamental importance. We can now state the main theorem we aim to prove (see Theorem \ref{theoremMainSecond}).

\begin{theorem}
\label{theoremMainFirst}
Let $\absDiff$ be a half-integral infinitesimal character (Section \ref{ssInfChar}) and suppose the rank of $\realGroupCover = \realSpinGroupCover{p}{q}$ is even. Fix a genuine central character $\genTorusChar$ of $\text{Z}(\realGroupCover)$ (Definition \ref{defCC}) and let $\mathcal{B} = \left\{\gamma_{1},\ldots,\gamma_{r}\right\}$ be the collection of genuine parameters in $\hcBasisGenSpin{p}{q}$ with central character $\genTorusChar$. Then there is a group $\realGroupCoverDual = \realSpinGroupCover{p^{\vee}}{q^{\vee}}$, a subset $\mathcal{B}' \subset \hcBasisGenSpin{p^{\vee}}{q^{\vee}}$, and a bijection $\weylGroupInvMap : \mathcal{B} \to \mathcal{B}'$ such that
\[
\Mult{\gamma_{i}}{\gamma_{j}} = \epsilon_{ij} \mult{\weylGroupInvMap(\gamma_{j})}{\weylGroupInvMap(\gamma_{i})}
\]
where $\epsilon_{ij} = \pm 1$.
\end{theorem}

Central characters are discussed in Section \ref{ssCentralCharacter} and the group $\realGroupCoverDual$ is defined in Section \ref{ssNumDuality}. The map $\weylGroupInvMap$ will be defined on various sets throughout the course of these notes (Definitions \ref{defWeylGroupInvMap}, \ref{defPsiOnHcBasisGen}). Theorem \ref{theoremMainFirst} will be proved in Section \ref{ssCharMultDuality} and is essentially a formal consequence of the structure and representation theoretic properties of the map $\weylGroupInvMap$ (Theorems \ref{theoremDualityCommutesWithCA}, \ref{theoremDualityCommutesWithCT}, and \ref{theoremDualityCommutesWithECA}).


\section{Involutions in $\weylGroup$}
\label{secWeylGroupInvolutions}

The set of involutions in a Weyl group of type $B_{n}$ will be of fundamental importance in what follows. In this elementary section we develop notation and recall various properties of these objects.


\subsection{Abstract Weyl Group}
\label{ssAbsWeylGroup}
We begin with a fundamental definition.

\begin{definition}
\label{defAbsCartan}
Recall $\complexLieAlgebra$ denotes the complex Lie algebra of $\complexGroup = \complexSpinGroup[2n+1]$ and let $\absLieAlgebra \subset \complexLieAlgebra$ be a fixed, \emph{abstract} Cartan subalgebra. Write $\weylGroup = \weylGroup(\complexLieAlgebra, \absLieAlgebra)$ for the Weyl group and let $\rootSystem = \rootSystem(\complexLieAlgebra, \absLieAlgebra) \subset \vSpace$ be the corresponding root system. With the usual choice of coordinates and inner product on $\vSpace$ we have
\[
\rootSystem = \left\{ \pm \ei \pm \ei[j] \mid 1 \le i < j \le n \right\} \cup \left\{ \pm \ei \mid 1 \le i \le n \right\}
\]
where $(\ei, \ei[j]) = \delta_{ij}$. Let $\posRootSystem$ denote the set of positive roots
\[
\posRootSystem = \left\{ \ei \pm \ei[j] \mid 1 \le i < j \le n \right\} \cup \left\{ \ei \mid 1 \le i \le n \right\}
\]
and $\simpleRoots$ the corresponding simple roots
\[
\simpleRoots = \left\{ \ei - \ei[i+1] \mid 1 \le i \le n-1 \right\} \cup \{\ei[n]\}.
\]
\end{definition}

Abstractly
\[
\weylGroup \cong \mathbb{Z}_{2}^{n} \rtimes \symGroup \cong \langle \rootReflection \mid \alpha \in \simpleRoots \rangle
\]
where $\rootReflection$ is the root reflection in the simple root $\alpha$. For $w \in \weylGroup$, write $w = \displaystyle{(\bit[1]\bit[2] \ldots \bit[n], \sigma)}$, where $\bit \in \left\{ {0,1} \right\}$ and $\sigma \in \symGroup$. The $\bit$ appearing in such an expression will be referred to as \emph{bits}. 

\begin{definition}
\label{defInvParams}
Given $w \in \weylGroup$ define
\begin{eqnarray*}
\imaginaryBits[w] & = & \left\{ \bit \mid \sigma (i) = i \text{ and } \bit = 0 \right\} \\
\realBits[w] & = & \left\{ \bit \mid \sigma (i) = i \text{ and } \bit = 1 \right\} \\
\complexBits[w] & = & \left\{ \bit \mid \sigma (i) \ne i \right\}
\end{eqnarray*}
and
\begin{eqnarray*}
\numImaginaryBits[w] & = & \lvert \imaginaryBits[w] \rvert \\
\numRealBits[w] & = & \lvert \realBits[w] \rvert \\
\numComplexBits[w] & = & \lvert \complexBits[w] \rvert.
\end{eqnarray*}
Then $\numImaginaryBits[w]$ counts the number of bits fixed by $\sigma$ that are equal to zero, $\numRealBits[w]$ counts the number fixed bits equal to one, and $\numComplexBits[w]$ counts the number of bits not fixed by $\sigma$. When the element $w \in \weylGroup$ is clear from context, we may write $\numImaginaryBitsSimple$ for $\numImaginaryBits[w]$ and so forth. It is often important to know when $w$ possesses certain properties and for this purpose we define the \emph{indicator bits}
\begin{eqnarray*}
\imaginaryBitsBit[w] & = & \left\{\begin{array}{rl}
0 & \numImaginaryBits[w] = 0 \\
1 & \numImaginaryBits[w] \ne 0
\end{array} \right. \\
\realBitsBit[w] & = & \left\{\begin{array}{rl}
0 & \numRealBits[w] = 0 \\
1 & \numRealBits[w] \ne 0
\end{array} \right. \\
\imaginaryBitsParityBit[w] & = & \left\{\begin{array}{rl}
0 & \numImaginaryBits[w] \text{ is even} \\
1 & \numImaginaryBits[w] \text{ is odd}
\end{array} \right. \\
\realBitsParityBit[w] & = & \left\{\begin{array}{rl}
0 & \numRealBits[w] \text{ is even} \\
1 & \numRealBits[w] \text{ is odd}
\end{array} \right..
\end{eqnarray*}
\end{definition}

An element $\weylElt \in \weylGroup$ with $\weylElt^{2} = 1$ is called an \emph{involution} and $\weylGroupInv \subset \weylGroup$  will denote the set of involutions in $\weylGroup$. We now describe a convenient way of representing involutions in terms of combinatorial objects called \emph{diagrams}. A diagram is a picture of the action of an involution on the standard basis for $\vSpace$ (equivalently the short roots in $\posRootSystem$). Let $\inv \in \weylGroupInv$ and let $\ei$ denote the $i$th standard basis vector. The corresponding diagram $\diagram$ is a sequence of symbols (read left to right) where the $i$th symbol represents the action of $\inv$ on $\ei$ in an obvious way.

\begin{example}
The diagram for $\inv = (0001,(23))$ is given by
\[
+~3~2~-.
\]
Here the symbol $+$ represents that $\ei[1]$ is fixed by $\inv$ and the symbol $-$ represents that $\ei[4]$ is sent to $-\ei[4]$ by $\inv$. The vectors $\ei[3]$ and $\ei[2]$ are interchanged by $\inv$ and this is represented by placing the number $3$ in position $2$ and the number $2$ in position $3$.
\end{example}

It is also possible for $\inv$ to interchange and negate two standard basis vectors. Negation of interchanged vectors will be represented in diagrams by placing parentheses around the corresponding numbers. 

\begin{example}
If $\inv = (1110,(23))$ the corresponding diagram is
\[
-~(3)~(2)~+.
\]
\end{example}

The map $\weylGroupInvMap$ (Section \ref{secIntro}) has a simple definition on the level of involutions.

\begin{definition}
\label{defWeylGroupInvMap}
Let $\inv = (\bit[1]\bit[2] \ldots \bit[n], \sigma) \in \weylGroupInv$ and define
\[
\weylGroupInvMap(\inv) = -\inv = (\tilde{\bit[1]} \tilde{\bit[2]} \ldots \tilde{\bit[n]}, \sigma)
\]
where $\tilde{\bit} = \bit + 1$. For diagrams, $\weylGroupInvMap$ interchanges $+$ and $-$ signs as well as changes the parentheses on transpositions. The examples above represent two diagrams interchanged by $\weylGroupInvMap$.
\end{definition}


\subsection{Conjugacy Classes in $\weylGroupInv$}
\label{ssCountInv}
It will be important to understand the $\weylGroup$-conjugacy classes in $\weylGroupInv$. We begin with the following proposition whose proof is easy.

\begin{proposition}
\label{propCC} 
Two involutions $\inv_{1} = \displaystyle{(\bit[1]\bit[2] \ldots \bit[n], \sigma_{1})}$ and $\inv_{2} = \displaystyle{(\bit[1]'\bit[2]' \ldots \bit[n]', \sigma_{2}')}$ are conjugate in $\weylGroup$ if and only if $\sigma_{1}$ and $\sigma_{2}$ are conjugate in $\symGroup$ and $\numImaginaryBits[\inv_{1}] = \numImaginaryBits[\inv_{2}]$ (equivalently $\numRealBits[\inv_{1}] = \numRealBits[\inv_{2}]$). In particular, the conjugacy class of $\inv$ is uniquely determined by the numbers $\numComplexBits$ and $\numImaginaryBits$.
\end{proposition}

In terms of diagrams, Proposition \ref{propCC} implies the conjugacy class of $\inv$ is determined by the number of numbers ($\numComplexBits$) and the number of $+$ signs ($\numImaginaryBits$) that appear. Therefore a general element of $\weylGroupInv$ is conjugate to one whose diagram is of the form
\[
2~1~4~3\cdots +~+~\cdots +~-~-\cdots -.
\]

\begin{corollary} 
\label{corNumCC}
In the setting above
\begin{eqnarray*}
\left| \invConjClasses \right| & = & \sum_{k=0}^{\lfloor n/2 \rfloor}{n-2k+1}.
\end{eqnarray*}
If $n$ is even, this number is a perfect square.
\end{corollary}
 
Since $\numComplexBits[\weylGroupInvMap(\inv)] = \numComplexBits$, Proposition \ref{propCC} implies $\weylGroupInvMap$ descends to an involution (still denoted $\weylGroupInvMap)$ on $\invConjClasses$. In particular, $\weylGroupInvMap$ sends the conjugacy class determined by $(\numComplexBits, \numImaginaryBits)$ to the one determined by $(\numComplexBits, \numRealBits)$.


\subsection{Involutions and Root Systems}
\label{ssInvRootSystems}
Let $\inv \in \weylGroupInv$ and recall $\inv$ acts on $\rootSystem = \rootSystem(\complexLieAlgebra,\absLieAlgebra)$. We begin by recalling the following fundamental definitions.

\begin{definition}
\label{defRootTypes}
Let
\begin{eqnarray*}
\imaginaryRoots & = & \left\{ \alpha \in \rootSystem \mid \inv (\alpha) = \alpha \right\} \\
\realRoots & = & \left\{ \alpha \in \rootSystem \mid \inv (\alpha) = -\alpha \right\} \\
\complexRoots & = & \left\{ \alpha \in \rootSystem \mid \inv (\alpha) \ne \pm \alpha \right\}
\end{eqnarray*}
denote the \emph{imaginary, real, and complex roots} for $\inv$. The imaginary and real roots form subsystems of $\rootSystem$ and we denote their corresponding Weyl (sub)groups by $\imaginaryWeylGroup$ and $\realWeylGroup$. The set $\complexRoots$ is \emph{not} a root system, however if we write
\begin{eqnarray*}
\halfSumIm[] & = & \frac{1}{2}\sum_{\simpleRoot \in (\imaginaryRoots)^{+}}\simpleRoot \\
\halfSumReal[] & = & \frac{1}{2}\sum_{\simpleRoot \in (\realRoots)^{+}}\simpleRoot \\
\complexOrthogonalRoots & = & \left\{\simpleRoot \in \rootSystem \mid (\simpleRoota,\halfSumIm[]) = (\simpleRoota,\halfSumReal[]) = 0 \right\}
\end{eqnarray*}
then $\complexOrthogonalRoots$ \emph{is} a root system consisting of complex roots (see \cite{IC4}, Definition 3.10).
\end{definition}

\begin{proposition}[\cite{IC4}, Lemma 3.11]
\label{propComplexOrthogonalWeylGroup}
In the setting of Definition \ref{defRootTypes}, we can write
\begin{eqnarray*}
\complexOrthogonalRoots & = & \rootSystem_{1} \cup \rootSystem_{2}
\end{eqnarray*}
as an orthogonal disjoint union with $\inv(\rootSystem_{1}) = \rootSystem_{2}$. In particular, 
\begin{eqnarray*}
\complexWeylGroup & = & \left\{(\weylElt,\inv\weylElt) \mid \weylElt \in \weylGroup(\rootSystem_{1})\right\}
\end{eqnarray*}
is a Weyl subgroup of $\weylGroup$ isomorphic to $\weylGroup(\rootSystem_{1})$.
\end{proposition}

\begin{proposition}
\label{propInvWeylGroups}
Let $\inv \in \weylGroupInv$ and let $k = \frac{\numComplexBits}{2}$. Then we have
\begin{eqnarray*}
\imaginaryWeylGroup & \cong & W(B_{\numImaginaryBits}) \times W(A_{1})^{k}\\
\realWeylGroup & \cong & W(B_{\numRealBits}) \times W(A_{1})^{k}\\
\complexWeylGroup & \cong & W(A_{k-1})
\end{eqnarray*}
where $W(A_{i})$ and $W(B_{i})$ are Weyl groups for root systems of type $A_{i}$ and $B_{i}$, respectively.
\end{proposition}

Let $\invWeylGroup$ denote the centralizer in $\weylGroup$ of the element $\inv \in \weylGroupInv$. Corollary \ref{corNumCC} and the following theorem allow us to count elements of order two in $\weylGroup$.

\begin{theorem}[\cite{IC4}, Proposition 3.12]
\label{theoremWeylGroupCentralizer}
The subgroups $\imaginaryWeylGroup$ and $\realWeylGroup$ are normal subgroups in $\invWeylGroup$. Moreover we have 
\begin{eqnarray*}
\invWeylGroup & \cong & (\imaginaryWeylGroup \times \realWeylGroup) \rtimes \complexWeylGroup.
\end{eqnarray*}
\end{theorem}

Using Proposition \ref{propInvWeylGroups} and Theorem \ref{theoremWeylGroupCentralizer} we easily obtain the following corollary.

\begin{corollary}
\label{corNumInvConjugate}
If $\inv \in \weylGroupInv$, the number of involutions in $\weylGroup$ conjugate to $\inv$ is given by
\[
\frac{n!}{{(n-\numComplexBits)!}{(\frac{\numComplexBits}{2})!}} \binom {n-\numComplexBits}{\numImaginaryBits}.
\]
\end{corollary}

Fix an even number $0 \le k \le n$ and consider the set
\begin{eqnarray*}
X_{k} & = & \left\{\inv \in \weylGroupInv \mid \numComplexBits = k \right\}.
\end{eqnarray*}
Proposition $\ref{propCC}$ implies $\weylGroupInvMap$ preserves $X_{k}$ and thus permutes its $n-k+1$ conjugacy classes. Note the left term in Corollary \ref{corNumInvConjugate} depends only on the numbers $k$ and $n$ and is therefore constant for fixed $X_{k}$. Hence the order two symmetry observed in the map $\weylGroupInvMap$ appears numerically as the familiar symmetry of binomial coefficients.


\section{Cartan Subgroups}
\label{secCartanSubgroups}
A key ingredient in the Langlands classification (Sections \ref{ssHCModIntro}, \ref{ssGenHcMod}) for a real group is a concrete understanding of its (conjugacy classes of) Cartan subgroups. In this section we explicitly describe the structure of the Cartan subgroups in $\realGroup = \realSpinGroup{p}{q}$ and $\realGroupCover = \realSpinGroupCover{p}{q}$ (Definition \ref{defGCover}). Although this material is well-known, the particular constructions used here will be important in what follows.

\subsection{Abstract Pairs}
\label{ssInvAndGradings}

Let $\realTorus$ be a $\cinv$-stable Cartan subgroup of $\realGroup = \realSpinGroup{p}{q}$. The complexified Lie algebra $\complexLieAlgebra[h] \subset \complexLieAlgebra$ is also $\cinv$-stable and thus $\cinv$ acts on $\rootSystem(\complexLieAlgebra,\complexLieAlgebra[h])$. Since $\realGroup$ contains a compact Cartan subgroup, the action of $\cinv$ is equivalent to the regular action of an involution in $\weylGroup(\complexLieAlgebra, \complexLieAlgebra[h])$ \cite{BC1}. 

\begin{definition}
\label{defAbsInv}
Choose a conjugation map $\absConj[] : \absLieAlgebra \to \complexLieAlgebra[h]$. The \emph{abstract} involution corresponding to $\complexLieAlgebra[h]$ and $\absConj[]$ is given by
\[
\inv = \absConj[]^{-1} \cdot \cinv \cdot \absConj[].
\]
The involution $\inv$ depends on the choice of $\absConj[]$ up to conjugacy in $\weylGroup = \weylGroup(\complexLieAlgebra,\absLieAlgebra)$.
\end{definition}

\begin{proposition}[\cite{BC1}]
\label{propInvCartanMap}
Let $\realTorus' \subset \realGroup$ be a $\cinv$-stable Cartan subgroup, choose a conjugation map
\[
j : \absLieAlgebra \to \complexLieAlgebra[h']
\]
and write $\inv' = j^{-1} \cdot \cinv \cdot j$ for the corresponding abstract involution. Then $\realTorus$ and $\realTorus'$ are $\maxRealCompact$-conjugate in $\realGroup$ if and only if $\inv$ and $\inv'$ are conjugate in $\weylGroup$.
\end{proposition}

\begin{remark}
\label{remInvCartanMap}
Proposition \ref{propInvCartanMap} implies we have a well-defined \emph{injection}
\begin{eqnarray*}
\left\{\begin{array}{c}
\cinv\text{-stable Cartan} \\
\text{subgroups}\end{array}\right\} / \maxRealCompact & \longhookrightarrow & \left\{\begin{array}{c}
\text{Involutions} \\
\text{in }\weylGroup\end{array}\right\} / \weylGroup.
\end{eqnarray*}
In particular, Corollary \ref{corNumCC} gives an upper bound on the number of conjugacy classes of $\cinv$-stable Cartan subgroups of $\realGroup$. This map is \emph{bijective} if and only if $\realGroup$ is split \cite{AdamsFokko}.
\end{remark}


\begin{definition}
\label{defGrading}
A \emph{grading} of a root system $\rootSystem$ is a map $\grading : \rootSystem \to \mathbb{Z}_{2}$ such that
\begin{eqnarray*}
\grading(\simpleRoota) & = & \grading(-\simpleRoota)
\end{eqnarray*}
for all $\simpleRoota \in \rootSystem$. Moreover if $\simpleRoota, \simpleRootb,$ and $\simpleRoot + \simpleRootb$ are in $\rootSystem$ we require
\begin{eqnarray*}
\grading(\simpleRoota + \simpleRootb) & = & \grading(\simpleRoota) + \grading(\simpleRootb).
\end{eqnarray*}
\end{definition}

Let $\realTorus$ be $\cinv$-stable Cartan subgroup of $\realGroup$. If $\simpleRoot \in \rootSystem(\complexLieAlgebra,\complexLieAlgebra[h])$ is imaginary (Definition \ref{defRootTypes}), the corresponding root space $\rootSpace$ is fixed by $\cinv$ and thus entirely contained in the positive or negative eigenspace. We say $\simpleRoot$ is \emph{compact} if $\rootSpace$ is contained in the positive eigenspace and \emph{noncompact} if it is contained in the negative eigenspace.

\begin{proposition}[\cite{IC4}]
\label{propGrading}
Let $\realTorus$ be a $\cinv$-stable Cartan subgroup of $\realGroup$ and let $\imaginaryRoots[\cinv](\complexLieAlgebra,\complexLieAlgebra[h])$ denote the imaginary roots for $\realTorus$ with respect to $\cinv$. Define a map $\eta : \imaginaryRoots[\cinv](\complexLieAlgebra,\complexLieAlgebra[h]) \to \mathbb{Z}_{2}$ via
\begin{eqnarray*}
\eta (\simpleRoot ) & = & \left\{
\begin{array}{ll}
0 & \simpleRoot \text{ compact} \\
1 & \simpleRoot \text{ noncompact} \\
\end{array} \right.
\end{eqnarray*}
Then $\eta$ is a grading on $\imaginaryRoots[\cinv](\complexLieAlgebra,\complexLieAlgebra[h])$.
\end{proposition}

\begin{corollary} \label{corCptRoots}
The set $\cptImaginaryRoots[\cinv](\complexLieAlgebra,\complexLieAlgebra[h])$ of imaginary compact roots is itself a root system.
\end{corollary}

Proposition \ref{propGrading} gives a natural grading $\eta$ on the imaginary roots $\imaginaryRoots[\cinv](\complexLieAlgebra, \complexLieAlgebra[h])$ determined by $\cinv$. For the following definition, recall our choice of conjugation map
\[
\absConj[] : \absLieAlgebra \to \complexLieAlgebra[h].
\]

\begin{definition}
\label{defAbsPair}
The \emph{abstract pair} corresponding to $\complexLieAlgebra[h]$ and $\absConj[]$ is the ordered pair $\absPair$, where $\inv$ is the abstract involution of Definition \ref{defAbsInv} and $\grading$ is the \emph{abstract grading} on $\imaginaryRoots(\complexLieAlgebra, \absLieAlgebra)$ defined via $\grading(\simpleRootb) = \eta(\simpleRoot)$, with $\simpleRootb = \simpleRoot \cdot \absConj[]$. We say an abstract imaginary root $\simpleRootb$ is compact if $\grading(\simpleRootb) = 0$ and noncompact if $\grading(\simpleRootb) = 1$. Note $\absPair$ depends on the choice of $\absConj[]$ up to conjugation in $\weylGroup(\complexLieAlgebra,\absLieAlgebra)$.
\end{definition}

Let $\absPair$ be and abstract pair and recall $\imaginaryRoots(\complexLieAlgebra, \absLieAlgebra)$ is  a root system of type $B_{m} \times A_{1}^{k}$, where $m=\numImaginaryBits$ and $k=\frac{\numComplexBits}{2}$ (Proposition \ref{propInvWeylGroups}). The long imaginary roots that form the $A_{1}^{k}$ factor can each be written as a sum of two short complex roots interchanged by $\inv$. It is easy to verify that such roots must be noncompact. In particular, not all possible abstract gradings arise from the construction in Definition \ref{defAbsPair}.

It will be convenient to incorporate abstract gradings into our involution diagrams (Section \ref{ssAbsWeylGroup}). The above discussion suggests we need only modify the portion of the diagram corresponding to the $B_{m}$ factor of $\imaginaryRoots$. In a root system of type $B_m$, any grading is determined by its values on a set of positive short roots. Since these are exactly the roots represented by $+$ signs in the diagram for $\inv$, a grading can be described by indicating the $+$ signs that represent noncompact roots. We do this by drawing a circle around the corresponding $+$ signs. 

\begin{example}
Suppose $m=6$ and $\inv=1$ so that all roots are imaginary. If 
\begin{eqnarray*}
\grading(\ei) & = & \left\{
\begin{array}{rl}
1 & 1 \le i \le 3 \\
0 & 4 \le i \le 6 \\
\end{array}\right.
\end{eqnarray*}
is an abstract grading, the corresponding diagram $\diagram(\grading)$ is
\[
\oplus~\oplus~\oplus~+~+~+.
\]
\end{example}

\subsection{Cayley Transforms}
\label{ssCT}

We now recall a fundamental tool for transferring information between conjugacy classes of Cartan subgroups. Suppose $\realTorus=\realGroup[T]\realGroup[A]$ is a $\cinv$-stable Cartan subgroup of $\realGroup$ with Lie algebra $\realLieAlgebra[h]$. Let $\realRoots[\cinv](\complexLieAlgebra, \complexLieAlgebra[h])$ and $\imaginaryRoots[\cinv](\complexLieAlgebra, \complexLieAlgebra[h])$ denote the roots in $\rootSystem(\complexLieAlgebra, \complexLieAlgebra[h])$ that are real and imaginary with respect to $\cinv$.

\begin{lemma}[\cite{VGr}, Lemma 4.3.7]
\label{lemXaSign}

Suppose $\simpleRoot \in \realRoots[\cinv](\complexLieAlgebra, \complexLieAlgebra[h])$. Then there exists a root vector $\rootVector \in \rootSpace \cap \realLieAlgebra$ with the property
\begin{eqnarray*}
\left[\cinv\rootVector,\rootVector\right] & = & \coroot,
\end{eqnarray*}
where $\coroot \in \complexLieAlgebra[h]$ is the coroot for $\simpleRoot$. Moreover, the vector $\rootVector$ is unique up to sign.
\end{lemma}

\begin{lemma}
\label{lemXbSign}
Suppose $\simpleRootb \in \imaginaryRoots[\cinv](\complexLieAlgebra, \complexLieAlgebra[h])$ is imaginary and noncompact. Then there exists a root vector $\rootVector[\simpleRootb] \in \rootSpace[\simpleRootb]$ with the property
\begin{eqnarray*}
\left[\rootVector[\simpleRootb],\overline{\rootVector[\simpleRootb]}\right] & = & \coroot[\simpleRootb],
\end{eqnarray*}
where $\coroot[\simpleRootb] \in \complexLieAlgebra[h]$ is the coroot for $\simpleRootb$. Moreover, the vector $\rootVector[\simpleRootb]$ is unique up to sign.
\end{lemma}

\begin{lemma}
\label{lemXcSign}
Suppose $\simpleRootc \in \imaginaryRoots[\cinv](\complexLieAlgebra, \complexLieAlgebra[h])$ is imaginary and compact. Then there exists a root vector $\rootVector[\simpleRootc] \in \rootSpace[\simpleRootc]$ with the property
\begin{eqnarray*}
\left[\overline{\rootVector[\simpleRootc]},\rootVector[\simpleRootc]\right] & = & \coroot[\simpleRootc],
\end{eqnarray*}
where $\coroot[\simpleRootc] \in \complexLieAlgebra[h]$ is the coroot for $\simpleRootc$. Moreover, the vector $\rootVector[\simpleRootc]$ is unique up to sign.
\end{lemma}

\begin{definition}
\label{defCT}
Suppose $\simpleRoota \in \realRoots[\cinv](\complexLieAlgebra, \complexLieAlgebra[h])$, $\simpleRootb \in \imaginaryRoots[\cinv](\complexLieAlgebra,\complexLieAlgebra[h])$ is noncompact, and choose nonzero root vectors $\rootVector \in \rootSpace$ and $\rootVector[\simpleRootb] \in \rootSpace[\simpleRootb]$ according to Lemmas \ref{lemXaSign} and \ref{lemXbSign}. Let
\begin{eqnarray*}
\ctLaOp & = & \frac{\pi i}{4}\left(\cinv\rootVector - \rootVector\right) \\
\ctLaOpNci[\simpleRootb] & = & \frac{\pi}{4}\left(\overline{\rootVector[\simpleRootb]} - \rootVector[\simpleRootb]\right)
\end{eqnarray*}
and define the corresponding \emph{Cayley transform operators}
\begin{eqnarray*}
\ctOp & = & \Ad{\text{exp}(\ctLaOp)} \\
\ctOpNci[\simpleRootb] & = & \Ad{\text{exp}(\ctLaOpNci[\simpleRootb])}.
\end{eqnarray*}
\end{definition}

The Cayley transform operators depend on the choices of $\rootVector$ and $\rootVector[\simpleRootb]$ up to inverse. On the level of Cartan subalgebras we have \cite{KGr}
\begin{eqnarray*}
\ctOp(\realLieAlgebra[h]) & = & \text{ker}(\simpleRoot|_{\realLieAlgebra[h]}) \oplus \mathbb{R}(\rootVector +\cinv\rootVector) \\
\ctOpNci[\simpleRootb](\realLieAlgebra[h]) & = & \text{ker}(\simpleRootb|_{\realLieAlgebra[h]}) \oplus \mathbb{R}(\rootVector[\simpleRootb] + \overline{\rootVector[\simpleRootb]}).
\end{eqnarray*}
Let
\begin{eqnarray*}
\realLieAlgebraCt & = & \ctOp(\realLieAlgebra[h]) \\
\realLieAlgebraCtNci[\simpleRootb] & = & \ctOpNci[\simpleRootb](\realLieAlgebra[h])
\end{eqnarray*}
denote the \emph{Cayley transforms} of $\realLieAlgebra[h]$ with respect to the roots $\simpleRoot$ and $\simpleRootb$ respectively. These are $\cinv$-stable Cartan subalgebras that are \emph{not} $\maxRealCompact$-conjugate to $\realLieAlgebra[h]$ and we write $\realTorusCt=\realTorusCtGp{T}\realTorusCtGp{A}$ and $\realTorusCtNci=\realTorusCtNciGp{T}\realTorusCtNciGp{A}$ for the corresponding Cartan subgroups of $\realGroup$. Roughly speaking, $\realTorusCt$ is `more compact' than $\realTorus$ and $\realTorusCtNci$ is `more split'. It is a fundamental fact that we obtain any $\cinv$-stable Cartan subgroup of $\realGroup$ (up to $\maxRealCompact$-conjugacy) by starting with $\realTorus$ and performing a sequence of Cayley transforms with respect to real or noncompact imaginary roots \cite{KGr}. 

Suppose $\complexLieAlgebra[h] \subset \complexLieAlgebra$ is a $\cinv$-stable Cartan subalgebra and choose a conjugation map 
\[
\absConj[] : \absLieAlgebra \to \complexLieAlgebra[h].
\]
It is easy to describe how the corresponding abstract pair $\absPair$ (Definition \ref{defAbsPair}) is affected by Cayley transform.

\begin{proposition}[\cite{KGr}, Proposition 6.72]
\label{propAbsPairCt}
Suppose $\simpleRoot \in \rootSystem(\complexLieAlgebra, \complexLieAlgebra[h])$ is a noncompact imaginary root and choose a conjugation map
\[
j : \absLieAlgebra \to \ctOpNci(\complexLieAlgebra[h]).
\]
Then the abstract pair corresponding to $\ctOpNci(\complexLieAlgebra[h])$ and $j$ is conjugate to $(\rootReflection \inv, \eta)$ where $\eta : \imaginaryRoots[\rootReflection\inv](\complexLieAlgebra, \absLieAlgebra) \to \zTwo$ and
\[
\eta(\simpleRootb) = \left\{\begin{array}{cl}
\grading(\simpleRootb)+1 & \simpleRoot + \simpleRootb \text{ is a root} \\
\grading(\simpleRootb) & \simpleRoot + \simpleRootb \text{ is not a root} 
\end{array}\right\}.
\]
\end{proposition}

In general, imaginary gradings are determined by the action of $\cinv$ on $\complexLieAlgebra$ (not just $\complexLieAlgebra[h]$) and thus depend on the particular real form $\realGroup = \realSpinGroup{p}{q}$. In order to make this relationship precise, we need the following definition.

\begin{definition}
\label{defNumCptRoots}
Let $\inv \in \weylGroup(\complexLieAlgebra, \absLieAlgebra)$ be an involution and suppose $\grading$ is an abstract grading for $\imaginaryRoots(\complexLieAlgebra, \absLieAlgebra)$. Set
\begin{eqnarray*}
\numCptBits(\grading) & = & \left|\left\{\simpleRoot \in \posRootSystem(\complexLieAlgebra, \absLieAlgebra) \mid \simpleRoot \text{ short and compact}\right\}\right| \\
\numNcptBits(\grading) & = & \left|\left\{\simpleRoot \in \posRootSystem(\complexLieAlgebra, \absLieAlgebra) \mid \simpleRoot \text{ short and noncompact}\right\}\right|
\end{eqnarray*}
and observe $\numImaginaryBits = \numCptBits(\grading) + \numNcptBits(\grading)$. As in Definition \ref{defInvParams} we also define
\[
\NcptBitsBit(\grading) = \left\{\begin{array}{rl}
0 & \numNcptBits(\grading) = 0 \\
1 & \numNcptBits(\grading) \ne 0
\end{array} \right.. \\
\]
\end{definition}

\begin{proposition}
\label{propRealForm}
Recall $\realGroup = \realSpinGroup{p}{q}$ with $p > q$, $p+q=2n+1$, and suppose $\realTorus \subset \realGroup$ is a $\cinv$-stable Cartan subgroup. If $\absPair$ is an abstract pair representing $\realTorus$ then
\begin{eqnarray*}
p & = & 2\cdot\text{max}(\numNcptBits(\grading),\numCptBits(\grading)) + \numRealBits + \numComplexBits + \left\{\begin{array}{rl} 1 & \numNcptBits(\grading) \le \numCptBits(\grading) \\ 0 & \text{else}\end{array}\right. \\
q & = & 2\cdot\text{min}(\numNcptBits(\grading),\numCptBits(\grading)) + \numRealBits + \numComplexBits + \left\{\begin{array}{rl} 1 & \numNcptBits(\grading) > \numCptBits(\grading) \\ 0 & \text{else}\end{array}\right..
\end{eqnarray*}
\end{proposition}

\begin{remark}
The formulas for $p$ and $q$ are independent of the choice of pair $\absPair$ representing $\realTorus$.
\end{remark}

\begin{proof}
Up to conjugation, the diagram for the abstract pair of a maximally split Cartan subgroup of $\realGroup = \realSpinGroup{p}{q}$ is
\[
\underset{1}{-} \underset{2}{-} \cdots \underset{q}{-} \underset{q+1}{+} + \cdots \underset{n}{+}.
\]
In particular, the real form is determined by the number of `$-$' signs appearing. Applying the above formulas to this case we have
\begin{eqnarray*}
p & = & 2(n-q) + q + 0 + 1 = 2n-q+1 = p\\
q & = & 2(0) + q + 0 + 0 = q
\end{eqnarray*}
as desired. If $\realTorus$ is not maximally split, it is easy to verify the value of $q$ is unaffected by performing Cayley transforms in noncompact imaginary roots (Proposition \ref{propAbsPairCt}). The formula for $p$ follows from $p + q = 2n+1$.
\end{proof}

\subsection{The Group $\mGroupTorus$}
\label{ssMGroupTorus}

To begin, suppose $\realTorus=\realGroup[T]\realGroup[A]$ is a $\cinv$-stable Cartan subgroup of $\realGroup$ with Lie algebra $\realLieAlgebra[h]$. Recall $\realRoots[\cinv](\complexLieAlgebra, \complexLieAlgebra[h])$ and $\imaginaryRoots[\cinv](\complexLieAlgebra, \complexLieAlgebra[h])$ denote the roots in $\rootSystem(\complexLieAlgebra, \complexLieAlgebra[h])$ that are real and imaginary with respect to $\cinv$.

\begin{definition}
\label{defZa}
Suppose $\simpleRoot \in \realRoots[\cinv](\complexLieAlgebra, \complexLieAlgebra[h])$ and choose $\rootVector \in \rootSpace$ according to Lemma \ref{lemXaSign}. Let $\zLa  =  \rootVector + \cinv\rootVector$ and define the elements
\begin{eqnarray*}
\oa & = & \expg(\frac{\pi}{2}\zLa) \in \maxRealCompact \\
\oac & = & \expgc(\frac{\pi}{2}\zLa) \in \maxRealCompactCover \\
\ma & = & \expg(\pi\zLa) = \exp(\pi i\coroot) \in \realTorus \\
\mac & = & \expgc(\pi\zLa) \in \realTorusCover.
\end{eqnarray*}
Then $\oa$, $\oac$ are representatives of the root reflection $\rootReflection[\simpleRoot]$ and we have $\oa^{2} = \ma$ and $\oac^{2} = \mac$ (\cite{KGr}, Proposition 6.52(c)).
\end{definition}

\begin{remark}
\label{remMaSign}
The conditions in Lemma \ref{lemXaSign} determine the vector $\rootVector$ (and thus $\zLa$) \emph{only up to sign}. This ambiguity potentially affects the elements of Definition \ref{defZa}. We have
\begin{eqnarray*}
\expg(2\pi\zLa) & = & \ma^{2} = 1 \\
\expg(\pi\zLa) & = & \expg(-\pi\zLa)
\end{eqnarray*}
so the element $\ma$ is determined by the root $\simpleRoot$. However 
\begin{eqnarray*}
\expg(-\frac{\pi}{2}\zLa) & = & \oa^{-1} \\
& = & \oa\ma
\end{eqnarray*}
and conjugation in $\oa$ and $\oa^{-1}$ induces the same action on $\realTorus$ or $\complexLieAlgebra[h]$, but not on $\realGroup$ or $\complexLieAlgebra$. In $\realGroupCover$, if $\simpleRoot$ is short then everything is the same as for $\realGroup$. If $\simpleRoot$ is long
\begin{eqnarray*}
\expgc(2\pi\zLa) & = & \mac^{2} = -1 \\
\expgc(\pi\zLa) & = & -\expgc(-\pi\zLa)
\end{eqnarray*}
and the element $\mac$ is determined by the root $\simpleRoot$ \emph{only up to inverse}. Moreover
\begin{eqnarray*}
\expgc(-\frac{\pi}{2}\zLa) & = & \oac^{-1} \\
& = & -\oac\mac
\end{eqnarray*}
so conjugation in $\oac$ and $\oac^{-1}$ induces the same action on $\realTorusCover$ if and only if $\mac$ is central in $\realTorusCover$.
\end{remark}

In particular, some care is required when discussing the elements $\oa$, $\oac$, and $\mac$ since these elements (and their corresponding actions) are not always determined by the root $\simpleRoot$. Depending on the context we may need to explicitly choose root vectors satisfying the condition in Lemma \ref{lemXaSign}. One notable exception is when the root $\simpleRoot$ is short, in which case $\mac$ and the actions of $\oa, \oac$ are well-defined (Proposition \ref{propMShort}).

\begin{lemma}
\label{lemCaCbMbCommute}
Let $\realLieAlgebra[h] \subset \realLieAlgebra$ be a $\cinv$-stable Cartan subalgebra and suppose $\simpleRoota, \simpleRootb \in \realRoots[\cinv](\complexLieAlgebra, \complexLieAlgebra[h])$ are strongly orthogonal. Then the operators $\left\{\ctOp[{\simpleRoota}], \ctOp[{\simpleRootb}]\right\}$ and $\left\{\ctOp[{\simpleRoota}], \Ad{\ma[{\simpleRootb}]}\right\}$ commute (Definition \ref{defCT}).
\end{lemma}

\begin{proof}
Since these operators are given by the adjoint representation, we need only check that $\left[\zLa[{\simpleRoota}],\zLa[{\simpleRootb}]\right] = 0$. However this follows easily from the fact the roots $\simpleRoota$ and $\simpleRootb$ are strongly orthogonal.
\end{proof}

\begin{definition}
\label{defMGroup}
Let $\realTorus \subset \realGroup$ be a Cartan subgroup. Denote by $\mGroupTorus$ the subgroup of $\realTorus$ generated by the $\ma$, i.e., 
\[
\mGroupTorus = \left<~\ma \mid \simpleRoot \in \realRoots[\cinv](\complexLieAlgebra, \complexLieAlgebra[h])~\right>.
\]
\end{definition}

We will record with some care the structure of $\mGroupTorus$ for arbitrary Cartan subgroups of $\realGroup$. For the following proposition, recall $\realGroup = \realSpinGroup{p}{q}$ and $Z(\realGroup) \cong \zTwo$.

\begin{proposition}[\cite{ABVPT}, Lemma 5.1]
\label{propMShort}
If $\simpleRoota \in \realRoots[\cinv](\complexLieAlgebra[g],\complexLieAlgebra[h])$ is \emph{short}, then $\ma[\simpleRoota]$ is equal to the nontrivial element of $Z(\realGroup)$.
\end{proposition}

In particular $\ma[\simpleRoota] = \ma[\simpleRootb]$ whenever $\simpleRoota$ and $\simpleRootb$ are short real roots. The following proposition effectively determines the remaining structure of $\mGroupTorus$.

\begin{proposition}
\label{propMStruct}
Let $\simpleRoota$, $\simpleRootb$, and $\simpleRootc$ be real roots in $\realRoots[\cinv](\complexLieAlgebra, \complexLieAlgebra[h])$ and suppose
\begin{eqnarray*}
\rootCheck[\simpleRootc] & = & \rootCheck[\simpleRoota] + \rootCheck[\simpleRootb].
\end{eqnarray*}
Then
\begin{eqnarray*}
\ma[\simpleRootc] & = & \ma[\simpleRoota]\ma[\simpleRootb].
\end{eqnarray*}
\end{proposition}

\begin{proof}
This follows immediately from Definition \ref{defZa}. See also \cite{VGr}, Corollary 4.3.20.
\end{proof}

Proposition \ref{propMStruct} implies
\[
\mGroupTorus = \left<~\ma \mid \simpleRoot \in \simpleRoots \subset \realRoots[\cinv](\complexLieAlgebra, \complexLieAlgebra[h])~\right>
\]
where $\simpleRoots$ is a choice of simple roots for $\realRoots[\cinv](\complexLieAlgebra, \complexLieAlgebra[h])$. We now turn our attention to noncompact imaginary roots.

\begin{definition}
\label{defZb}
Suppose $\simpleRootb \in \imaginaryRoots[\cinv](\complexLieAlgebra, \complexLieAlgebra[h])$ is imaginary and noncompact and choose a root vector $\rootVector[\simpleRootb]$ as in Lemma \ref{lemXbSign}. Let $\zLa[\simpleRootb]  =  \rootVector[\simpleRootb] + \overline{\rootVector[\simpleRootb]}$ and define
\begin{eqnarray*}
\ma[\simpleRootb] & = & \exp(\pi i\zLa[\simpleRootb]) = \expg(\pi i\coroot[\simpleRootb])  \in \realTorus \\
\mac[\simpleRootb] & = & \expgc(-\pi i\coroot[\simpleRootb]) \in \realTorusCover.
\end{eqnarray*}
\end{definition}

\begin{remark}[\cite{TA}]
\label{remXbSign}
As before, Lemma \ref{lemXbSign} defines the vector $\rootVector[\simpleRootb]$ only up to sign. However
\begin{eqnarray*}
\expg(2\pi i\zLa[\simpleRootb]) & = & \ma[\simpleRootb]^{2} = 1 \\
\expg(\pi i\zLa[\simpleRootb]) & = & \expg(-\pi i\zLa[\simpleRootb])
\end{eqnarray*}
and the element $\ma[\simpleRootb]$ is determined by the root $\simpleRootb$. The same is true in the group $\realGroupCover$ if $\simpleRootb$ is short. However, if $\simpleRootb$ is long we have
\begin{eqnarray*}
\expgc(-2\pi i\coroot[\simpleRootb]) & = & \mac[\simpleRootb]^{2} = -1 \\
\expgc(-\pi i\coroot[\simpleRootb]) & = & -\expgc(\pi i\coroot[\simpleRootb]).
\end{eqnarray*}
\end{remark}

\begin{proposition}[\cite{KGr}, 6.68]
\label{propCtAd}
If $\simpleRoot \in \realRoots[\cinv](\complexLieAlgebra, \complexLieAlgebra[h])$, we have
\begin{eqnarray*}
\ctOp(\coroot) & = & i\zLa \\
\ctOp(\zLa) & = & i\coroot
\end{eqnarray*}
(Definition \ref{defZa}).
\end{proposition}

\begin{proposition}[\cite{KGr}, 6.66]
If $\simpleRootb \in \imaginaryRoots[\cinv](\complexLieAlgebra, \complexLieAlgebra[h])$ is imaginary and noncompact, we have
\begin{eqnarray*}
\ctOpNci[\simpleRootb](\coroot[\simpleRootb]) & = & \zLa[\simpleRootb] \\
\ctOpNci[\simpleRootb](\zLa[\simpleRootb]) & = & -\coroot[\simpleRootb]
\end{eqnarray*}
(Definition \ref{defZb}).
\end{proposition}

Write
\begin{eqnarray*}
\ctOp & : & \rootSystem(\complexLieAlgebra, \complexLieAlgebra[h]) \to \rootSystem(\complexLieAlgebra, \complexLieAlgebraCt) \\
\ctOpNci[\simpleRootb] & : & \rootSystem(\complexLieAlgebra, \complexLieAlgebra[h]) \to \rootSystem(\complexLieAlgebra, \complexLieAlgebraCtNci[\simpleRootb])
\end{eqnarray*}
for the induced Cayley transform operations on roots. 

\begin{proposition}[\cite{KGr}, Proposition 6.69]
\label{propCtInv}
The root $\ctOpNci[\simpleRootb](\simpleRootb) \in \rootSystem(\complexLieAlgebra, \complexLieAlgebraCtNci[\simpleRootb])$ is real and the root $\ctOp(\simpleRoot) \in \rootSystem(\complexLieAlgebra, \complexLieAlgebraCt)$ is imaginary and noncompact. Moreover, there exists a suitable choice of root vectors so that
\[
\ctOpNci[{\ctOp(\simpleRoot)}] \circ \ctOp = \ctOp[{\ctOpNci[\simpleRootb](\simpleRootb)}] \circ \ctOpNci[\simpleRootb] = \identity.
\]
\end{proposition}

In particular, the Cayley transform operators are essentially inverses.

\begin{corollary}
Suppose $\simpleRoot \in \realRoots[\cinv](\complexLieAlgebra, \complexLieAlgebra[h])$ and $\simpleRootb \in \imaginaryRoots[\cinv](\complexLieAlgebra, \complexLieAlgebra[h])$ are short roots with $\simpleRootb$ noncompact. Then $\ma = \ma[\simpleRootb]$.
\end{corollary}

\begin{lemma}
\label{lemCtCoroot}
Let $\realLieAlgebra[h] \subset \realLieAlgebra$ be a $\cinv$-stable Cartan subalgebra and suppose $\simpleRoota, \simpleRootb \in \rootSystem(\complexLieAlgebra, \complexLieAlgebra[h])$ are orthogonal. If $\simpleRoota$ is imaginary and noncompact, then 
\[
\ctOpNci(\coroot[\simpleRootb]) = \coroot[\simpleRootb]
\]
where $\coroot[\simpleRootb]$ is the coroot of $\simpleRootb$.
\end{lemma}

\begin{proof}
This follows easily from Definition \ref{defCT} and the fact
\[
\left[\overline\rootVector - \rootVector, \coroot[\simpleRootb]\right] = 0.
\]
\end{proof}

\begin{proposition} 
\label{propSameMa}
Let $\simpleRoot \in \realRoots[\cinv](\complexLieAlgebra, \complexLieAlgebra[h])$ be a real root and choose $\rootVector \in \rootSpace$ according to Lemma \ref{lemXaSign}. Then $\mac = \mac[{\ctOp(\simpleRoot)}]$.
\end{proposition}

\begin{proof}
The coroot for $\ctOp(\simpleRoot)$ is $\ctOp(\coroot) = i\zLa \in \complexLieAlgebraCt$ by Proposition \ref{propCtAd}. Therefore we have
\begin{eqnarray*}
\mac[{\ctOp(\simpleRoot)}] & = & \expgc(-\pi i(i\zLa)) \\
& = & \expgc(\pi\zLa) \\
& = & \mac
\end{eqnarray*}
as desired.
\end{proof}

For completeness, we now describe the situation for compact imaginary roots.

\begin{definition}
\label{defZc}
Suppose $\simpleRootc \in \imaginaryRoots[\cinv](\complexLieAlgebra, \complexLieAlgebra[h])$ is imaginary and compact and choose a root vector $\rootVector[\simpleRootc]$ as in Lemma \ref{lemXcSign}. Let $\zLa[\simpleRootc]  =  \rootVector[\simpleRootc] + \overline{\rootVector[\simpleRootc]}$ and define
\begin{eqnarray*}
\oa[\simpleRootc] & = & \expg(\frac{\pi}{2}\zLa[\simpleRootc]) \in \maxRealCompact \\
\oac[\simpleRootc] & = & \expgc(\frac{\pi}{2}\zLa[\simpleRootc]) \in \maxRealCompactCover \\
\ma[\simpleRootc] & = & \expg(\pi\zLa[\simpleRootc]) = \expg(\pi i\coroot[\simpleRootc]) \in \realTorus \\
\mac[\simpleRootc] & = & \expgc(\pi\zLa[\simpleRootc]) \in \realTorusCover.
\end{eqnarray*}
\end{definition}

\begin{remark}[\cite{TA}]
As before, Lemma \ref{lemXcSign} defines the vector $\rootVector[\simpleRootc]$ only up to sign. However
\begin{eqnarray*}
\expg(-\frac{\pi}{2}\zLa[\simpleRootc]) & = & \oa[\simpleRootc]^{-1} = \oa[\simpleRootc]\ma[\simpleRootc] \\
\expg(2\pi i\zLa[\simpleRootc]) & = & \ma[\simpleRootc]^{2} = 1 \\
\expgc(2\pi\zLa[\simpleRootc]) & = & \mac[\simpleRootc]^{2} = 1
\end{eqnarray*}
and the elements $\ma[\simpleRootc]$, $\mac[\simpleRootc]$ are both determined by the root $\simpleRootc$ (compare with Remark \ref{remXbSign}).
\end{remark}

\subsection{Structure of $\pi_{0}(\realTorus)$}
\label{ssAlgTorus}

In this section we determine the connected components of the $\cinv$-stable Cartan subgroups of $\realGroup = \realSpinGroup{p}{q}$. We begin with the following fundamental proposition.

\begin{proposition}[\cite{BC1}]
\label{propCartanStructure}
If $\realTorus \subset \realGroup$ is a $\cinv$-stable Cartan subgroup, then
\begin{eqnarray*}
\realTorus & \cong & \sOnen[a] \times \rCrossn[b] \times \cCrossn[c]
\end{eqnarray*}
with $a + b + 2c = \frac{p+q-1}{2} = n$. 
\end{proposition}

In particular, $\pi_{0}(\realTorus) \cong \zTwo^{b}$ and it remains to determine $b$. To begin let $\realTorus \subset \realGroup$ be a $\cinv$-stable Cartan subgroup, choose a conjugation map
\[
\absConj[] : \absLieAlgebra \to \complexLieAlgebra[h],
\]
and suppose $\absPair$ denotes the corresponding abstract pair (Definition \ref{defAbsPair}). It is a remarkable fact that the numbers $a,b,c$ are determined entirely by (the conjugacy class of) $\inv$ \cite{BC1}. In particular, the structure of $\realTorus$ is independent of the grading $\grading$ and thus the particular real form $\realGroup$. By Remark \ref{remInvCartanMap} it suffices to consider only the split real form $\realGroupSplit = \realSpinGroup{n+1}{n}$ and for the remainder of this section we assume $\realGroup = \realGroupSplit$.

Suppose $\realTorusSplit \subset \realGroupSplit$ is a split Cartan subgroup and choose a simple system $\simpleRoots \subset \realRoots[-\identity](\complexLieAlgebra, \splitLieAlgebra) = \rootSystem(\complexLieAlgebra, \splitLieAlgebra)$. Since $\realTorusSplit$ is split, we have $\realTorusSplit \cong \rCrossn[n]$ and thus $\pi_{0}(\realTorusSplit) \cong \zTwo^{n}$. The following proposition gives explicit generators for the connected components of $\realTorusSplit$.

\begin{proposition}[\cite{AH}]
\label{propMCompGroup}
The elements $\left\{\ma \mid \simpleRoot \in \simpleRoots\right\}$ (Definition \ref{defZa}) live in distinct connected components of $\realTorusSplit$ and 
\[
\mGroupTorus[\realTorusSplit] = \left<~\ma \mid \simpleRoot \in \simpleRoots~\right>  \cong \pi_{0}(\realTorusSplit).
\]
In particular, the simple $\ma$ are a basis for the component group of $\realTorusSplit$.
\end{proposition}

More generally, the following proposition justifies our interest in $\mGroupTorus$.

\begin{proposition}
\label{propMCompGroupCT}
Let $\realTorus \subset \realGroup$ be a $\cinv$-stable Cartan subgroup. Then
\[
\realTorus = \mGroupTorus\realTorusId.
\]
In particular, the simple $\ma$ generate the component group of $\realTorus$.
\end{proposition}

\begin{proof}
This is discussed in greater generality in \cite{AH}. The specialization to the above form is justified by the fact that $\realGroup$ is connected.
\end{proof}

To compute the component groups of the nonsplit Cartan subgroups, we build each one (up to conjugacy) from $\realTorusSplit$ via an iterated sequence of Cayley transforms (Definition \ref{defCT}). Along the way, we track the $\ma$ from Proposition \ref{propMCompGroup} that remain in distinct connected components. The following proposition and corollary describes how to do this.

\begin{proposition}
\label{propMaRootSpace}
Let $\realTorus \subset \realGroup$ be a $\cinv$-stable Cartan subgroup and suppose $\simpleRoota, \simpleRootb \in \rootSystem(\complexLieAlgebra, \complexLieAlgebra[h])$ with $\simpleRootb$ real or imaginary. If $\rootVector[\simpleRoota]$ is any nonzero root vector for $\simpleRoota$ we have
\[
\Ad{\ma[\simpleRootb]}(\rootVector[\simpleRoota]) = (-1)^{(\simpleRoota, \rootCheck[\simpleRootb])}\rootVector[\simpleRoota].
\]
\end{proposition}

\begin{proof}
From the definition of $\ma[\simpleRootb]$
\begin{eqnarray*}
\Ad{\ma[\simpleRootb]}(\rootVector[\simpleRoota]) & = & \Ad{\expg(\pi i\coroot[\simpleRootb])}(\rootVector[\simpleRoota]) \\
& = & e^{\ad{\pi i\coroot[\simpleRootb]}}(\rootVector[\simpleRoota]) \\
& = & \sum_{k=0}^{\infty}\frac{(\pi i)^{k}}{k!}\ad{\coroot[\simpleRootb]}^{k}(\rootVector[\simpleRoota]) \\
& = & \sum_{k=0}^{\infty}\frac{(\pi i)^{k}}{k!}(\simpleRoota, \rootCheck[\simpleRootb])^{k}\rootVector[\simpleRoota] \\
& = & e^{\pi i(\simpleRoota, \rootCheck[\simpleRootb])}\rootVector[\simpleRoota] \\
& = & (-1)^{(\simpleRoota, \rootCheck[\simpleRootb])}\rootVector[\simpleRoota]
\end{eqnarray*}
as desired.
\end{proof}

\begin{corollary} \label{corMaRootSpace}
In the setting of Proposition \ref{propMaRootSpace}, suppose $\simpleRoota$ and $\simpleRootb$ are distinct (positive) roots in $\realRoots[\cinv](\complexLieAlgebra, \complexLieAlgebra[h])$ with $\simpleRootb$ long. Then $\ma[\simpleRootb] \in \realTorusCt$ if and only if $\simpleRoota$ and $\simpleRootb$ are orthogonal. If $\simpleRootb$ is short, then $\ma[\simpleRootb] \in \realTorusCt$.
\end{corollary}

\begin{proof}
Recall
\[
\realLieAlgebraCt =  \text{ker}(\simpleRoot|_{\realLieAlgebra[h]}) \oplus \mathbb{R}(\rootVector +\cinv\rootVector)
\]
and $\realTorusCt$ is the corresponding Cartan subgroup. Since $\ma[\simpleRootb] \in \realTorus$, $\ma[\simpleRootb]$ centralizes $\text{ker}(\simpleRoot|_{\realLieAlgebra[h]})$ by definition. Therefore $\ma[\simpleRootb]$ will be an element of $\realTorusCt$ if and only if $\Ad{\ma[\simpleRootb]}$ centralizes $\rootVector$. If $\simpleRootb$ is long, this will happen if and only if
\[
(\simpleRoota, \rootCheck[\simpleRootb]) = (\simpleRoota, \simpleRootb) \in 2\mathbb{Z}.
\]
by Proposition \ref{propMaRootSpace}. Since $\simpleRoota \ne \pm\simpleRootb$ we have $(\simpleRoota, \simpleRootb) \in \left\{0,1,-1\right\}$ and the first result follows. If $\simpleRootb$ is short the result follows from Proposition \ref{propMShort}.
\end{proof}

\begin{definition}
\label{defMGroupCt}
Suppose $\realTorus \subset \realGroup$ is a $\cinv$-stable Cartan subgroup and $\simpleRoota \in \realRoots[\cinv](\complexLieAlgebra, \complexLieAlgebra[h])$. Let
\[
\mGroupTorusCt = \mGroupTorus \cap \realTorusCt
\]
denote the subgroup of elements in $\mGroupTorus$ that are also in $\realTorusCt$. 
\end{definition}

\begin{proposition} 
\label{propMgroupCtContain}
In the setting of Definition \ref{defMGroupCt} we have
\[
\mGroupTorus[\realTorusCt] \subseteq \mGroupTorusCt.
\]
\end{proposition}

\begin{proof}
$\mGroupTorus[\realTorusCt]$ is contained in $\realTorusCt$ by definition. Moreover we have
\begin{eqnarray*}
\mGroupTorus[\realTorusCt] & = & \left<~\ma[\simpleRootb] \in \realTorusCt \mid \simpleRootb \in \realRoots[\cinv](\complexLieAlgebra, \complexLieAlgebraCt) ~\right> \\
& = & \left<~\ma[\simpleRootb] \in \realTorus \mid \simpleRootb \in \realRoots[\cinv](\complexLieAlgebra, \complexLieAlgebra[h]) \text{ and } (\simpleRoota, \simpleRootb) = 0 ~\right> \\
& \subseteq & \mGroupTorus
\end{eqnarray*}
as desired.
\end{proof}

Let $\realTorus = \realGroup[T]\realGroup[A]$ be a $\cinv$-stable Cartan subgroup of $\realGroup$ and suppose $\simpleRoota \in \realRoots[\cinv](\complexLieAlgebra, \complexLieAlgebra[h])$. Choose a nonzero root vector $\rootVector[\simpleRoota] \in \rootSpace$ and set
\[
\sOneCt = \text{exp}_{\realGroup}\left\{\mathbb{R}(\rootVector +\cinv\rootVector)\right\}.
\]
In particular, $\sOneCt \subset \realTorusCtGp{T}$ is a connected compact abelian subgroup of $\realTorusCt$.

\begin{proposition}[\cite{VGr}, Lemma 8.3.13] \label{propMaIdCompH}
In the above setting,
\[
\realGroup[T] \cap \sOneCt = \left\{1,\ma\right\}.
\]
In particular $\ma \in \realTorusCtId$, the identity component of $\realTorusCt$.
\end{proposition}

\begin{corollary} \label {corMCtStruct}
Let $\realTorus \subset \realGroup$ be a $\cinv$-stable Cartan subgroup and suppose $\simpleRoota \in \realRoots[\cinv](\complexLieAlgebra, \complexLieAlgebra[h])$. Then
\[
\mGroupTorusCt = \left<~ \mGroupTorus[\realTorusCt],~\ma~\right>.
\]
\end{corollary}

\begin{proof}
Propositions \ref{propMgroupCtContain} and \ref{propMaIdCompH} imply
\[
\mGroupTorusCt \supseteq \left<~ \mGroupTorus[\realTorusCt],~\ma~\right>.
\]

Conversely, Proposition \ref{propInvWeylGroups} implies the real roots of $\realTorus$ form a root system of type $B_{m} \times A_{1}^{l}$. Choose a system of positive roots in $\realRoots[\cinv](\complexLieAlgebra[g],\complexLieAlgebra[h])$ for which $\simpleRoot$ is simple. If $\simpleRoot$ is contained in the $A_{1}^{l}$ factor, then the other simple roots are orthogonal to $\simpleRoot$ and the result easily follows as in Proposition \ref{propMgroupCtContain}. Hence we may assume $\realRoots[\cinv](\complexLieAlgebra[g],\complexLieAlgebra[h])$ is of type $B_{m}$. 

For $m \in \mGroupTorus$, write $m = \ma[{\simpleRooti[1]}] \cdots \ma[{\simpleRooti[k]}]$ where the $\simpleRooti$ are simple and suppose $m \in \mGroupTorusCt$. Then $m$ centralizes $\rootVector$ (the root vector for $\simpleRoot$) and we need to show $m \in \left<~ \mGroupTorus[\realTorusCt],~\ma~\right>$. If $(\simpleRoot, \simpleRooti) \in \left\{0,2\right\}$ for all $i$, then each $\ma[{\simpleRooti}] \in \left<~ \mGroupTorus[\realTorusCt],~\ma~\right>$ and we are done. Otherwise $\simpleRoot$ is long and there are exactly two terms (say $\ma[{\simpleRooti}]$ and $\ma[{\simpleRooti[j]}]$) in the expression for $m$ with $(\simpleRoot,\simpleRooti) = (\simpleRoot,\simpleRooti[j]) = -1$. Define a new root
\[
\simpleRootc = \simpleRooti + \simpleRoot + \simpleRooti[j]
\]
and observe $(\simpleRoot, \simpleRootc) = 0$ so that $\ma[\simpleRootc] \in \mGroupTorus[\realTorusCt]$. But $\ma[{\simpleRootc}]\ma[{\simpleRoot}] = \ma[{\simpleRooti}]\ma[{\simpleRooti[j]}]$ by Proposition \ref{propMStruct} and thus $\ma[{\simpleRooti}]\ma[{\simpleRooti[j]}] \in \left<~ \mGroupTorus[\realTorusCt],~\ma~\right>$. It follows $m \in \left<~ \mGroupTorus[\realTorusCt],~\ma~\right>$ and we have shown
\[
\mGroupTorusCt \subseteq \left<~ \mGroupTorus[\realTorusCt],~\ma~\right>
\]
as desired.
\end{proof}

Informally, the corollary implies we can move the root $\simpleRoota$ inside the parentheses as long as we add the element $\ma[\simpleRoota]$ to the resulting group. Let $\realTorus$ be a $\cinv$-stable Cartan subgroup and suppose $\simpleRoota_{1},\simpleRoota_{2},\ldots,\simpleRoota_{k}$ are mutually orthogonal roots in $\realRoots[\cinv](\complexLieAlgebra, \complexLieAlgebra[h])$. Then the \emph{iterated} Cayley transform 
\[
\realTorusCt[\simpleRoota_{1}\simpleRoota_{2}\ldots\simpleRoota_{k}] = (\realTorusCtGp[\simpleRoota_{2}]{{(\realTorusCt[\simpleRoota_{1}])}})\ldots)_{\simpleRoota_{k}}
\]
of $\realTorus$ with respect to $\simpleRoota_{1},\simpleRoota_{2},\ldots,\simpleRoota_{k}$ is defined. Similarly we can define
\[
\mGroupTorusCt[\simpleRoota_{1}\simpleRoota_{2}\ldots\simpleRoota_{k}] = \mGroupTorus \cap \realTorusCt[\simpleRoota_{1}\simpleRoota_{2}\ldots\simpleRoota_{k}].
\]
If $\rSeq = \simpleRoota_{1}\simpleRoota_{2}\ldots\simpleRoota_{k}$ denotes a sequence of mutually orthogonal roots in $\realRoots[\cinv](\complexLieAlgebra,\complexLieAlgebra[h])$ we will sometimes write
\begin{eqnarray*}
\realTorusCt[\rSeq] & = & \realTorusCt[\simpleRoota_{1}\simpleRoota_{2}\ldots\simpleRoota_{k}] \\
\mGroupTorusCt[\rSeq] & = & \mGroupTorusCt[\simpleRoota_{1}\simpleRoota_{2}\ldots\simpleRoota_{k}].
\end{eqnarray*}
We conclude this section with the following extension of Corollary \ref{corMCtStruct}.

\begin{proposition} \label{propMIterCtStruct}
Suppose $\realTorus$ is a $\cinv$-stable Cartan subgroup of $\realGroup$ and $\simpleRoota_{1},\simpleRoota_{2},\ldots,\simpleRoota_{k}$ is a sequence of mutually orthogonal real roots in $\rootSystem(\complexLieAlgebra, \complexLieAlgebra[h])$. Then
\[
\mGroupTorusCt[\simpleRoota_{1}\simpleRoota_{2}\ldots\simpleRoota_{k}] = \left<~\mGroupTorus[{\realTorusCt[\simpleRoota_{1}\simpleRoota_{2}\ldots\simpleRoota_{k}]}],~\ma[\simpleRoota_{1}],~\ma[\simpleRoota_{2}],\ldots,\ma[\simpleRoota_{k}]~\right>.
\]
\end{proposition}

\begin{proof} The proof is by induction on the number of roots. The base case is Corollary \ref{corMCtStruct}. In general we have
\begin{eqnarray*}
\mGroupTorusCt[\simpleRoota_{1}\simpleRoota_{2}\ldots\simpleRoota_{k}] & = & \mGroupTorus \cap \realTorusCt[\simpleRoota_{1}\simpleRoota_{2}\ldots\simpleRoota_{k}] \\
& = & \mGroupTorus \cap \realTorusCt[\simpleRoota_{1}\simpleRoota_{2}\ldots\simpleRoota_{k-1}] \cap \realTorusCt[\simpleRoota_{1}\simpleRoota_{2}\ldots\simpleRoota_{k}] \\
& = & \left<~\mGroupTorus[{\realTorusCt[\simpleRoota_{1}\ldots\simpleRoota_{k-1}]}],~\ma[\simpleRoota_{1}],\ldots,\ma[\simpleRoota_{k-1}]~\right> \cap \realTorusCt[\simpleRoota_{1}\ldots\simpleRoota_{k}] \\
& = & \left<~\mGroupTorus[{\realTorusCt[\simpleRoota_{1}\ldots\simpleRoota_{k-1}]}]  \cap \realTorusCt[\simpleRoota_{1}\ldots\simpleRoota_{k}]  ,~\ma[\simpleRoota_{1}],\ldots,\ma[\simpleRoota_{k-1}]~\right> \\
& = & \left<~\mGroupTorusGpCt[\simpleRoota_{k}]{{\realTorusCt[\simpleRoota_{1}\simpleRoota_{2}\ldots\simpleRoota_{k-1}]}},~\ma[\simpleRoota_{1}],~\ma[\simpleRoota_{2}],\ldots,\ma[\simpleRoota_{k-1}]\right> \\
& = & \left<~\mGroupTorus[{\realTorusCt[\simpleRoota_{1}\simpleRoota_{2}\ldots\simpleRoota_{k}]}],~\ma[\simpleRoota_{1}],~\ma[\simpleRoota_{2}],\ldots,\ma[\simpleRoota_{k}]~\right>
\end{eqnarray*}
as desired.
\end{proof}

\begin{proposition} 
\label{propCompH} 
Suppose $\rSeq = \simpleRooti[1],\simpleRooti[2],\ldots,\simpleRooti[m]$ is any sequence of mutually orthogonal real roots in $\realTorusSplit$ and
\[
\mGroupTorusGpCt[\rSeq]{\realTorusSplit} = \left<~\mGroupTorus[{\realTorusSplitCt[\rSeq]}],~\ma[{\simpleRooti[1]}],\ldots,\ma[{\simpleRooti[m]}]~\right>
\]
as in Proposition \ref{propMIterCtStruct}. Then
\begin{eqnarray*}
\pi_{0}({\realTorusSplitCt[\rSeq]}) & \cong & \mGroupTorusGpCt[\rSeq]{\realTorusSplit}~/ \left<~\ma[{\simpleRooti[1]}],\ldots,\ma[{\simpleRooti[m]}]~\right> \\& \cong & \mGroupTorus[{\realTorusSplitCt[\rSeq]}]~/~ (\mGroupTorus[{\realTorusSplitCt[\rSeq]}] \cap \left<~\ma[{\simpleRooti[1]}],\ldots,\ma[{\simpleRooti[m]}]~\right>).
\end{eqnarray*}
\end{proposition}

\begin{proof}
The first isomorphism follows from Proposition \ref{propMCompGroup}, Proposition \ref{propMCompGroupCT}, and Proposition \ref{propMaIdCompH}. The second isomorphism follows from the Second Isomorphism Theorem for groups.
\end{proof}

Proposition \ref{propCompH} extends Proposition \ref{propMCompGroup} to the Cartan subgroups $\realTorusSplitCt[\rSeq]$. The extension is almost perfect in the sense that $\pi_{0}(\realTorusSplitCt[\rSeq])$ is generated by elements of $\mGroupTorus[{\realTorusSplitCt[\rSeq]}]$. Although distinct elements of $\mGroupTorus[{\realTorusSplitCt[\rSeq]}]$ can live in the same connected component of $\realTorusSplitCt[\rSeq]$, this relationship is completely determined by the sequence $\rSeq$. In particular, we obtain a model for $\pi_{0}(\realTorusSplitCt[\rSeq])$ in $\realTorusSplitCt[\rSeq]$ by choosing elements of $\mGroupTorus[{\realTorusSplitCt[\rSeq]}]$ that are distinct modulo those elements determined by $\rSeq$.

We should remark the form of Proposition \ref{propCompH} admits an interesting combinatorial description that is often useful in computations. Fix a split Cartan subgroup $\realTorusSplit \subset \realGroupSplit$. Proposition \ref{propMStruct} implies
\[
\mGroupTorus[\realTorusSplit] \cong \coRootLattice / 2\coRootLattice
\]
where $\coRootLattice$ is the coroot lattice in $\splitLieAlgebra$. Let $\quotWeightLatticeMap : \coRootLattice \to \mGroupTorus[\realTorusSplit]$ denote the corresponding quotient map and suppose $\rSeq = \simpleRootai[1],\simpleRootai[2],\ldots,\simpleRootai[m]$ is a sequence in $\realRoots[\cinv](\complexLieAlgebra,\splitLieAlgebra)$ of pairwise orthogonal roots. If $\invSplit = -\identity$, write
\[
\inv = \invSplitCt[\rSeq] = \rootReflection[{\simpleRootai[m]}]\ldots\rootReflection[{\simpleRootai[2]}]\rootReflection[{\simpleRootai[1]}]\invSplit
\]
and define
\begin{eqnarray*}
\posCorootLattice & = & \left\{\simpleRoot \in \coRootLattice \mid \inv(\simpleRoot) = \simpleRoot\right\} \\
\negCorootLattice & = & \left\{\simpleRoot \in \coRootLattice \mid \inv(\simpleRoot) = -\simpleRoot\right\} \\
\posQuotCorootLattice & = & \quotWeightLatticeMap(\posCorootLattice) \\
\negQuotCorootLattice & = & \quotWeightLatticeMap(\negCorootLattice) \\
\posnegQuotCorootLattice & = & \posQuotCorootLattice \cap \negQuotCorootLattice.
\end{eqnarray*}
Clearly then
\begin{eqnarray*}
\mGroupTorus[{\realTorusSplitCt[\rSeq]}] & \cong & \negQuotCorootLattice \\
\left<\ma[{\simpleRootai[1]}],\ma[{\simpleRootai[2]}],\ldots,\ma[{\simpleRootai[m]}]\right> & = &  \posQuotCorootLattice 
\end{eqnarray*}
and Proposition \ref{propCompH} implies
\[
\pi_{0}(\realTorusSplitCt[\rSeq]) \cong \negQuotCorootLattice / \negQuotCorootLattice \cap \posQuotCorootLattice \cong \negQuotCorootLattice / \posnegQuotCorootLattice. 
\]
In particular, the component groups of real tori in $\realGroup$ are computable in terms of involutions and lattices. A similar result is true in general and is the main topic discussed in \cite{BC1}.

\subsection{Structure of $\pi_{0}(\realTorusCover)$}
\label{AlgTorusCover}
Let $\realGroupCover = \realSpinGroupCover{p}{q}$ (Definition \ref{defGCover}) and suppose $\realTorusCover \subset \realGroupCover$ is a Cartan subgroup (Definition \ref{defCSGCover}). Although $\realTorusCover$ may not be abelian, we do have the following lemma.

\begin{lemma}
\label{lemAbelianCompGroup}
$\realTorusCoverId$ is central (and thus abelian) in $\realTorusCover$.
\end{lemma}

\begin{proof}
Let $g$ and $h$ be elements in $\realTorus$. Since $\realTorus$ is abelian, we have
\[
[g,h] = ghg^{-1}h^{-1} = 1.
\]
Therefore the commutator of any two elements in $\realTorusCover$ lands in $\left\{\pm1\right\}$ and thus elements of $\realTorusCoverId$ must have trivial commutator with $\realTorusCover$.
\end{proof}

As in the previous section, it suffices to consider only the split form $\realGroupCover[\realGroupSplit]$. To begin, fix a split Cartan subgroup $\realTorusSplit \subset \realGroupSplit$ and recall the group $\mGroupTorus[\realTorusSplit]$ of Definition \ref{defMGroup}. For each $\ma \in \mGroupTorus[\realTorusSplit]$, choose once and for all an inverse image $\mac$ under the projection map $\projOp$ and write 
\begin{eqnarray*}
\projOpInv(\ma) & = & \left\{\mac,-\mac\right\} \\
\projOpInv(\mGroupTorus[\realTorusSplit]) & = & \mGroupTorusCover[\realTorusSplitCover].
\end{eqnarray*}
Fix a set of simple roots $\simpleRoots \subset \rootSystem(\complexLieAlgebra, \splitLieAlgebra)$ and the usual ordering $\simpleRooti[1] < \simpleRooti[2] < \ldots < \simpleRooti[n]$ of elements in $\simpleRoots$. The (nonabelian) structure of $\mGroupTorusCover[\realTorusSplitCover]$ is given by the following proposition and corollary.

\begin{proposition}[\cite{ABVPT}, Lemma 4.8] 
\label{propMCStruct}
The group $\mGroupTorusCover[\realTorusSplitCover]$ is generated by the elements $\left\{\mac[{\simpleRoot}] \mid \simpleRoot \in \simpleRoots\right\}$
subject to the following relations
\begin{eqnarray*}
\mac[\simpleRoot]^{2} & = & \left\{\begin{array}{rl}
-1 & \simpleRoot \text{ is long } \\
1 & \simpleRoot \text{ is short } \\
\end{array} \right. \\
\left[\mac[\simpleRoota],\mac[\simpleRootb]\right] & = & \left\{\begin{array}{ll}
(-1)^{(\simpleRoota, \rootCheck[\simpleRootb])} & \simpleRoota,\simpleRootb \text{ are \emph{both} long} \\
1 & \text{otherwise} \\
\end{array} \right.. \\
\end{eqnarray*}
\end{proposition}

\begin{corollary} \label{corMCStruct}
Every element of $\mGroupTorusCover[\realTorusSplitCover]$ has a unique expression of the form
\[
\pm \mac[{\simpleRooti[1]}]^{\epsilon_{1}}\mac[{\simpleRooti[2]}]^{\epsilon_{2}} \ldots \mac[{\simpleRooti[n]}]^{\epsilon_{n}}
\]
where $\epsilon_{i} \in \left\{0,1\right\}$.
\end{corollary}

The following result is the expected analog of Proposition \ref{propMCompGroup} for $\realTorusSplitCover$.

\begin{proposition} \label{propMCCompGroup}
The elements $\left\{1,-1\right\}$  live in distinct connected components of $\realTorusSplitCover$. Moreover we have
\[
\mGroupTorusCover[\realTorusSplitCover] = \left<~\mac \mid \simpleRoot \in \simpleRoots~\right>  \cong \pi_{0}(\realTorusSplitCover).
\]
\end{proposition}

We now proceed as in the previous section to determine the structure of $\pi_{0}(\realTorusCover)$. Let $\rSeq = \simpleRootdi[1],\simpleRootdi[2],\ldots,\simpleRootdi[m]$ be a sequence of mutually orthogonal roots in $\rootSystem(\complexLieAlgebra, \splitLieAlgebra) = \realRoots[\cinv](\complexLieAlgebra, \splitLieAlgebra)$. Recall Proposition \ref{propCompH} implies
\[
\pi_{0}({\realTorusSplitCt[\rSeq]}) \cong \mGroupTorus[{\realTorusSplitCt[\rSeq]}]~/~ (\mGroupTorus[{\realTorusSplitCt[\rSeq]}] \cap \left<~\ma[{\simpleRootdi[1]}],\ldots,\ma[{\simpleRootdi[m]}]~\right>).
\]
In words, representatives for the connected components of $\realTorusSplitCt[\rSeq]$ are given by choosing elements of $\mGroupTorus[{\realTorusSplitCt[\rSeq ]}]$ that are distinct `modulo  $\rSeq$'. The following proposition (whose proof is easy) extends this to $\realTorusSplitCtCover[\rSeq]$.

\begin{proposition} \label {propNumCompHCtCover}
In the situation above
\[
\left|\pi_{0}(\realTorusSplitCtCover[\rSeq])\right| = \left\{
\begin{array}{rl}
2\left|\pi_{0}(\realTorusSplitCt[\rSeq])\right|, & -1 \notin \realTorusSplitCtCoverId[\rSeq] \\
\left|\pi_{0}(\realTorusSplitCt[\rSeq])\right|, & -1 \in \realTorusSplitCtCoverId[\rSeq] \\
\end{array}\right..
\]
Moreover, if $\ma[\simpleRootd] \in \mGroupTorus[{\realTorusSplitCt[\rSeq]}]$ then $\mac[\simpleRootd]$ and $-\mac[\simpleRootd]$ are contained in distinct connected components of $\realTorusSplitCtCover[\rSeq]$ if and only if $-1 \notin \realTorusSplitCtCoverId[\rSeq]$.
\end{proposition}

Let $\left\{\ma[{\simpleRootei[1]}],\ldots,\ma[{\simpleRootei[k]}]\right\}$ be a set of representatives for the distinct connected components of $\realTorusSplitCt[\rSeq]$. Proposition \ref{propNumCompHCtCover} implies $(\pm)\left\{\mac[{\simpleRootei[1]}],\ldots,\mac[{\simpleRootei[k]}]\right\}$ is a set of representatives for the distinct connected components of $\realTorusSplitCtCover[\rSeq]$. Since the elements $\pm\mac[\simpleRootei]$ are contained in $\mGroupTorusCover[\realTorusSplitCover]$ by definition, their multiplicative structure is given by Proposition \ref{propMCStruct}. Combined with Lemma \ref{lemAbelianCompGroup}, this gives a complete description of the multiplicative structure of $\realTorusSplitCtCover[\rSeq]$.

\section{Genuine Triples}
\label{secGenTriples}

\subsection{Half-Integral Infinitesimal Characters}
\label{ssInfChar}

Conjugation to $\absLieAlgebra$ has been our main tool for relating algebraic structures associated with different Cartan subalgebras in $\complexLieAlgebra$. To this point, maps relating Cartan subalgebras have been specified only up to Weyl group conjugation. From now on we will be working with finer structure and are thus forced to be more precise about our conjugation maps.

\begin{definition}
\label{defAbsConjMap}
Fix a nonsingular element $\absDiff \in \absLieAlgebraDual$. Let $\complexLieAlgebra[h]$ be a Cartan subalgebra of $\complexLieAlgebra$, $\diff \in \complexLieAlgebra[h]^{*}$, and suppose $\absDiff$ and $\diff$ define the same infinitesimal character (Section \ref{ssHCModIntro}). Then there is an inner automorphism
\[
\absConjTwo{\absDiff} : \complexGroup \to \complexGroup
\]
whose differential induces a map (also denoted $\absConjTwo{\absDiff}$)
\begin{eqnarray*}
\absLieAlgebraDual & \to & \complexLieAlgebra[h]^{*} \\
\absDiff & \to & \diff.
\end{eqnarray*}
\end{definition}

The map $\absConjTwo{\absDiff}$ is not unique, however the restriction of any two such maps to $\absLieAlgebraDual$ is the same. Hence we have a well-defined family of maps $\left\{\absConjTwo[\diff']{\absDiff}\right\}$ taking $\absLieAlgebraDual$ to $\complexLieAlgebra[h]^{*}$, where $\diff' \in \complexLieAlgebra[h]^{*}$ is in the $\weylGroup(\complexLieAlgebra, \complexLieAlgebra[h])$-orbit of $\diff$. We also write $\absConjTwo{\absDiff}$ for the induced maps on Weyl groups, root systems, and so forth. When the element $\absDiff$ is fixed or clear from context, we will often just write $\absConj$.

The Harish-Chandra isomorphism allows us to specify an infinitesimal character by selecting a $\weylGroup(\complexLieAlgebra, \absLieAlgebra)$-orbit in $\absLieAlgebraDual$. Since we have fixed a positive system for $\rootSystem(\complexLieAlgebra, \absLieAlgebra)$, a nonsingular infinitesimal character is uniquely determined by an element of the corresponding dominant chamber. For the purpose of conjugation, it will often be convenient to specify an infinitesimal character via a dominant element of $\absLieAlgebraDual$. In this context we may refer to a dominant nonsingular element of $\absLieAlgebraDual$ as an infinitesimal character. Combined with Definition \ref{defAbsConjMap}, we are led to the following enhancement of Definition \ref{defAbsPair}.

\begin{definition}
\label{defAbsTriple}
Fix an infinitesimal character $\absDiff \in \absLieAlgebraDual$, a $\cinv$-stable Cartan subgroup $\realTorusCover \subset \realGroupCover$, and suppose there exists a genuine triple $\genTripleInf$ as in Definition \ref{defTriple}. Write $\eta : \imaginaryRoots[\cinv](\complexLieAlgebra,\complexLieAlgebra[h]) \to \mathbb{Z}_{2}$ for the corresponding grading of $\imaginaryRoots[\cinv](\complexLieAlgebra,\complexLieAlgebra[h])$ (Proposition \ref{propGrading}). The \emph{abstract triple} associated to $\genTripleInf$ is the 3-tuple $\absTriple$, where
\begin{eqnarray*}
\inv & = & \absConjTwoInv{\absDiff} \cdot \cinv \cdot \absConjTwo{\absDiff} \\
\grading(\simpleRoot) & = & \eta(\absConjTwo{\absDiff}(\simpleRoot))
\end{eqnarray*}
with $\simpleRoot$ an abstract imaginary root for $\inv$. In particular, $\inv$ is an involution of $\rootSystem(\complexLieAlgebra, \absLieAlgebra)$ corresponding to $\cinv$ and $\grading$ is a grading of $\imaginaryRoots[\inv](\complexLieAlgebra, \absLieAlgebra)$.
\end{definition}

Fix a genuine triple $\genTripleInf$ and recall $\genTorusChar$ is a genuine representation of $\realTorusCover$ compatible with $\diff$. Technically the abstract triple of Definition \ref{defAbsTriple} is determined entirely by $\absConjTwo{\absDiff}$ and the pair $\genPair$. However, the following proposition implies the existence of $\genTorusChar$ places restrictions on the abstract $\absDiff$ that are allowed.

\begin{proposition}[\cite{TA}] \label{propSuppRep}
Let $\genTripleInf$ be a genuine triple and suppose $\absTriple$ is the corresponding abstract triple. If $\simpleRoot \in \rootSystem$ is a long imaginary root for $\inv$ then
\[
(\absDiff, \rootCheck) \in  \left\{
\begin{array}{ll}
\mathbb{Z} & \grading(\simpleRoot) = 0\\
\mathbb{Z} + \frac{1}{2} & \grading(\simpleRoot) = 1\\
\end{array}
\right..
\]
If $\simpleRoot$ is a long complex root for $\inv$, then we have
\[
(\absDiff, \rootCheck + \inv(\rootCheck)) \in \left\{
\begin{array}{ll}
\mathbb{Z} & (\simpleRoot, \inv(\rootCheck)) = 0\\
\mathbb{Z} + \frac{1}{2} & (\simpleRoot, \inv(\rootCheck)) \ne 0\\
\end{array}
\right..
\]
\end{proposition}

Define a half-integer to be a number $x$ such that $2x \in \mathbb{Z}$ and a \emph{strict} half-integer to be an element of $\mathbb{Z} + \frac{1}{2}$. Because of Proposition \ref{propSuppRep}, we will generally be making the following assumption.

\begin{assumption}
\label{ass1}
Unless otherwise stated, an infinitesimal character $\absDiff \in \absLieAlgebraDual$ is assumed to be half-integral when written in abstract coordinates.
\end{assumption}

\begin{definition}
\label{defSuppTriple}
An abstract triple satisfying the conditions of Proposition \ref{propSuppRep} is said to be \emph{supportable}. 
\end{definition}

Fix $\absDiff \in \absLieAlgebraDual$ half-integral and let $\realTorusCover$ be a $\cinv$-stable Cartan subgroup of $\realGroupCover$. Suppose $\diff \in \complexLieAlgebra[h]^{*}$ such that $\absDiff$ and $\diff$ define the same infinitesimal character. We would like to know when there exists a genuine triple of the form $\genTripleInf$. Conjugating the pair $\genPair$ by $\absConjTwoInv{\absDiff}$ gives an abstract triple $\absTriple$ and Proposition \ref{propSuppRep} provides a necessary condition on $\absTriple$. A wonderful fact is that this condition is also sufficient.

\begin{proposition}[\cite{TA}] \label{propIffSupp}
In the situation above, a genuine triple $\genTripleInf$ exists if and only if the abstract triple $\absTriple$ is supportable.
\end{proposition}

\subsection{Genuine Triples}
\label{ssGenTriples}

Let $\realTorusCover \subset \realGroupCover$ be a $\cinv$-stable Cartan subgroup, fix a regular element $\diff \in \complexLieAlgebraDual[h]$, and suppose $\absTriple$ is the abstract triple corresponding to the pair $\genPair$. If $\absTriple$ is supportable (Definition \ref{defSuppTriple}), Proposition \ref{propIffSupp} implies there exists a genuine character $\genTorusChar$ such that $\genTripleInf$ is a genuine triple for $\realGroupCover$. In this section we determine how many such extensions exist.

\begin{definition}
\label{defNumGenTriples}
Set $\numCorGenTriplesInf$ to be the number of distinct (conjugacy classes of) genuine triples extending $\genPairInf$.
\end{definition}

Let $\genReps$ denote the set of (equivalence classes of) irreducible genuine representations of $\realTorusCover$ and let $\genRepsCenter$ denote the same set for the center of $\realTorusCover$. The key to computing $\numCorGenTriplesInf$ is the following proposition. 

\begin{proposition}[\cite{ABVPT}, Proposition 2.2]
\label{propGenRepBij}
There is a bijection
\[
\pi : \genRepsCenter \to \genReps
\]
sending an element $\chi \in \genRepsCenter$ to $\pi(\chi) \in \genReps$. Here $\pi(\chi)$ is the unique element of $\genReps$ satisfying $\pi(\chi)|_{Z(\realTorusCover)} = n\chi$, where $n = \left|\realTorusCover/Z(\realTorusCover)\right|^{\frac{1}{2}}$ is the dimension of $\pi(\chi)$.
\end{proposition} 

Suppose there exists a genuine triple $\genTripleInf$ with infinitesimal character $\absDiff \in \absLieAlgebraDual$. Proposition \ref{propGenRepBij} implies the possibilities for $\genTorusChar$ are parameterized by (genuine) central characters of $\realTorusCover$ compatible with $\diff$ (Section \ref{ssHCModIntro}). However, $\realTorusCoverId$ is central in $\realTorusCover$ by Lemma \ref{lemAbelianCompGroup} and the behavior of $\genTorusChar$ on $\realTorusCoverId$ is specified by $\diff$. In particular, the distinct possibilities for $\genTorusChar$ are determined by the structure of the connected components of $\realTorusCover$ and this structure was determined in Section \ref{AlgTorusCover}. Putting everything together we obtain the following results. As usual, it suffices to consider only the split form $\realGroupSplit$.

\begin{proposition} \label{propNumGenReps}
Fix a split Cartan subgroup $\realTorusSplit \subset \realGroupSplit$. Let $\absDiff$ be a half-integral infinitesimal character, $\rSeq$ a sequence of orthogonal roots in $\rootSystem(\complexLieAlgebra, \splitLieAlgebra)$, and suppose $\absTriple$ is the abstract triple corresponding to the pair $\genPairSplitCtInf$. Let $\left\{\ma[{\simpleRootei[1]}],\ldots,\ma[{\simpleRootei[k]}]\right\}$ be a representative set for $\pi_{0}(\realTorusSplitCt[\rSeq])$ and assume $\absTriple$ is supportable. Then the number of genuine triples $\genTripleSplitCtInf[\rSeq]$ extending $\genPairSplitCtInf$ is given by
\begin{eqnarray*}
\numCorGenTriplesSplitCtInf & = & \left|\left\{\mac[\simpleRootei] \mid \left[\mac[\simpleRootei],\mac[{\simpleRootei[j]}]\right] = 1, \text{\emph{ for all }} 1 \le j \le k\right\}\right| \\
& = & \left|\left\{\mac[\simpleRootei] \mid \mac[\simpleRootei] \text{ is central in } \realTorusSplitCtCover[\rSeq]\right\}\right|.
\end{eqnarray*}
\end{proposition}

\begin{remark}
The above statement does not depend on the choices for the $\mac[\simpleRootei]$.
\end{remark}

\begin{proof}
Suppose $-1 \in \realTorusSplitCtCoverId[\rSeq]$. Then the number of genuine triples extending $\genPairSplitCtInf$ is equal to the number of connected components of $\realTorusSplitCtCover[\rSeq]$ contained in the center. Since the connected components are represented by the elements $\left\{\mac[{\simpleRootei[1]}],\ldots,\mac[{\simpleRootei[k]}]\right\}$, we simply need to determine the central $\mac[\simpleRootei]$ and the result follows.

If $-1 \notin \realTorusSplitCtCoverId[\rSeq]$, then for each element of
\[
\left\{\mac[\simpleRootei] \mid \left[\mac[\simpleRootei],\mac[{\simpleRootei[j]}]\right] = 1, \text{\emph{ for all }} 1 \le j \le k\right\}
\]
there are two corresponding central connected components in $\realTorusSplitCtCover[\rSeq]$ that differ by a factor of $-1$. Since we are interested in \emph{genuine} representations of $\realTorusSplitCtCover[\rSeq]$, the action of $-1$ is fixed and we have only a single corresponding genuine character. Therefore the genuine characters are parameterized by the same set as the previous case and the result follows.
\end{proof}

\begin{theorem} \label{theoremNumGenReps}
Let $\absDiff$ be a half-integral infinitesimal character, $\rSeq$ an orthogonal sequence in $\rootSystem(\complexLieAlgebra,\splitLieAlgebra)$, and suppose the abstract triple $\absTriple$ corresponding to the pair $\genPairSplitCtInf$ is supportable. Then
\[
\numCorGenTriplesSplitCtInf \in \left\{1,2,4\right\}.
\]
\end{theorem}

\begin{proof}
Let
\begin{eqnarray*}
\simpleRooti & = & \ei - \ei[i+1] \\
\simpleRootbi & = & \ei
\end{eqnarray*}
with respect to the usual coordinates for a positive system in $\rootSystem(\complexLieAlgebra, \splitLieAlgebra)$. Every Cartan subgroup $\realTorusCover \subset \realGroupCover[\realGroupSplit]$ is $\maxRealCompactCover$-conjugate to one of the form $\realTorusSplitCtCover[\rSeq]$, with $\rSeq = {\simpleRooti[1]\simpleRooti[3]\ldots\simpleRooti[k-1]\simpleRootbi[k+1]\ldots\simpleRootbi[k+s]}$ for some numbers $k$ and $s$. We proceed by cases based on the form of $\rSeq$.
\begin{steplist}
\item [Case I]
Suppose $\rSeq = \emptyset$ and recall $n$ denotes the rank of $\realGroupCover[\realGroupSplit]$. If $n$ is even define $z = \mac[{\simpleRootai[1]}]\mac[{\simpleRootai[3]}]\cdots\mac[{\simpleRootai[n-1]}]$. It is easy to verify
\[
\text{Z}(\mGroupTorusCover[\realTorusSplitCover]) = \left\{\begin{array}{rl}
\pm\left\{1,\mac[{\simpleRootbi[n]}]\right\} & n \text{ odd} \\
\pm\left\{1,\mac[{\simpleRootbi[n]}],z,z\mac[{\simpleRootbi[n]}]\right\} & n \text{ even} \\
\end{array}\right.
\]
so that Proposition \ref{propNumGenReps} implies
\[
\numCorGenTriplesSplitInf = \left\{
\begin{array}{cl}
2 & n \text{ is odd } \\
4 & n \text{ is even }
\end{array}
\right..
\]

\item [Case II]
Suppose $\rSeq = {\simpleRootbi[1]\ldots\simpleRootbi[n]}$ so that $\invSplitCt[\rSeq] = \identity$. Then $\realTorusSplitCtCover[\rSeq]$ is connected and we have 
\[
\numCorGenTriplesSplitInf = 1.
\]

\item [Case III]
Suppose $\rSeq = {\simpleRootbi[1]\ldots\simpleRootbi[s]}$ with $0 < s < n$ and let $r=n-s$. Then the diagram (up to conjugacy) for $\realTorusSplitCtCover[\rSeq]$ is 
\[
\underset{1}{+}~\cdots \underset{s}{+}~-\cdots~\underset{n}{-}
\]
and Proposition \ref{propCompH} implies $\pi_{0}(\realTorusSplitCt[\rSeq])$ has generators
\[
\left\{\ma[{\simpleRooti[s+1]}],\ma[{\simpleRooti[s+2]}],\ldots,\ma[{\simpleRooti[n-1]}]\right\}.
\]
Arguing as in Case I it is easy to verify
\[
\numCorGenTriplesSplitCtInf = \left\{
\begin{array}{rl}
1 & \text{r is odd} \\
2 & \text{r is even}
\end{array}\right..
\]

\item [Case IV]
Suppose $\rSeq = {\simpleRooti[1]\simpleRooti[3]\ldots\simpleRooti[k-1]\simpleRootbi[k+1]\ldots\simpleRootbi[n]}$ with $0 < k < n$. Then the diagram (up to conjugacy) for $\realTorusSplitCtCover[\rSeq]$ is 
\[
\underset{1}{(2)}~(1)~\cdots~(k)~\underset{k}{(k-1)}~\underset{k+1}{+}~\cdots \underset{n}{+}
\]
and $\realTorusSplitCtCover[\rSeq]$ is connected. Therefore
\[
\numCorGenTriplesSplitCtInf = 1
\]
as desired.

\item [Case V]
Suppose $\rSeq = \simpleRooti[1]\simpleRooti[3]\ldots\simpleRooti[k-1]$ with $0 < k < n$ and let $r=n-k$. Then the diagram (up to conjugacy) for $\realTorusSplitCtCover[\rSeq]$ is
\[
\underset{1}{(2)}~(1)~\cdots~(k)~\underset{k}{(k-1)}~-~\cdots \underset{n}{-}
\]
and Proposition \ref{propCompH} implies $\pi_{0}(\realTorusSplitCt[\rSeq])$ has generators
\[
\left\{\ma[{\simpleRooti[k+1]}],\ma[{\simpleRooti[k+2]}],\ldots,\ma[{\simpleRooti[n-1]}],\ma[\simpleRootb]\right\}.
\]
Again as in case I it is easy to verify
\[
\numCorGenTriplesSplitCtInf = \left\{
\begin{array}{rl}
2 & \text{r is odd} \\
4 & \text{r is even}
\end{array}\right..
\]

\item [Case VI]
Suppose $\rSeq = {\simpleRooti[1]\simpleRooti[3]\ldots\simpleRooti[k-1]\simpleRootbi[k+1]\ldots\simpleRootbi[k+s]}$ with $0 < k < n$ and $0 < s < n-k$. Let $r = n-k-s$. Then the diagram (up to conjugacy) for $\realTorusSplitCtCover[\rSeq]$ is
\[
\underset{1}{(2)}~(1)~\cdots~(k)~\underset{k}{(k-1)}~\underset{k+1}{+}~\cdots \underset{k+s}{+}-~\cdots \underset{n}{-}
\]
and $\pi_{0}(\realTorusSplitCt[\rSeq])$ has generators
\[
\left\{\ma[{\simpleRooti[k+s+1]}],\ma[{\simpleRooti[k+s+2]}],\ldots,\ma[{\simpleRooti[n-1]}]\right\}.
\]
Again as in case I it is easy to verify
\[
\numCorGenTriplesSplitCtInf = \left\{
\begin{array}{rl}
1 & \text{r is odd} \\
2 & \text{r is even}
\end{array}\right..
\]

\item [Case VII]
Suppose $\rSeq = \simpleRooti[1]\simpleRooti[3]\ldots\simpleRooti[n-1]$ (possible only when $n$ is even). Then the diagram (up to conjugacy) for $\realTorusSplitCtCover[\rSeq]$ is
\[
\underset{1}{(2)}~(1)~\cdots~(n)~\underset{n}{(n-1)}
\]
and $\pi_{0}(\realTorusSplitCt[\rSeq])$ has a single generator $\left\{\ma[\simpleRootb]\right\}$. Therefore
\[
\numCorGenTriplesSplitCtInf = 2.
\]
\end{steplist}
\end{proof}

Recall the indicator bits $\imaginaryBitsBit$, $\realBitsBit$, and $\realBitsParityBit$ defined in Section \ref{ssAbsWeylGroup}. The results of Theorem \ref{theoremNumGenReps} are summarized in the following corollary.

\begin{corollary} \label{corNumGenRepsFormula}
Let $\absDiff$ be a half-integral infinitesimal character, $\rSeq$ a sequence of orthogonal roots in $\rootSystem(\complexLieAlgebra, \splitLieAlgebra)$, and suppose the abstract triple $\absTriple$ corresponding to the pair $\genPairSplitCtInf$ is supportable. Then
\begin{eqnarray*}
\numCorGenTriplesSplitCtInf & = & 2^{1-\imaginaryBitsBit}2^{\realBitsBit(1-\realBitsParityBit)}.
\end{eqnarray*}
\end{corollary}

\begin{proof}
Check this for each of the cases in Theorem \ref{theoremNumGenReps}.
\end{proof}

\section{$\maxRealCompact$-orbits}
\label{secKorbits}

Fix a nonsingular element $\absDiff \in \absLieAlgebraDual$. We have seen it is important to understand the $\maxRealCompact$-conjugacy classes of pairs $(\realTorus', \diff')_{\absDiff}$, where $\realTorus'$ is a $\cinv$-stable Cartan subgroup in $\realGroup = \realSpinGroup{p}{q}$ and $\diff'$ is a nonsingular element in $(\complexLieAlgebra[h]')^{*}$ such that $\diff'$ and $\absDiff$ define the same infinitesimal character. The $\maxRealCompact$-conjugacy classes of such pairs will be referred to as \emph{$\maxRealCompact$-orbits} (for $\absDiff$). If we fix a set $\left\{\realTorus_{i}\right\}$ of $\cinv$-stable Cartan subgroup representatives, the $\maxRealCompact$-orbits for $\absDiff$ are parameterized by $\realNormalizer[\maxRealCompact](\realTorus_{i})$-orbits on pairs of the form $\pairInfi$. The stabilizer of any such pair is $Z_{\maxRealCompact}(\realTorus_{i})$ and the quotient
\[
\cartanWeylGroup{\realTorus_{i}} = \realNormalizer[\maxRealCompact](\realTorus_{i}) / Z_{\maxRealCompact}(\realTorus_{i})
\]
is the real Weyl group for $\realTorus_{i}$. 

Fix a $\cinv$-stable Cartan subgroup $\realTorus$ and let $\weylOrbit$ denote the $\algWeylGroup{\complexLieAlgebra}{\complexLieAlgebra[h]}$-orbit of a nonsingular element $\diff \in (\complexLieAlgebra[h])^{*}$. Then $\cartanWeylGroup{\realTorus}$ acts freely on $\weylOrbit$ and thus freely on pairs of the form $\pairInf$. Therefore the $\maxRealCompact$-orbits whose first entry is conjugate to $\realTorus$ are parameterized by $\cartanWeylGroup{\realTorus}$-orbits in $\weylOrbit$. For this reason $\cartanWeylGroup{\realTorus}$-orbits in $\weylOrbit$ will also be referred to as $\maxRealCompact$-orbits for (the conjugacy class of) $\realTorus$. Since
\[
\left| \algWeylGroup{\complexLieAlgebra}{\complexLieAlgebra[h]}\right| = \left|\weylOrbit\right|
\]
the number of $\maxRealCompact$-orbits for $\realTorus$ is given by
\[
\left| \algWeylGroup{\complexLieAlgebra}{\complexLieAlgebra[h]} / \cartanWeylGroup{H} \right|.
\]

In Section \ref{ssInfChar} we constructed an abstract triple $\absTriple$ corresponding to each pair $\pairInf$. In this section we study the relationship between $\maxRealCompact$-orbits and abstract triples.


\subsection{The Cross Action on $\maxRealCompact$-orbits}
\label{ssCrossAction}

Let $\absDiff \in \absLieAlgebraDual$ be a nonsingular element, $\realTorus$ a $\cinv$-stable Cartan subgroup of $\realGroup$, and suppose $\diff \in \complexLieAlgebra[h]^{*}$ and $\absDiff$ determine the same infinitesimal character. Let $\weylOrbit \subset \complexLieAlgebra[h]^{*}$ denote the $\algWeylGroup{\complexLieAlgebra}{\complexLieAlgebra[h]}$-orbit of $\diff$ and suppose $\inv \in \algWeylGroup{\complexLieAlgebra}{\complexLieAlgebra[h]}$ is an involution. Recall Theorem \ref{theoremWeylGroupCentralizer} implies the centralizer of $\inv$ in $\algWeylGroup{\complexLieAlgebra}{\complexLieAlgebra[h]}$ has the form
\begin{eqnarray*}
\invWeylGroup(\complexLieAlgebra,\complexLieAlgebra[h]) & \cong & (\imaginaryWeylGroup(\complexLieAlgebra,\complexLieAlgebra[h]) \times \realWeylGroup(\complexLieAlgebra,\complexLieAlgebra[h])) \rtimes \complexWeylGroup(\complexLieAlgebra,\complexLieAlgebra[h]).
\end{eqnarray*} 

\begin{proposition}[\cite{IC4}, Proposition 4.16] \label {propCWGroup}
The group $\cartanWeylGroup{\realTorus}$ is a subgroup of $\invWeylGroup(\complexLieAlgebra,\complexLieAlgebra[h])$. Moreover
\begin{eqnarray*}
\cartanWeylGroup{\realTorus} & \cong & (\rGroup \times \realWeylGroup(\complexLieAlgebra,\complexLieAlgebra[h])) \rtimes \complexWeylGroup(\complexLieAlgebra,\complexLieAlgebra[h])
\end{eqnarray*}
where $\rGroup \subset \imaginaryWeylGroup(\complexLieAlgebra,\complexLieAlgebra[h])$ with $\rGroup \cong A \ltimes \cptImaginaryWeylGroup(\complexLieAlgebra,\complexLieAlgebra[h])$ and $A$ is an elementary abelian two-group.
\end{proposition}

This section makes frequent use of the conjugation maps $\left\{\absConjTwo[\diff']{\absDiff} = \absConj[\diff']\right\}_{\diff'\in\weylOrbit}$ (Definition \ref{defAbsConjMap}), so we now record some easy formal properties. To simplify notation, recall $\weylGroup = \weylGroup(\complexLieAlgebra, \absLieAlgebra)$ and $\rootSystem = \rootSystem(\complexLieAlgebra, \absLieAlgebra)$.

\begin{lemma} \label{lemConj}
Let $w, \cinv \in \algWeylGroup{\complexLieAlgebra}{\complexLieAlgebra[h]}$ and $\inv \in \weylGroup$. Then
\begin{eqnarray*}
\absConj[w\diff] & = & w \cdot \absConj \\
\absConjInv[w\diff] & = & \absConjInv \cdot w^{-1} \\
\absConj(\inv) & = & \absConj \cdot \inv \cdot \absConjInv \\
\absConjInv(\cinv) & = & \absConjInv \cdot \cinv \cdot \absConj \\
\absConj[w\diff](\inv) & = & w \cdot \absConj(\inv) \cdot w^{-1} \\
\absConjInv[w\diff](\cinv) & = & \absConjInv(w^{-1} \cdot \cinv \cdot w).
\end{eqnarray*}
\end{lemma}

\begin{proof}
The first two equalities are obvious and the second two are by definition. For the fifth equality we have
\begin{eqnarray*}
\absConj[w\diff](\inv) & = & \absConj[w\diff] \cdot \inv \cdot \absConjInv[w\diff] \\
& = & w \cdot \absConj \cdot \inv \cdot \absConjInv \cdot w^{-1} \\
& = & w \cdot \absConj(\inv) \cdot w^{-1}
\end{eqnarray*}
as desired. The last equality is just the corresponding inverse statement.
\end{proof}

\begin{proposition} \label{propKOGrading}
Let $\pairInf$ be a pair as above and suppose $\weylElt \in \cartanWeylGroup{\realTorus}$. Then the abstract triples associated with $\pairInf$ and $\pairInfDiff[\weylElt\diff]$ are the same.
\end{proposition}

\begin{proof}
Let $\cinv$ denote the Cartan involution and suppose $\absTriple$ is the abstract triple for $\pairInf$. By Lemma \ref{lemConj} and Proposition \ref{propCWGroup} we have
\begin{eqnarray*}
\absConjInv[\weylElt\diff](\cinv) & = & \absConjInv(w^{-1} \cdot \cinv \cdot w) \\
& = & \absConjInv(\cinv) \\
& = & \inv.
\end{eqnarray*}
Moreover if $\simpleRoot \in \complexLieAlgebra[h]^{*}$ is a root,
\begin{eqnarray*}
\absConjInv[\weylElt\diff](\simpleRoot) & = & \absConjInv \cdot w^{-1}(\simpleRoot).
\end{eqnarray*}
However, elements of $\cartanWeylGroup{\realTorus}$ preserve the set of compact roots and thus the corresponding abstract grading is unchanged.
\end{proof}

Proposition \ref{propKOGrading} implies the association of abstract triples to pairs $\pairInf$ is defined on the level of $\maxRealCompact$-orbits for $\realTorus$ (viewed as $\cartanWeylGroup{\realTorus}$-orbits in $\complexLieAlgebra[h]^{*}$). We would like to understand the extent to which this association is unique. We begin by recalling the familiar action of the abstract Weyl group on $\weylOrbit$.

\begin{definition}
\label{defCrossAction}
Let $\weylElt \in \weylGroup$ and $\diff \in \weylOrbit$. Define the \emph{cross action} of $\weylElt$ on $\diff$ as
\begin{eqnarray*}
\weylElt \cross \diff & = & \absConj \cdot w^{-1} \cdot \absConjInv (\diff).
\end{eqnarray*}
Alternatively we can define
\begin{eqnarray*}
\weylEltDiff & = & \absConj(\weylElt)
\end{eqnarray*}
so that
\begin{eqnarray*}
\weylElt \cross \diff & = & \weylEltDiffInv(\diff).
\end{eqnarray*}
If $\pairInf$ is the corresponding pair, the cross action of $\weylElt$ on $\pairInf$ will be denoted $\weylElt \cross \pairInf = \pairInfDiff[\weylElt \cross \diff]$. 
\end{definition}

\begin{lemma} \label{lemicross}
Let $\weylElt \in \weylGroup$ and $\diff \in \weylOrbit$. Then 
\begin{eqnarray*}
\absConj[\weylElt \cross \diff] & = & \absConj \cdot \weylElt^{-1} \\
\absConjInv[\weylElt \cross \diff] & = & \weylElt \cdot \absConjInv.
\end{eqnarray*}
\end{lemma}

\begin{proof}
It suffices to check the first equality. We have
\begin{eqnarray*}
\absConj[\weylElt \cross \diff] & = & \absConj[\weylEltDiffInv(\diff)] \\
& = & \weylEltDiffInv \cdot \absConj \\
& = & (\absConj(\weylElt))^{-1} \cdot \absConj \\
& = & (\absConj \cdot \weylElt \cdot \absConjInv)^{-1} \cdot \absConj \\
& = & \absConj \cdot \weylElt^{-1} \cdot \absConjInv \cdot \absConj \\
& = & \absConj \cdot \weylElt^{-1}
\end{eqnarray*}
as desired.
\end{proof}

\begin{proposition} \label{propCA}
The cross action defines a left action of $\weylGroup$ on $\weylOrbit$. Moreover, the cross action commutes with the usual action of $\algWeylGroup{\complexLieAlgebra}{\complexLieAlgebra[h]}$ and thus descends to a transitive action at the level of $\maxRealCompact$-orbits. 
\end{proposition}

\begin{proof}
This is well known but we prove it as an illustration of the above formalism. For $w_{1}, w_{2} \in \weylGroup$ and $\diff \in \weylOrbit$ we have
\begin{eqnarray*}
w_{2} \cross (w_{1} \cross \diff) & = & \absConj[w_{1} \cross \diff] \cdot w_{2}^{-1} \cdot \absConjInv[w_{1} \cross \diff] (w_{1} \cross \diff) \\
& = & \absConj \cdot w_{1}^{-1} \cdot w_{2}^{-1} \cdot w_{1} \cdot \absConjInv (\absConj \cdot w_{1}^{-1} \cdot \absConjInv (\diff)) \\
& = & \absConj \cdot w_{1}^{-1} \cdot w_{2}^{-1} \cdot \absConjInv(\diff) \\
& = & (w_{2}w_{1}) \cross \diff.
\end{eqnarray*}
This proves the first claim. For the second claim, let $w \in \weylGroup$ and $\sigma \in \algWeylGroup{\complexLieAlgebra}{\complexLieAlgebra[h]}$. Then
\begin{eqnarray*}
w \cross (\sigma\diff) & = & \absConj[\sigma\diff] \cdot w^{-1} \cdot \absConjInv[\sigma\diff](\sigma\diff) \\
& = & \sigma \cdot \absConj \cdot w^{-1} \cdot \absConjInv \cdot \sigma^{-1} (\sigma\diff) \\
& = & \sigma \cdot \absConj \cdot w^{-1} \cdot \absConjInv (\diff) \\
& = & \sigma \cdot w \cross \diff
\end{eqnarray*}
as desired.
\end{proof}

If $\diff \in \weylOrbit$, we denote the $\maxRealCompact$-orbit of $\diff$ by $\korbit$ and the corresponding cross action by $\weylElt \cross \korbit$. We conclude this section with a generalization of Proposition \ref{propKOGrading}.

\begin{proposition} \label {propCATriples}
Let $\weylElt \in \weylGroup$ and suppose $\pairInf$ is a pair with corresponding abstract triple $\absTriple$. Then the abstract triple associated to the pair $\weylElt \cross \pairInf$ is given by
\[
\weylElt \cross \absTriple = (\weylElt\cdot\inv\cdot\weylElt^{-1}, \weylElt \cross \grading, \absDiff)
\]
where $\weylElt \cross \grading$ is the grading defined via
\[
(\weylElt \cross \grading)(\simpleRoot) = \grading (\weylElt^{-1}\simpleRoot).
\]
In particular, the cross action on pairs induces an action (also called the cross action) on abstract triples.
\end{proposition}

\begin{proof}
Let $\cinv$ denote the Cartan involution. Then
\begin{eqnarray*}
\absConjInv[\weylElt \cross \diff](\cinv) & = & \absConjInv[\weylElt \cross \diff] \cdot \cinv \cdot \absConj \\
& = & \weylElt \cdot \absConjInv \cdot \cinv \cdot \absConj \cdot \weylElt^{-1} \\
& = & \weylElt \cdot \absConjInv(\cinv) \cdot \weylElt^{-1} \\
& = & \weylElt \cdot \inv \cdot \weylElt^{-1}.
\end{eqnarray*}
So the abstract involution induced by the pair $\weylElt \cross \pairInf$ is $\weylElt$-conjugate to the involution induced by $\pairInf$. Now let $\simpleRoot \in \complexLieAlgebra[h]^{*}$ be a root. By Lemma \ref{lemicross} we have
\begin{eqnarray*}
\absConjInv[\weylElt \cross \diff](\simpleRoot) & = & \weylElt \cdot \absConjInv(\simpleRoot).
\end{eqnarray*}
Therefore the cross action alters the abstract root correspondence by the regular action of $\weylElt$. Hence the new abstract grading $\weylElt \cross \grading$ will be the same as the abstract grading for $\pairInf$ if we first compose $\grading$ with $\weylElt^{-1}$. Finally, it is trivial to check this defines an action on the set of abstract triples.
\end{proof}

\subsection{The Fiber Over an Involution}
\label{ssInvFiber}

We begin with our usual setup. Let $\absDiff$ be a nonsingular element in $\absLieAlgebraDual$, $\realTorus$ a $\cinv$-stable Cartan subgroup of $\realGroup$, and $\diff \in \complexLieAlgebra[h]^{*}$ such that $\diff$ and $\absDiff$ define the same infinitesimal character. Let $\weylOrbit \subset \complexLieAlgebra[h]^{*}$ be the $\algWeylGroup{\complexLieAlgebra}{\complexLieAlgebra[h]}$-orbit of $\diff$ and suppose $\absTriple$ is the abstract triple corresponding to $\pairInf$. 

\begin{definition}
\label{defFiberInv}
The \emph{fiber over} $\inv$ (denoted $\fiber$) is the set of $\maxRealCompact$-orbits in $\weylOrbit$ whose corresponding abstract triples begin with $\inv$. Similarly, the \emph{fiber over} $\absTriple$ (denoted $\absTripleFiber$) is defined to be the set of $\maxRealCompact$-orbits in $\weylOrbit$ with associated abstract triple $\absTriple$.
\end{definition}

The order of the fiber over a fixed involution is easy to describe.

\begin{proposition} \label{propFiberOrder}
Let $\pairInf$ be a pair with corresponding abstract triple $\absTriple$. Then
\begin{eqnarray*}
\fiberOrder = \left| \frac{\invWeylGroup}{\cartanWeylGroup{\realTorus}} \right|
\end{eqnarray*}
where $\cartanWeylGroup{\realTorus} \subset \weylGroup$ denotes the image of the real Weyl group for $\realTorus$ under the map $\absConjInv$.
\end{proposition}

\begin{proof}
By Proposition \ref{propCA} and Proposition \ref{propCATriples}, $\invWeylGroup$ acts transitively on $\fiber$. Moreover, for $w \in \weylGroup$ we have (by definition)
\[
w \cross \korbit = \korbit \Longleftrightarrow w \in \cartanWeylGroup{\realTorus}.
\]
So $\cartanWeylGroup{\realTorus}$ is the stabilizer of this action and the result follows.
\end{proof}

Proposition \ref{propFiberOrder} implies the set of $\maxRealCompact$-orbits for $\realTorus$ can be viewed as a product of involutions and fibers. Numerically we have
\[
\underset{\text{\# of $\maxRealCompact$-orbits for $\realTorus$}}{\left|\frac{\weylGroup(\complexLieAlgebra, \complexLieAlgebra[h])}{\cartanWeylGroup{H}}\right|} = \underset{\text{\# of involutions}}{\left|\frac{\weylGroup(\complexLieAlgebra, \complexLieAlgebra[h])}{\invWeylGroup}\right|} \times \underset{\text{size of each fiber}}{\left|\frac{\invWeylGroup}{\cartanWeylGroup{H}}\right|}.
\]

In addition to determining $\fiberOrder$, it will be important to understand the transitive action of $\invWeylGroup$ on $\fiber$. Since the situation in not changed by conjugation, it suffices to determine this action for a set of representative involutions in $\weylGroup$. According to Proposition \ref{propCWGroup}, the action of $w \in \invWeylGroup$ on $\fiber$ is trivial if $w$ is an element of $\cptImaginaryWeylGroup, \realWeylGroup$, or $\complexWeylGroup$. Hence it remains to understand the action of $\imaginaryWeylGroup$ on $\fiber$, and specifically the action for (reflections corresponding to) noncompact imaginary roots.

\begin{definition}
\label{defTypeITypeII}
A noncompact imaginary root in $\rootSystem(\complexLieAlgebra,\complexLieAlgebra[h])$ is said to be of \emph{type I} if the corresponding reflection is not induced by an element of $\maxRealCompact$ (i.e., by an element in $\realNormalizer[\maxRealCompact](\realTorus))$. Otherwise it is said to be of \emph{type II}, with analogous terminology used for associated abstract roots. In particular, the cross action through an abstract noncompact imaginary root is nontrivial if and only if it is of type I.
\end{definition}

We now outline an effective method for describing the action of $\imaginaryWeylGroup$ on $\fiber$. Let $\pairInf$ be a pair with corresponding abstract triple $\absTriple$ and suppose the imaginary roots $\simpleRoota,\simpleRootb \in \imaginaryRoots$ are noncompact. We have seen there exists a `dual' pair $\pairDual$ whose corresponding abstract involution is $-\inv$. At the end of Section \ref{ssMGroupTorus} we associated the following data to $-\inv$
\begin{eqnarray*}
\mGroup & = & \coRootLattice / 2\coRootLattice \\
\posCorootLattice[-\inv] & = & \left\{\simpleRoot \in \coRootLattice \mid -\inv(\simpleRoot) = \simpleRoot\right\} \\
\negCorootLattice[-\inv] & = & \left\{\simpleRoot \in \coRootLattice \mid -\inv(\simpleRoot) = -\simpleRoot\right\} \\
\posQuotCorootLattice[-\inv] & = & \quotWeightLatticeMap(\posCorootLattice[-\inv]) \\
\negQuotCorootLattice[-\inv] & = & \quotWeightLatticeMap(\negCorootLattice[-\inv]) \\
\posnegQuotCorootLattice[-\inv] & = & \posQuotCorootLattice[-\inv] \cap \negQuotCorootLattice[-\inv].
\end{eqnarray*}
Note the roots $\simpleRoota, \simpleRootb$ are real for $-\inv$ and recall the associated elements in $\negQuotCorootLattice[-\inv]$ were denoted $\ma[\simpleRoota],\ma[\simpleRootb]$.
\begin{definition}
\label{defMca}
Let $\weylElt \in \imaginaryWeylGroup$ and write $\weylElt = \rootReflection[{\simpleRooti[n]}] \ldots \rootReflection[{\simpleRooti[1]}]$ as a minimal length product of simple reflections in $\imaginaryWeylGroup$. Define
\begin{eqnarray*}
\mca{\weylElt} & = & \grading(\simpleRooti[1])\ma[{\simpleRooti[1]}] + \rootReflection[{\simpleRooti[1]}] \cross \grading(\simpleRooti[2])\ma[{\simpleRooti[2]}] + \cdots + \rootReflection[{\simpleRooti[n-1]}] \cross \cdots \cross \rootReflection[{\simpleRooti[1]}] \cross \grading(\simpleRooti[n])\ma[{\simpleRooti[n]}] \\
& = & \grading(\simpleRooti[1])\ma[{\simpleRooti[1]}] + \grading(\rootReflection[{\simpleRooti[1]}]\simpleRooti[2])\ma[{\simpleRooti[2]}] + \cdots + \grading(\rootReflection[{\simpleRooti[1]}]\cdots\rootReflection[{\simpleRooti[n-1]}]\simpleRooti[n])\ma[{\simpleRooti[n]}]
\end{eqnarray*}
where the $\ma[{\simpleRooti}]$ are viewed as elements in $\negQuotCorootLattice[-\inv] / \posnegQuotCorootLattice[-\inv]$. 
\end{definition}

Note the element $\mca{\weylElt}$ depends on the grading $\grading$ and thus on the $\maxRealCompact$-orbit $\korbit$.

\begin{proposition}
The definition above gives a well-defined map 
\[
\weylElt \mapsto \mca{\weylElt}
\]
of $\imaginaryWeylGroup$ into $\negQuotCorootLattice[-\inv] / \posnegQuotCorootLattice[-\inv]$.
\end{proposition}

\begin{proof} The issue is the choice of reduced expression for $\weylElt$. Proposition \ref{propInvWeylGroups} implies
\[
\imaginaryWeylGroup \cong W(B_{m}) \times W(A_{1})^{l}
\]
for some integers $m$ and $l$ and thus it suffices to consider the following three cases (\cite{BB}, Theorem 3.3.1).
\begin{steplist}
\item [Case I]
Suppose $\rootReflection[\simpleRoota]\rootReflection[\simpleRootb] = \rootReflection[\simpleRootb]\rootReflection[\simpleRoota]$, where $\simpleRoota,\simpleRootb$ are simple roots in $\imaginaryRoots$. Then we have
\begin{eqnarray*}
\mca{\rootReflection[\simpleRootb]\rootReflection[\simpleRoota]} & = & \grading(\simpleRoota)\ma[\simpleRoota] + \grading(\rootReflection[\simpleRoota]\simpleRootb)\ma[\simpleRootb] \\
& = & \grading(\simpleRoota)\ma[\simpleRoota] + \grading(\simpleRootb)\ma[\simpleRootb] \\
& = & \grading(\simpleRootb)\ma[\simpleRootb] +  \grading(\rootReflection[\simpleRootb]\simpleRoota)\ma[\simpleRoota] \\
& = & \mca{\rootReflection[\simpleRoota]\rootReflection[\simpleRootb]}.
\end{eqnarray*}

\item [Case II]
Suppose $\rootReflection[\simpleRoota]\rootReflection[\simpleRootb]\rootReflection[\simpleRoota] = \rootReflection[\simpleRootb]\rootReflection[\simpleRoota]\rootReflection[\simpleRootb]$, where $\simpleRoota,\simpleRootb$ are (long) adjacent simple roots in $\imaginaryRoots$. Then we have
\begin{eqnarray*}
\mca{\rootReflection[\simpleRoota]\rootReflection[\simpleRootb]\rootReflection[\simpleRoota]} & = & \grading(\simpleRoota)\ma[\simpleRoota] + \grading(\rootReflection[\simpleRoota]\simpleRootb)\ma[\simpleRootb] + \grading(\rootReflection[\simpleRoota]\rootReflection[\simpleRootb]\simpleRoota)\ma[\simpleRoota] \\
& = & \grading(\simpleRoota)\ma[\simpleRoota] + \grading(\simpleRoota+\simpleRootb)\ma[\simpleRootb] + \grading(\simpleRootb)\ma[\simpleRoota] \\
& = & \grading(\simpleRoota+\simpleRootb)\ma[\simpleRoota] + \grading(\simpleRoota+\simpleRootb)\ma[\simpleRootb] \\
& = & \grading(\simpleRootb)\ma[\simpleRootb] + \grading(\simpleRootb+\simpleRoota)\ma[\simpleRoota] + \grading(\simpleRoota)\ma[\simpleRootb] \\
& = & \grading(\simpleRootb)\ma[\simpleRootb] + \grading(\rootReflection[\simpleRootb]\simpleRoota)\ma[\simpleRoota] + \grading(\rootReflection[\simpleRootb]\rootReflection[\simpleRoota]\simpleRootb)\ma[\simpleRootb] \\
& = & \mca{\rootReflection[\simpleRootb]\rootReflection[\simpleRoota]\rootReflection[\simpleRootb]}.
\end{eqnarray*}

\item [Case III]
Suppose $\rootReflection[\simpleRoota]\rootReflection[\simpleRootb]\rootReflection[\simpleRoota]\rootReflection[\simpleRootb] = \rootReflection[\simpleRootb]\rootReflection[\simpleRoota]\rootReflection[\simpleRootb]\rootReflection[\simpleRoota]$, where $\simpleRoota,\simpleRootb$ are adjacent simple roots in $\imaginaryRoots$ with $\simpleRoota$ short. Then
\begin{eqnarray*}
\mca{\rootReflection[\simpleRootb]\rootReflection[\simpleRoota]\rootReflection[\simpleRootb]\rootReflection[\simpleRoota]} & = & \grading(\simpleRoota)\ma[\simpleRoota] + \grading(\rootReflection[\simpleRoota]\simpleRootb)\ma[\simpleRootb] + \grading(\rootReflection[\simpleRoota]\rootReflection[\simpleRootb]\simpleRoota)\ma[\simpleRoota] + \grading(\rootReflection[\simpleRoota]\rootReflection[\simpleRootb]\rootReflection[\simpleRoota]\simpleRootb)\ma[\simpleRootb] \\
& = & \grading(\simpleRoota)\ma[\simpleRoota] + \grading(\simpleRootb)\ma[\simpleRootb] + \grading(\simpleRoota + \simpleRootb)\ma[\simpleRoota] + \grading(\simpleRootb)\ma[\simpleRootb] \\
& = & \grading(\simpleRoota)\ma[\simpleRoota] + \grading(\simpleRoota)\ma[\simpleRoota] + \grading(\simpleRootb)\ma[\simpleRootb] + \grading(\simpleRootb)\ma[\simpleRootb] + \grading(\simpleRootb)\ma[\simpleRoota] \\
& = & \grading(\simpleRootb)\ma[\simpleRoota].
\end{eqnarray*}
Similarly
\begin{eqnarray*}
\mca{\rootReflection[\simpleRoota]\rootReflection[\simpleRootb]\rootReflection[\simpleRoota]\rootReflection[\simpleRootb]} & = & \grading(\simpleRootb)\ma[\simpleRootb] + \grading(\rootReflection[\simpleRootb]\simpleRoota)\ma[\simpleRoota] + \grading(\rootReflection[\simpleRootb]\rootReflection[\simpleRoota]\simpleRootb)\ma[\simpleRootb] + \grading(\rootReflection[\simpleRootb]\rootReflection[\simpleRoota]\rootReflection[\simpleRootb]\simpleRoota)\ma[\simpleRoota] \\
& = & \grading(\simpleRootb)\ma[\simpleRootb] + \grading(\simpleRootb + \simpleRoota)\ma[\simpleRoota] + \grading(\simpleRootb)\ma[\simpleRootb] + \grading(\simpleRoota)\ma[\simpleRoota] \\
& = & \grading(\simpleRootb)\ma[\simpleRootb] + \grading(\simpleRootb)\ma[\simpleRootb] + \grading(\simpleRoota)\ma[\simpleRoota] + \grading(\simpleRoota)\ma[\simpleRoota] + \grading(\simpleRootb)\ma[\simpleRoota] \\
& = & \grading(\simpleRootb)\ma[\simpleRoota]
\end{eqnarray*}
and thus
\[
\mca{\rootReflection[\simpleRootb]\rootReflection[\simpleRoota]\rootReflection[\simpleRootb]\rootReflection[\simpleRoota]} = \mca{\rootReflection[\simpleRoota]\rootReflection[\simpleRootb]\rootReflection[\simpleRoota]\rootReflection[\simpleRootb]}.
\]
\end{steplist}
\end{proof}

The following proposition computes the elements $\mca{{\rootReflection[\simpleRoota]}}$, where $\simpleRoota \in \rootSystem$ is an arbitrary (not necessarily simple) imaginary root. The result is as expected.

\begin{proposition}
Let $\korbit$ be a $\maxRealCompact$-orbit in $\fiber$ and let $\simpleRootb \in \rootSystem$ be an imaginary root for $\inv$. Then
\[
\mca{\rootReflection[\simpleRootb]} = \grading(\simpleRootb)\ma[\simpleRootb].
\]
\end{proposition}

\begin{proof}
It suffices to consider the case when $\imaginaryWeylGroup \cong W(B_{m})$. Denote the long simple roots in $\imaginaryRoots$ by 
\[
\simpleRooti = \ei - \ei[i+1]
\]
for $1 \le i \le m-1$. We have the following three cases for $\simpleRootb$.

\begin{steplist}
\item [Case I]
Suppose $\simpleRootb$ is of the form
\[
\simpleRootb = \ei - \ei[j+1] = \simpleRooti + \simpleRooti[i+1] + \cdots + \simpleRooti[j]
\]
for $1 \le i \le j$ so that
\[
\rootReflection[\simpleRootb] = \rootReflection[{\simpleRooti[j]}]\rootReflection[{\simpleRooti[j-1]}] \cdots \rootReflection[{\simpleRooti[i+1]}]\rootReflection[{\simpleRooti[i]}]\rootReflection[{\simpleRooti[i+1]}] \cdots \rootReflection[{\simpleRooti[j-1]}]\rootReflection[{\simpleRooti[j]}]
\]
is a reduced expression for $\rootReflection[{\simpleRootb}]$. Then
\begin{eqnarray*}
\mca{\rootReflection[\simpleRootb]} & = & \grading(\simpleRooti[j])\ma[{\simpleRooti[j]}] + \grading(\rootReflection[{\simpleRooti[j]}]\simpleRooti[j-1])\ma[{\simpleRooti[j-1]}] + \cdots + \grading(\rootReflection[{\simpleRooti[j]}]\cdots\rootReflection[{\simpleRooti[i+1]}]\simpleRooti[i])\ma[{\simpleRooti[i]}] \\
& & {} +\grading(\rootReflection[{\simpleRooti[j]}]\cdots\rootReflection[{\simpleRooti[i]}]\simpleRooti[i+1])\ma[{\simpleRooti[i+1]}] + \cdots + \grading(\rootReflection[{\simpleRooti[j]}]\cdots\rootReflection[{\simpleRooti[i]}]\cdots\rootReflection[{\simpleRooti[j-1]}]\simpleRooti[j])\ma[{\simpleRooti[j]}] \\
& = & \grading(\simpleRooti[j])\ma[{\simpleRooti[j]}] + \grading(\simpleRooti[j] + \simpleRooti[j-1])\ma[{\simpleRooti[j-1]}] + \cdots + 
\grading(\simpleRooti[j]+\simpleRooti[j-1]+\cdots+\simpleRooti)\ma[{\simpleRooti[i]}] \\
& & {} + \grading(\simpleRooti[i])\ma[{\simpleRooti[i+1]}] + \grading(\simpleRooti[i]+\simpleRooti[i+1])\ma[{\simpleRooti[i+2]}] + \cdots + \grading(\simpleRooti[i]+\cdots+\simpleRooti[j-1])\ma[{\simpleRooti[j]}] \\
& = & \grading(\simpleRooti[i] + \cdots + \simpleRooti[j])\ma[{\simpleRooti[i]}] + \cdots + \grading(\simpleRooti[i] + \cdots + \simpleRooti[j])\ma[{\simpleRooti[j]}] \\
& = & \grading(\simpleRootb)\ma[{\simpleRooti[i]}] + \cdots + \grading(\simpleRootb)\ma[{\simpleRooti[j]}] \\
& = & \grading(\simpleRootb)\ma[\simpleRootb]
\end{eqnarray*}
as desired.

\item [Case II]
Suppose $\simpleRootb$ is of the form
\[
\simpleRootb = \ei = \simpleRooti + \simpleRooti[i+1] + \cdots + \simpleRooti[m-1] + \ei[m]
\]
for $1 \le i \le m$ so that
\[
\rootReflection[\simpleRootb] = \rootReflection[{\simpleRooti[i]}]\rootReflection[{\simpleRooti[i+1]}] \cdots \rootReflection[{\simpleRooti[m-1]}]\rootReflection[{\ei[m]}]\rootReflection[{\simpleRooti[m-1]}] \cdots \rootReflection[{\simpleRooti[i+1]}]\rootReflection[{\simpleRooti[i]}]
\]
is a reduced expression for $\rootReflection[{\simpleRootb}]$. Then
\begin{eqnarray*}
\mca{\rootReflection[\simpleRootb]} & = & \grading(\simpleRooti[i])\ma[{\simpleRooti[i]}] + \grading(\rootReflection[{\simpleRooti[i]}]\simpleRooti[i+1])\ma[{\simpleRooti[i+1]}] + \cdots + \grading(\rootReflection[{\simpleRooti[i]}]\cdots\rootReflection[{\simpleRooti[m-1]}]\ei[m])\ma[{\ei[m]}] \\
& & {} + \grading(\rootReflection[{\simpleRooti[i]}]\cdots\rootReflection[{\ei[m]}]\simpleRooti[m-1])\ma[{\simpleRooti[m-1]}] + \cdots + \grading(\rootReflection[{\simpleRooti[i]}]\cdots\rootReflection[{\ei[m]}]\cdots\rootReflection[{\simpleRooti[i+1]}]\simpleRooti[i])\ma[{\simpleRooti[i]}] \\
& = & \grading(\simpleRooti[i])\ma[{\simpleRooti[i]}] + \grading(\simpleRooti[i]+\simpleRooti[i+1])\ma[{\simpleRooti[i+1]}] + \cdots + 
\grading(\simpleRooti[i]+\cdots+\simpleRooti[m-1]+\ei[m])\ma[{\ei[m]}] \\
& & {} + \grading(\simpleRooti[i] + \cdots + \simpleRooti[m-1] + 2\ei[m])\ma[{\simpleRooti[m-1]}] + \cdots + \\
& & {} + \grading(\simpleRooti[i] + 2\simpleRooti[i+1] + \cdots + 2\simpleRooti[m-1] + 2\ei[m])\ma[{\simpleRooti[i]}] \\
& = & \grading(2\simpleRootb)\ma[{\simpleRooti[i]}] + \cdots + \grading(2\simpleRootb)\ma[{\simpleRooti[m-1]}] + \grading(\simpleRootb)\ma[{\ei[m]}] \\
& = & \grading(\simpleRootb)\ma[{\ei[m]}] \\
& = & \grading(\simpleRootb)\ma[\simpleRootb]
\end{eqnarray*}
as desired.

\item [Case III]
Suppose $\simpleRootb = \ei + \ei[j+1]$. This case is handled in the same fashion as the previous cases. The reader is spared the details.
\end{steplist}
\end{proof}

Let $\korbit$ and $\korbit[\psi]$ be two $\maxRealCompact$-orbits for $\realTorus$ in $\fiber$. For $\weylElt \in \imaginaryWeylGroup$, the following proposition describes how $\mca{\weylElt}$ and $\mca[{\korbit[\psi]}]{\weylElt}$ are related.

\begin{proposition} \label{propMcaRel}
Let $\korbit$ and $\korbit[\psi]$ be two $\maxRealCompact$-orbits in $\fiber$ and suppose $\korbit[\psi] = \tau \cross \korbit$. For $\weylElt \in \imaginaryWeylGroup$ we have
\[
\mca[{\korbit[\psi]}]{\weylElt} = \mca{\weylElt\tau} + \mca{\tau}.
\]
\end{proposition}

\begin{proof}
Let $\grading$ denote the abstract grading corresponding to $\korbit$ and choose a reduced expression $\weylElt = \rootReflection[{\simpleRooti[n]}] \cdots \rootReflection[{\simpleRooti[1]}]$. By definition we have
\begin{eqnarray*}
\mca[{\korbit[\psi]}]{\weylElt} & = & \mca[\tau \cross \korbit]{\weylElt} \\
& = & \tau \cross \grading(\simpleRooti[1])\ma[{\simpleRooti[1]}] + \tau \cross \grading(\rootReflection[{\simpleRooti[1]}]\simpleRooti[2])\ma[{\simpleRooti[2]}] + \cdots + \tau \cross \grading(\rootReflection[{\simpleRooti[1]}]\cdots\rootReflection[{\simpleRooti[n-1]}]\simpleRooti[n])\ma[{\simpleRooti[n]}]\\
& = & \grading(\tau^{-1}\simpleRooti[1])\ma[{\simpleRooti[1]}] + \grading(\tau^{-1}\rootReflection[{\simpleRooti[1]}]\simpleRooti[2])\ma[{\simpleRooti[2]}] + \cdots + \grading(\tau^{-1}\rootReflection[{\simpleRooti[1]}]\cdots\rootReflection[{\simpleRooti[n-1]}]\simpleRooti[n])\ma[{\simpleRooti[n]}] \\
& = & \mca{\weylElt\tau} + \mca{\tau}
\end{eqnarray*}
as desired.
\end{proof}

Proposition \ref{propMcaRel} provides an iterative method for computing $\mca{\weylElt}$.

\begin{corollary}
Let $w \in \imaginaryWeylGroup$ and suppose $w = \rootReflection[{\simpleRooti[n]}] \cdots \rootReflection[{\simpleRooti[1]}]$ is a reduced expression for $w$. Then
\[
\mca{w} = \mca{\rootReflection[{\simpleRooti[1]}]} + \mca[{\rootReflection[{\simpleRooti[1]}] \cross \korbit}]{\rootReflection[{\simpleRooti[2]}]} + \cdots + \mca[{\rootReflection[{\simpleRooti[n-1]}] \cross \cdots \cross~\rootReflection[{\simpleRooti[1]}] \cross \korbit}]{\rootReflection[{\simpleRooti[n]}]}.
\]
\end{corollary}

\begin{proof}
This follows easily by induction and Proposition \ref{propMcaRel}.
\end{proof}

\begin{corollary}
Let $\korbit$ and $\korbit[\psi]$ be two $\maxRealCompact$-orbits in $\fiber$ and suppose $\korbit[\psi] = \tau \cross \korbit$. Then
\[
\left|\left\{\mca{\weylElt} \mid \weylElt \in \imaginaryWeylGroup\right\}\right| = \left|\left\{\mca[{\korbit[\psi]}]{\weylElt} \mid \weylElt \in \imaginaryWeylGroup\right\}\right|.
\]
\end{corollary}

\begin{proof}
Proposition \ref{propMcaRel} implies these sets differ by translation through $\mca{\tau}$.
\end{proof}

Our interest in the elements $\mca{\weylElt}$ is explained by the following remarkable theorem.

\begin{theorem}[\cite{AdamsFokko}] \label {theoremFokko}
Let $\pairInf$ be a pair with corresponding abstract triple $\absTriple$ and let $w_{1}$ and $w_{2}$ be elements in $\imaginaryWeylGroup$. Then 
\[
w_{1} \cross \korbit = w_{2} \cross \korbit \Longleftrightarrow \mca{w_{1}} = \mca{w_{2}}.
\]
In particular, a noncompact imaginary root $\simpleRoot$ is of type II if and only if $\mca{{\rootReflection[\simpleRoot]}} = \ma[\simpleRoot]$ is trivial in the quotient $\negQuotCorootLattice[-\inv] / \posnegQuotCorootLattice[-\inv]$.
\end{theorem}

\begin{remark}
Recall elements in $\negQuotCorootLattice[-\inv] / \posnegQuotCorootLattice[-\inv]$ represent the connected components of a dual torus $\realTorus'$. Theorem \ref{theoremFokko} implies 
\[
w_{1} \cross \korbit = w_{2} \cross \korbit
\]
if and only if the elements $\mca{w_{1}},\mca{w_{2}} \in \realTorus'$ live in the same connected component.
\end{remark}

It follows from Theorem \ref{theoremFokko} that elements in $\fiber$ are parameterized by distinct elements of the form $\mca{w}$, for $w \in \imaginaryWeylGroup$. Formally we can write
\[
\fiber \longleftrightarrow \left\{\mca{w}\cdot\korbit \mid w \in \imaginaryWeylGroup \right\}
\]
where $\mca{w} \cdot \korbit$ is a formal symbol representing the $\maxRealCompact$-orbit $w \cross \korbit$ with $\mca{w} \cdot \korbit = \mca{\tau} \cdot \korbit$ if and only if $\mca{w} = \mca{\tau}$. The action of $\imaginaryWeylGroup$ on $\fiber$ is then
\begin{eqnarray*}
\tau \cross (\mca{w} \cdot \korbit) & = & \mca[{w \cross \korbit}]{\tau} + (\mca{w} \cdot \korbit) \\
& = & (\mca[{w \cross \korbit}]{\tau} + \mca{w}) \cdot \korbit \\
& = & \mca{\tau w} \cdot \korbit
\end{eqnarray*}
by Proposition \ref{propMcaRel}.

\begin{example} \label{egCaThree}
Let $\realGroup = \realSpinGroup{5}{4}$ and suppose $\absDiff \in \absLieAlgebraDual$ is nonsingular. Fix a compact Cartan subgroup $\realTorus \subset \realGroup$ so that $\inv = \identity$ and every $\maxRealCompact$-orbit for $\realTorus$ is in $\fiber$. Fix $\korbit \in \fiber$ and suppose $\absTriple$ is the abstract triple for $\korbit$ with $\grading$
given by the diagram
\[
+~+~\oplus~\oplus.
\]
Since $-\inv = -\identity$ and $\posQuotCorootLattice[-\inv]$ is trivial, Theorem \ref{theoremFokko} implies all noncompact imaginary roots are of type I. We now give a complete description of the action of $\imaginaryWeylGroup = \weylGroup$ on $\fiber$. Let
\begin{eqnarray*}
\simpleRooti[1] & = & \ei[1] - \ei[2] \\
\simpleRooti[2] & = & \ei[2] - \ei[3] \\
\simpleRooti[3] & = & \ei[3] - \ei[4] \\
\simpleRootb & = & \ei[4]
\end{eqnarray*}
denote the simple roots for $\weylGroup$. The following table summarizes the action of $\imaginaryWeylGroup$ on $\fiber$.

\[
\begin{array}{|c|c|c|cccc|}
\hline
orbit & \mca{w} & grading & \simpleRooti[1] & \simpleRooti[2] & \simpleRooti[3] & \simpleRootb \\
\hline
0 & e & +~+~\oplus~\oplus  & 0 & 1 & 0 & 2 \\
1 & \ma[{\simpleRootai[2]}] & +~\oplus~+~\oplus  & 3 & 0 & 4 & 5 \\
2 & \ma[{\simpleRootb}] & +~+~\oplus~\oplus  & 2 & 5 & 2 & 0 \\
3 & \ma[{\simpleRootai[1]}]\ma[{\simpleRootai[2]}] & \oplus~+~+~\oplus  & 1 & 3 & 6 & 7 \\
4 & \ma[{\simpleRootai[2]}]\ma[{\simpleRootai[3]}] & +~\oplus~\oplus~+  & 6 & 4 & 1 & 4 \\
5 & \ma[{\simpleRootai[2]}]\ma[{\simpleRootb}] & +~\oplus~+~\oplus  & 7 & 2 & 8 & 1 \\
6 & \ma[{\simpleRootai[1]}]\ma[{\simpleRootai[2]}]\ma[{\simpleRootai[3]}] & \oplus~+~\oplus~+  & 4 & 9 & 3 & 6 \\
7 & \ma[{\simpleRootai[1]}]\ma[{\simpleRootai[2]}]\ma[{\simpleRootb}] & \oplus~+~+~\oplus  & 5 & 7 & 10 & 3 \\
8 & \ma[{\simpleRootai[2]}]\ma[{\simpleRootai[3]}]\ma[{\simpleRootb}] & +~\oplus~\oplus~+  & 10 & 8 & 5 & 8 \\
9 & \ma[{\simpleRootai[1]}]\ma[{\simpleRootai[3]}] & \oplus~\oplus~+~+  & 9 & 6 & 9 & 9 \\
10 & \ma[{\simpleRootai[1]}]\ma[{\simpleRootai[2]}]\ma[{\simpleRootai[3]}]\ma[{\simpleRootb}] & \oplus~+~\oplus~+  & 8 & 11 & 7 & 10 \\
11 & \ma[{\simpleRootai[1]}]\ma[{\simpleRootai[3]}]\ma[{\simpleRootb}] & \oplus~\oplus+~+~  & 11 & 10 & 11 & 11 \\
\hline
\end{array}
\]

Each row in the table represents a particular $\maxRealCompact$-orbit in $\fiber$. The first column assigns each orbit a number, the second column gives the corresponding element of $\negQuotCorootLattice[-\inv]$, and the third column describes the associated abstract grading. The last four columns give images of cross actions for the simple roots (in terms of orbit numbers from column one).

Note that each abstract grading appears exactly twice and the corresponding elements in $\negQuotCorootLattice[-\inv]$ differ by $\ma[{\simpleRootb}]$ (i.e., by a cross action through a short noncompact imaginary root). Therefore we see the order of the fiber over any abstract triple is two.
\end{example}

\subsection{Fibers for $\realGroup$}
\label{ssGFibers}

Fix a nonsingular element $\absDiff \in \absLieAlgebraDual$ and suppose $\pairInf$ is a pair for $\realGroup$ with corresponding abstract triple $\absTriple$. The previous section outlined a procedure for computing the action of $\imaginaryWeylGroup$ on $\fiber$. In this section we use this procedure to describe the fibers over arbitrary involutions and abstract triples for $\realGroup$. As usual, the picture is unchanged by conjugation and thus it suffices to consider a representative set of involutions in $\weylGroup$. This will essentially give a complete description of the set of $\maxRealCompact$-orbits for $\absDiff$ in $\realGroup$.

Given $\absDiff$ and $\inv$, we can extend these parameters to an abstract triple $\absTriple$ by selecting an imaginary grading $\grading$. Recall $\grading$ specifies the particular real form of $\complexGroup$ (Proposition \ref{propRealForm}) and is determined by its values on the short imaginary roots for $\inv$. In particular, the number of abstract triples beginning with $\inv$ is given by $\displaystyle{\binom{\numImaginaryBits}{\numNcptBits(\grading)}}$ (Definition \ref{defNumCptRoots}). Proposition \ref{propKOGrading} implies there will be at least $\displaystyle{\binom{\numImaginaryBits}{\numNcptBits(\grading)}}$ elements in $\fiber$. 
The following theorem completes the picture by determining the number of elements in the fiber over each abstract triple for $\inv$.

\begin{theorem} \label{theoremKOrbits}
Let $\pairInf$ be a pair for $\realGroup$ and suppose $\absTriple$ is the corresponding abstract triple. Then
\[
\absTripleFiberOrder \in \left\{1,2\right\}.
\]
\end{theorem}

\begin{proof}
The proof is by cases for a representative set of involutions in $\weylGroup$. In each case we determine the number of elements in $\absTripleFiber$ as well as how the elements are related by cross action. It suffices to assume there exists a noncompact imaginary root for $\grading$ (i.e.~$\grading \not\equiv 0$) for otherwise we clearly have 
\[
\absTripleFiberOrder = 1.
\]

\begin{steplist}
\item [Case I]
Let $\inv \in \weylGroup$ be given by the diagram
\[
\diagram: \underset{1}{+}~+~\cdots~\underset{n}{+}
\]
so that $\numImaginaryBits = n$, $\numRealBits = 0$, and all roots are imaginary. Suppose $w \in \imaginaryWeylGroup$ with $w = (00 \cdots 0, \sigma)$. Then $w \cross \grading = \grading$ if and only if $\sigma$ preserves the set of compact roots. Such permutations are given by elements of $\cptImaginaryWeylGroup$ and thus act trivially on $\maxRealCompact$-orbits. 

Suppose now $w = (\bit[1]\bit[2]\cdots\bit[n],1)$ is a product of reflections in short noncompact imaginary roots. Since $\ma[{\ei}] = \ma[{\ei[j]}]$ in $\negQuotCorootLattice[-\inv]$ for all $i$ and $j$, Theorem \ref{theoremFokko} implies the cross action through any two such roots is the same. Moreover, this action is nontrivial since $\posQuotCorootLattice[-\inv] = \left\{0\right\}$. Hence 
\[
\absTripleFiberOrder = 2
\]
with the elements related by a cross action in any short noncompact imaginary root (this is a generalization of Example \ref{egCaThree}). We conclude
\[
\fiberOrder = 2\binom{n}{\numNcptBits(\grading)}.
\]

\item [Case I$'$]
Let $\inv \in \weylGroup$ be given by the diagram
\[
\diagram:~\underset{1}{+}~\cdots \underset{\numImaginaryBits}{+}~-\cdots~\underset{n}{-}
\]
with $\numRealBits = n - \numImaginaryBits \ne 0$. The situation is the same as the previous case, except the elements $\ma[\ei]$ are now in $\posnegQuotCorootLattice[-\inv]$. Hence all short noncompact imaginary roots are of type II and thus
\[
\absTripleFiberOrder = 1.
\]
Therefore
\[
\fiberOrder = \binom{\numImaginaryBits}{\numNcptBits(\grading)}.
\]

\item [Case II]
Let $\inv \in \weylGroup$ be given by the diagram
\[
\diagram:~\underset{1}{(2)}~(1)~\cdots~(k)~\underset{k}{(k-1)}~\underset{k+1}{+}~\cdots \underset{n}{+}
\]
where $k=\numComplexBits$, $\numImaginaryBits = n - k \ne 0$, and $\numRealBits = 0$. Then
\[
\imaginaryWeylGroup \cong W(A_{1})^{k/2} \times W(B_{n-k})
\]
and cross actions for elements of $\imaginaryWeylGroup$ that live in the $W(B_{n-k})$ factor behave as in Case I.

Suppose now $i$ is odd with $1 \le i \le k-1$ and define the roots
\begin{eqnarray*}
\simpleRooti & = & \ei - \ei[i+1] \\
\simpleRootbi & = & \ei + \ei[i+1].
\end{eqnarray*}
Then each root $\simpleRootbi$ is real for $\inv$ and imaginary for $-\inv$ and each root $\simpleRooti$ is noncompact imaginary for $\inv$ and real for $-\inv$. Therefore it remains to determine the type of each $\simpleRooti$. Since $\ma[\simpleRootbi] = \ma[\simpleRootai]\ma[{\ei[j]}]$, we have
\[
\ma[\simpleRooti] = \ma[{\ei[j]}]
\]
in $\negQuotCorootLattice[-\inv] / \posnegQuotCorootLattice[-\inv]$, where $k+1 \le j \le n$. Theorem \ref{theoremFokko} then implies the cross actions with respect to $\rootReflection[\simpleRooti]$ and $\rootReflection[{\ei[j]}]$ are the same (assuming $\ei[j]$ is noncompact). The analysis in Case I shows this action is nontrivial and thus each $\simpleRooti$ is of type I. Since the actions are the same we conclude
\[
\absTripleFiberOrder = 2.
\]
These elements are related by a cross action in any of the roots $\simpleRooti$ or in any short noncompact imaginary root $\ei[j]$. We then have
\[
\fiberOrder = 2\binom{n}{\numNcptBits(\grading)}.
\]

\item [Case III]
Let $\inv \in \weylGroup$ be given by the diagram
\[
\diagram:~\underset{1}{(2)}~(1)~\cdots~(k)~\underset{k}{(k-1)}~\underset{k+1}{-}~\cdots \underset{n}{-}
\]
where $k=\numComplexBits$ and $\numRealBits = n-k \ne 0$ and observe the abstract grading $\grading$ and associated abstract triple are unique for $\inv$. It is easy to verify the corresponding dual torus is connected and Theorem \ref{theoremFokko} implies
\[
\fiberOrder = \absTripleFiberOrder = 1.
\]

\item [Case IV]
Let $\inv \in \weylGroup$ be given by the diagram
\[
\diagram:~\underset{1}{(2)}~(1)~\cdots~(k)~\underset{k}{(k-1)}~\underset{k+1}{+}~\cdots \underset{k+\numImaginaryBits}{+}~-~\cdots~\underset{n}{-}
\]
where $k=\numComplexBits$ and $\numRealBits = n-k-\numImaginaryBits \ne 0$. Combining Case II with Case I$'$ we have
\[
\absTripleFiberOrder = 1
\]
and therefore
\[
\fiberOrder = \binom{n}{\numNcptBits(\grading)}.
\]

\item [Case V]
Let $\inv \in \weylGroup$ be given by the diagram
\[
\diagram:~\underset{1}{(2)}~(1)~\cdots~(n)~\underset{n}{(n-1)}.
\]
Note this is possible only when $n$ is even and $\realGroup$ is split. As in Case III, the abstract grading $\grading$ and associated abstract triple are unique for $\inv$. Suppose $i$ is odd with $1 \le i \le n-1$ and define the roots
\begin{eqnarray*}
\simpleRooti & = & \ei - \ei[i+1] \\
\end{eqnarray*}
as in Case II. The analysis there shows each $\simpleRooti$ is of type I and $\rootReflection[\simpleRooti] \cross \korbit = \rootReflection[{\simpleRooti[j]}] \cross \korbit$ for all $i,j$. Therefore
\[
\fiberOrder = \absTripleFiberOrder = 2.
\]
\end{steplist}
\end{proof}

If we assume $\realGroup \subset \complexGroup$ is \emph{not} the compact form, the results of Theorem \ref{theoremKOrbits} can be summarized as follows.

\begin{corollary} \label{corKOrbits}
Let $\pairInf$ be a pair for $\realGroup$ and suppose $\absTriple$ is the corresponding abstract triple. Then $\absTripleFiberOrder = 2$ if and only if $\numRealBits = 0$. In other words $\absTripleFiberOrder = 2$ if and only if the diagram $\diagram$ contains no $-$ signs. In particular we have
\[
\fiberOrder = 2^{1-\realBitsBit} \binom{\numImaginaryBits}{\numNcptBits(\grading)}.
\]
\end{corollary}

\begin{proof}
Check this for each of the above cases.
\end{proof}

Again suppose $\realGroup \subset \complexGroup$ is not compact. Given a pair $\pairInf$ with corresponding abstract triple $\absTriple$, we can use Theorem \ref{theoremKOrbits} to determine the order of $\cartanWeylGroup{\realTorus}$. We have
\begin{eqnarray*}
\cartanWeylGroup{\realTorus} & \cong & ((A \ltimes \cptImaginaryWeylGroup(\complexLieAlgebra,\complexLieAlgebra[h])) \times \realWeylGroup(\complexLieAlgebra,\complexLieAlgebra[h])) \rtimes \complexWeylGroup(\complexLieAlgebra,\complexLieAlgebra[h])
\end{eqnarray*}
where $A \cong \zTwo^{m}$ is the only unknown group. Let $k = \frac{\numComplexBits}{2}$. Combining Corollary \ref{corKOrbits} with Proposition \ref{propFiberOrder} gives
\begin{eqnarray*}
\fiberOrder = 2^{1-\realBitsBit} \binom{\numImaginaryBits}{\numNcptBits(\grading)} & = & \left| \frac{\invWeylGroup}{\cartanWeylGroup{\realTorus}} \right| \\
& = & \left| \frac{(\imaginaryWeylGroup \times \realWeylGroup) \rtimes \complexWeylGroup}{((A \ltimes \cptImaginaryWeylGroup) \times \realWeylGroup) \rtimes \complexWeylGroup} \right| \\
& = & \frac{\left| \imaginaryWeylGroup \right|}{\left|A \ltimes \cptImaginaryWeylGroup\right|} \\
& = & \frac{2^{k} 2^{\NcptBitsBit(\grading)}}{\left|A\right|} \binom{\numImaginaryBits}{\numNcptBits(\grading)}
\end{eqnarray*}
so that
\begin{eqnarray*}
\left| A \right| & = & \frac{2^{k} 2^{\NcptBitsBit(\grading)}}{2^{1-\realBitsBit}} \\
& = & 2^{k + \NcptBitsBit(\grading) + \realBitsBit - 1}.
\end{eqnarray*}


\subsection{Fibers for $\realGroupCover$}
\label{ssNLFibers}
Let $\absDiff \in \absLieAlgebraDual$ be a half-integral infinitesimal character. Recall the set $\hcBasisGenSpin{p}{q}$ is parameterized by the $\maxRealCompactCover$-conjugacy classes of genuine triples $\genTriple$, where $\realTorusCover$ is a $\cinv$-stable Cartan subgroup of $\realGroupCover = \realSpinGroupCover{p}{q}$, $\diff \in \complexLieAlgebra[h]^{*}$ such that $\diff$ and $\absDiff$ define the same infinitesimal character, and $\genTorusChar$ is a genuine representation of $\realTorusCover$. Theorem \ref{theoremNumGenReps} reduces the problem to understanding the $\maxRealCompactCover$-conjugacy classes of genuine pairs $\genPairInf$ whose corresponding abstract triples are supportable (Definition \ref{defSuppTriple}). A $\maxRealCompactCover$-conjugacy class of genuine pairs $\genPairInf$ will be called a \emph{$\maxRealCompactCover$-orbit} for $\absDiff$. We have seen the $\maxRealCompactCover$-orbits for $\absDiff$ correspond naturally to the $\maxRealCompact$-orbits for $\absDiff$, and these were parameterized in Theorem \ref{theoremKOrbits}.

\begin{definition}
\label{defGenFiber}
Let $\genPairInf$ be a genuine pair with corresponding abstract triple $\absTriple$. The \emph{genuine fiber over} $\inv$ (denoted $\genFiber$) is the set of $\maxRealCompact$-orbits in $\fiber$ whose corresponding abstract triples are supportable. In particular $\numCorGenTriplesInf \ne 0$ if and only if $\korbit \in \genFiber$.
\end{definition}

Although the order of $\fiber$ can be arbitrarily large (Corollary \ref{corKOrbits}), Theorem \ref{theoremGenFiber} belows shows the order of $\genFiber$ is always small. Before stating the theorem we need the following definitions.

\begin{definition}
\label{defIntBits}
Let $\inv \in \weylGroup$ be an involution and suppose $\absDiff \in \absLieAlgebraDual$ is a fixed half-integral infinitesimal character. Define
\begin{eqnarray*}
\numIntBits & = & \left|\left\{i \mid \bit \in \imaginaryBits[\inv] \text{ and } (\ei,\absDiff) \in \mathbb{Z}\right\}\right| \\
\numHalfIntBits & = & \left|\left\{i \mid \bit \in \imaginaryBits[\inv] \text{ and } (\ei,\absDiff) \in \mathbb{Z}+\frac{1}{2}\right\}\right|
\end{eqnarray*}
(Definition \ref{defInvParams}) and observe $\numImaginaryBits = \numIntBits + \numHalfIntBits$ since $\absDiff$ is half-integral.
\end{definition}

 In terms of diagrams, $\numIntBits$ is the number of $+$ signs whose coordinate is integral with respect to $\absDiff$ and $\numHalfIntBits$ is the number of $+$ signs whose coordinate is half-integral. Note the abstract triple $\absTriple$ is supportable only if $\numNcptBits(\grading) = \numIntBits$ or $\numNcptBits(\grading) = \numHalfIntBits$ (Proposition \ref{propSuppRep}).

\begin{definition}
\label{defInvSymmetric}
In the setting of Definition \ref{defIntBits}, we say $\inv$ is symmetric about $\absDiff$ if $\numImaginaryBits > 0$ and
\[
\numIntBits = \numHalfIntBits.
\]
This is clearly possible only when $\numImaginaryBits$ is even. For convenience we also define
\[
\symBit = \left\{\begin{array}{rl}
1 & \inv \text{ is symmetric about } \absDiff \\
0 & \text{ otherwise}
\end{array}\right..
\]
\end{definition}

As we'll see in the following theorem, the size of the genuine fiber $\genFiber$ increases when the symmetry condition of Definition \ref{defInvSymmetric} is satisfied.
 
\begin{theorem} \label{theoremGenFiber}
Let $\absDiff \in \absLieAlgebraDual$ be a half-integral infinitesimal character and suppose $\genPairInf$ is a genuine pair for $\absDiff$. Let $\inv \in \weylGroup$ be the corresponding abstract involution and assume there exists a supportable abstract triple beginning with $\inv$. Then
\[
\genFiberOrder \in \left\{1,2,4\right\}.
\]
\end{theorem}

\begin{proof}
We first determine the number of gradings $\grading$ that make the abstract triple $\absTriple$ supportable (by assumption there exists at least one). We then use Theorem \ref{theoremKOrbits} to determine the order of the fiber over each abstract triple. As usual we assume $\realGroupCover$ is \emph{not} compact and proceed by cases for the diagram of $\inv$.

\begin{steplist}
\item [Case I]
Let $\inv$ be given by the diagram
\[
\diagram: \underset{1}{+}~+~\cdots~\underset{n}{+}
\]
so that $\numImaginaryBits = n$. There are two possible gradings leading to supportable abstract triples if $\inv$ is symmetric about $\absDiff$ and only one otherwise. In either case, Theorem \ref{theoremKOrbits} implies the fiber over each abstract triple is two. Therefore we have
\[
\genFiberOrder = \left\{\begin{array}{rl}
4 & \inv \text{ is symmetric about } \absDiff \\
2 & \inv \text{ is not symmetric about } \absDiff
\end{array}
\right..
\]
Note $\inv$ can be symmetric about $\absDiff$ only if $n$ is even.

\item [Case I$'$]
Let $\inv$ be given by the diagram
\[
\diagram:~\underset{1}{+}~\cdots \underset{\numImaginaryBits}{+}~-\cdots~\underset{n}{-}
\]
with $\numRealBits = n - \numImaginaryBits \ne 0$. Combining Case I above with Case I$'$ of Theorem \ref{theoremKOrbits} gives
\[
\genFiberOrder = \left\{\begin{array}{rl}
2 & \inv \text{ is symmetric about } \absDiff \\
1 & \inv \text{ is not symmetric about } \absDiff
\end{array}
\right..
\]

\item [Case II]
Let $\inv$ be given by the diagram
\[
\diagram:~\underset{1}{(2)}~(1)~\cdots~(k)~\underset{k}{(k-1)}~\underset{k+1}{+}~\cdots \underset{n}{+}
\]
where $k=\numComplexBits$, $\numImaginaryBits = n - k \ne 0$, and $\numRealBits = 0$. Combining Case I above with Case II of Theorem \ref{theoremKOrbits} gives
\[
\genFiberOrder = \left\{\begin{array}{rl}
4 & \inv \text{ is symmetric about } \absDiff \\
2 & \inv \text{ is not symmetric about } \absDiff
\end{array}
\right..
\]

\item [Case III]
Let $\inv$ be given by the diagram
\[
\diagram:~\underset{1}{(2)}~(1)~\cdots~(k)~\underset{k}{(k-1)}~\underset{k+1}{-}~\cdots \underset{n}{-}
\]
where $k=\numComplexBits$ and $\numRealBits = n-k \ne 0$. There is only one possible abstract triple (supportable by hypothesis) whose fiber is a singleton by Theorem \ref{theoremKOrbits}. Therefore
\[
\genFiberOrder = 1.
\]

\item [Case IV]
Let $\inv$ be given by the diagram
\[
\diagram:~\underset{1}{(2)}~(1)~\cdots~(k)~\underset{k}{(k-1)}~\underset{k+1}{+}~\cdots \underset{k+\numImaginaryBits}{+}~-~\cdots~\underset{n}{-}
\]
where $k=\numComplexBits$ and $\numRealBits = n-k-\numImaginaryBits \ne 0$. Combining Case I above with Case IV of Theorem \ref{theoremKOrbits} gives
\[
\genFiberOrder = \left\{\begin{array}{rl}
2 & \inv \text{ is symmetric about } \absDiff \\
1 & \inv \text{ is not symmetric about } \absDiff
\end{array}
\right..
\]

\item [Case V]
Let $\inv$ be given by the diagram
\[
\diagram:~\underset{1}{(2)}~(1)~\cdots~(n)~\underset{n}{(n-1)}.
\]
Note this is possible only when $n$ is even and $\realGroup$ is split. As in Case III, there is a unique abstract triple beginning with $\inv$. Theorem \ref{theoremKOrbits} implies its fiber has order two. Therefore we have
\[
\genFiberOrder = 2.
\]
\end{steplist}
\end{proof}

If we assume $\realGroupCover$ is \emph{not} compact, the results of Theorem \ref{theoremGenFiber} can be summarized as follows.

\begin{corollary} \label{corGenFiber}
Let $\genPairInf$ be a genuine pair for $\realGroupCover$ with corresponding abstract involution $\inv \in \weylGroup$. If $\genFiberOrder \ne 0$ we have
\[
\genFiberOrder = 2^{1-\realBitsBit}2^{\symBit}.
\]
\end{corollary}

\begin{proof}
Check this for each of the cases above.
\end{proof}

Fix a half-integral infinitesimal character $\absDiff \in \absLieAlgebraDual$ and recall we have constructed a (well-defined) abstract triple $\absTriple$ for each element in $\hcBasisGenSpin{p}{q}$. Let $\hcBasisGenInvSpin{p}{q}$ denote the subset of $\hcBasisGenSpin{p}{q}$ whose elements have abstract triples beginning with $\inv$. We conclude this section with a description of the size of $\hcBasisGenInvSpin{p}{q}$, which also turns out to be small.

\begin{theorem} \label{theoremRepFiber}
Let $\absDiff \in \absLieAlgebraDual$ be a half-integral infinitesimal character and let $\inv \in \weylGroup$ be an involution. If $\genFiberOrder \ne 0$, then 
\begin{eqnarray*}
\hcBasisGenInvOrderSpin{p}{q} & = & 2^{1-\imaginaryBitsBit}2^{\realBitsBit(1-\realBitsParityBit)}2^{1-\realBitsBit}2^{\symBit} \\
& = & 2^{2 +\symBit - \imaginaryBitsBit - \realBitsBit\realBitsParityBit}.
\end{eqnarray*}
In particular
\[
\hcBasisGenInvOrderSpin{p}{q} \in \left\{1,2,4\right\}.
\]
\end{theorem}

\begin{proof}
Combine Corollary \ref{corNumGenRepsFormula} with Corollary \ref{corGenFiber}. For the last statement recall $\symBit = 1$ only if $\imaginaryBitsBit = 1$ (Definition \ref{defInvSymmetric}).
\end{proof}

\section{Central Characters}
\label{secCentralMa}

Throughout this section let $\realGroup = \realSpinGroup{p}{q}$ with $p > q$, $p+q = 2n+1$, and recall $\realGroupCover$ denotes the (connected) nonalgebraic double cover of $\realGroup$. Theorem \ref{theoremRepFiber} implies the elements of $\hcCatGen$ can be naturally partitioned into sets of size at most four. In this section we use the notion of central characters to reduce this by a factor of two.

\subsection{Abstract Bigradings}
\label{ssAbsBg}
Let $\absDiff \in \absLieAlgebraDual$ be a half-integral infinitesimal character and suppose $\genTripleInf$ is a genuine triple for $\absDiff$. Recall $\genTorusChar$ is an irreducible genuine representation of $\realTorusCover$ whose differential is compatible with $\diff$ (Definition \ref{defTriple}). Since the Cartan subgroup $\realTorusCover$ may not be abelian, the representation $\genTorusChar$ is not necessarily one-dimensional (Proposition \ref{propGenRepBij}). In this section we characterize the action of $\genTorusChar$ on certain finite order elements in $\realTorusCover$.

If $\simpleRoota$ and $\simpleRootb$ are short roots in $\realRoots[\cinv](\complexLieAlgebra,\complexLieAlgebra[h])$, Proposition \ref{propMShort} implies $\mac[\simpleRoota] = \pm \mac[\simpleRootb]$. The following proposition shows we always have equality.

\begin{proposition}
\label{propSameMacReal}
Let $\simpleRoota$ and $\simpleRootb$ be short orthogonal roots in $\realRoots[\cinv](\complexLieAlgebra,\complexLieAlgebra[h])$. Then $\mac[\simpleRootb] =  \mac[\simpleRoota]$.
\end{proposition}

\begin{proof}
Let $\simpleRootc = \simpleRoota - \simpleRootb$ so that
\begin{eqnarray*}
\rootReflection[\simpleRootc](\simpleRoota) & = & \simpleRoota - (\simpleRoota,\simpleRootc)\simpleRootc \\
& = & \simpleRoota - \simpleRootc \\
& = & \simpleRootb.
\end{eqnarray*}
Choose a root vector $\rootVector[\simpleRoota] \in \rootSpace[\simpleRoota]$ according to Lemma \ref{lemXaSign}. Then
\begin{eqnarray*}
\mac[\simpleRootb] & = & \expgc(\pi\Ad{\oac[\simpleRootc]}\zLa) \\
& = & \oac[\simpleRootc]\expgc(\pi\zLa)\oac[\simpleRootc]^{-1} \\
& = & \oac[\simpleRootc]\mac[\simpleRoota]\oac[\simpleRootc]^{-1} \\
& = & \mac[\simpleRoota]
\end{eqnarray*}
by Proposition \ref{propMShort}.
\end{proof}

A similar result is true for imaginary roots whenever they are of the same type.

\begin{proposition}
\label{propSameMacNci}
Let $\grading$ be the imaginary grading for $\rootSystem(\complexLieAlgebra, \complexLieAlgebra[h])$ and suppose $\simpleRoota$ and $\simpleRootb$ are short orthogonal imaginary roots. Then $\mac[\simpleRoota] = \mac[\simpleRootb]$ if and only if $\grading(\simpleRoota) = \grading(\simpleRootb)$.
\end{proposition}

\begin{proof}
Let $\simpleRootc = \simpleRoota - \simpleRootb$ and suppose $\grading(\simpleRoota) = \grading(\simpleRootb)$. Then Proposition \ref{propGrading} implies $\simpleRootc$ is compact and we can find an element $\oac[\simpleRootc] \in \maxRealCompactCover$ for which
\[
\Ad{\oac[\simpleRootc]}\coroot[\simpleRoota] = \coroot[\simpleRootb].
\]
We now proceed as in the proof of Proposition \ref{propSameMacReal}. Conversely suppose $\grading(\simpleRoota) \ne \grading(\simpleRootb)$ so that $\simpleRootc$ is noncompact.  On the level of coroots we have $\coroot[\simpleRoota] = \coroot[\simpleRootb] + 2\coroot[\simpleRootc]$ and
\begin{eqnarray*}
\mac[\simpleRoota] & = & -\expgc(-\pi i\coroot[\simpleRoota]) \\
& = & \expgc(-\pi i(\coroot[\simpleRootb] + 2\coroot[\simpleRootc])) \\
& = & \expgc(-\pi i\coroot[\simpleRootb])\expgc(-2\pi i\coroot[\simpleRootc]) \\
& = & \mac[\simpleRootb]\mac[\simpleRootc]^{2} \\
& = & -\mac[\simpleRootb]
\end{eqnarray*}
as desired.
\end{proof}

We now construct a grading on the \emph{real} roots of $\rootSystem(\complexLieAlgebra, \complexLieAlgebra[h])$. We begin with the following lemma.

\begin{lemma}[\cite{RT3}, Chapter 5]
\label{lemGma}
Let $\simpleRoot \in \realRoots[\cinv](\complexLieAlgebra, \complexLieAlgebra[h])$ be a real root and choose a corresponding $\mac \in \realTorusCover$. If $\simpleRoot$ is long we have 
\begin{eqnarray*}
\genTorusChar(\mac^{2}) & = & -\identity
\end{eqnarray*}
and $\genTorusChar(\mac)$ has eigenvalues $\pm i$ occurring with equal multiplicity. If $\simpleRoot$ is short we have
\begin{eqnarray*}
\genTorusChar(\mac) & = & \pm\identity.
\end{eqnarray*}
\end{lemma}

The data of Lemma \ref{lemGma} are conveniently packaged in the following definition.

\begin{definition}
\label{defPC}
Following \cite{RT3}, we say a long root $\simpleRoot \in \realRoots[\cinv](\complexLieAlgebra, \complexLieAlgebra[h])$ \emph{satisfies the parity condition} if it is strictly half-integral with respect to $\diff$ (i.e., $(\diff,\rootCheck) \in \mathbb{Z} + \frac{1}{2}$). We say a short root $\simpleRoot \in \realRoots[\cinv](\complexLieAlgebra, \complexLieAlgebra[h])$ satisfies the parity condition if
\begin{eqnarray*}
\genTorusChar(\mac) & = & (-1)(-1)^{(\diff, \rootCheck)}\cdot\identity
\end{eqnarray*}
(\cite{VGr}, Section 8.3). Recall the element $\mac$ is well-defined for short roots and we always have $(\diff, \rootCheck) \in \mathbb{Z}.$ Lemma \ref{lemGma} implies this definition makes sense and gives a well-defined map
\[
\realGrading : \realRoots[\cinv](\complexLieAlgebra,\complexLieAlgebra[h]) \to \mathbb{Z}_{2}
\]
where
\begin{eqnarray*}
\realGrading(\simpleRoot) & = & \left\{\begin{array}{rl}
0 & \simpleRoot \text{ does not satisfy the parity condition} \\
1 & \simpleRoot \text{ satisfies the parity condition}
\end{array}
\right..
\end{eqnarray*}
\end{definition}

The following proposition is the formal analog of Proposition \ref{propGrading} for real roots.

\begin{proposition}
\label{propRealGrading}
The map $\realGrading$ defines a grading (Definition \ref{defGrading}) on $\realRoots[\cinv](\complexLieAlgebra,\complexLieAlgebra[h])$.
\end{proposition}

\begin{proof}
It remains to verify the properties of Definition \ref{defGrading}. This follows easily from Proposition \ref{propSameMacReal} and the fact that the parity condition for a short root $\simpleRoot$ is determined by the integer $(\diff, \rootCheck)$.
\end{proof}

\begin{definition}
\label{defAbsBigrading}
In Section \ref{ssInvAndGradings} we defined a grading $\grading$ on the imaginary root system $\imaginaryRoots[\cinv](\complexLieAlgebra, \complexLieAlgebra[h])$ for any $\cinv$-stable Cartan subalgebra $\complexLieAlgebra[h] \subset \complexLieAlgebra$. Given a genuine triple $\genTripleInf$, Proposition \ref{propRealGrading} defines a grading $\realGrading$ on the real root system $\realRoots[\cinv](\complexLieAlgebra, \complexLieAlgebra[h])$ with the same formal properties as $\grading$. Applying the conjugation map $\absConjTwoInv{\absDiff}$ (Definition \ref{defAbsConjMap}) to $\left\{\genPair, \grading, \realGrading\right\}$ gives an \emph{abstract bigrading} $\absBg$ for the abstract root system $\rootSystem=\rootSystem(\complexLieAlgebra, \absLieAlgebra)$. This extends the notion of an abstract triple from Definition \ref{defAbsTriple}.
\end{definition}

It is easy to verify conjugating $\genTripleInf$ by an element of $\maxRealCompactCover$ does not change the abstract bigrading. In particular, abstract bigradings are defined on the level of genuine parameters $\hcBasisGenSpin{p}{q}$ for $\realGroupCover$. Note that an abstract bigrading depends (potentially) on the full representative $\genTripleInf$ and not just the representative pair $\genPairInf$. The exact nature of this dependence is determined in the next section.

\subsection{Central Character}
\label{ssCentralCharacter}

Let $\absDiff \in \absLieAlgebraDual$ be a half-integral infinitesimal character and fix a Cartan subgroup $\realTorusCover \subset \realGroupCover = \realSpinGroupCover{p}{q}$. 

\begin{definition}
\label{defCC}
The \emph{central character} of a genuine triple $\genTripleInf$ is the genuine representation of Z($\realGroupCover$) given by restricting $\genTorusChar$ to $\text{Z}(\realGroupCover) \subset \realTorusCover$.
\end{definition}

\begin{remark}
\label{remTwoCC}
The central character will turn out to be an important invariant of genuine triples. Since $\genTorusChar$ is assumed to be genuine, we automatically have $\genTorusChar(-1) = -\identity$. Therefore the central character of $\genTripleInf$ is determined by the action of $\genTorusChar$ on a single nontrivial element in $\text{Z}(\realGroupCover)$. Since $\left|\text{Z}(\realGroupCover)\right| = 4$, there are only two such possibilities.
\end{remark}

\begin{remark}
\label{remCompareCC}
Given a genuine triple $\genTripleInf$, Proposition \ref{propGenRepBij} implies 
\[
\genTorusChar|_{\text{Z}(\realTorusCover)} = m\chi
\]
where $m = \left|\realTorusCover/Z(\realTorusCover)\right|^{\frac{1}{2}}$ and $\chi$ is a genuine character of Z($\realTorusCover$). In particular, central characters are not always one-dimensional representations of Z($\realGroupCover$). This makes it technically incorrect to compare central characters for distinct genuine triples unless we happen to know the dimensions of their irreducible representations are equal. We remedy this by comparing the associated characters $\chi|_{\text{Z}(\realGroupCover)}$ from Proposition \ref{propGenRepBij}. Unless otherwise stated, this convention is in effect whenever we compare central characters for different genuine triples.
\end{remark}

\begin{remark}
The central character is clearly well defined on the level of genuine parameters $\hcBasisGenSpin{p}{q}$ for $\realGroupCover$.
\end{remark}

Fix a genuine triple $\genTripleInf$ and let $\realGrading$ and $\grading$ be the corresponding real and imaginary gradings. In most cases, the central character of $\genTripleInf$ is determined by either $\realGrading$ or $\grading$.

\begin{proposition}
\label{propCCReal}
In the above setting, suppose $\simpleRoot \in \realRoots[\cinv](\complexLieAlgebra, \complexLieAlgebra[h])$ is a short real root. Then $\mac$ is well-defined and central in $\realGroupCover$ and we have
\begin{eqnarray*}
\genTorusChar(\mac) & = & \left\{\begin{array}{ll}
(-1)^{1-\realGrading(\simpleRoot)}\cdot\identity & (\diff, \simpleRoot) \in \mathbb{Z} + \frac{1}{2} \\
(-1)^{\realGrading(\simpleRoot)}\cdot\identity & (\diff, \simpleRoot) \in \mathbb{Z}
\end{array}
\right..
\end{eqnarray*}
\end{proposition}

\begin{proof}
This follows from Definition \ref{defPC} directly.
\end{proof}

\begin{proposition}
\label{propCCNci}
In the above setting, suppose $\simpleRoot \in \imaginaryRoots[\cinv](\complexLieAlgebra, \complexLieAlgebra[h])$ is short. Then $\mac$ is well-defined and central in $\realGroupCover$ and we have
\begin{eqnarray*}
\genTorusChar(\mac) & = & \left\{\begin{array}{rl}
1\cdot\identity & (\diff, \simpleRoot) \in \mathbb{Z} + \frac{1}{2} \\
-1\cdot\identity & (\diff, \simpleRoot) \in \mathbb{Z}
\end{array}
\right..
\end{eqnarray*}
\end{proposition}

\begin{proof}
Recall $\mac = \expgc(-\pi i\coroot) = \expgc(\pi i\coroot)$ is a central element in $\realGroupCover$ with $i\coroot \in \realLieAlgebra$. It follows $\mac \in \realTorusCoverId$ (see also Proposition \ref{propCompH}) and it suffices to calculate $\text{d}\torusChar(-\pi i\coroot)$. From the identity $\simpleRoot(\coroot) = 2$ and the definition of $\text{d}\genTorusChar$ in Section \ref{ssHCModIntro} we have
\begin{eqnarray*}
\genTorusChar(\mac) & = & \genTorusChar(\expgc(-\pi i\coroot)) \\
& = & e^{\text{d}\torusChar(-\pi i\coroot)} \cdot \identity \\
& = & e^{-2\pi i\left((\diff,\simpleRoot) + (\halfSumIm, \simpleRoot) - (2\halfSumImCpt, \simpleRoot)\right)} \cdot \identity.
\end{eqnarray*}
Now it is easy to check $(\halfSumIm, \simpleRoot) \in \mathbb{Z} + \frac{1}{2}$ and $(2\halfSumImCpt, \simpleRoot) \in \mathbb{Z}$ so that 
\begin{eqnarray*}
\genTorusChar(\mac) & = & -e^{-2\pi i\cdot(\diff,\simpleRoot)} \cdot \identity
\end{eqnarray*}
and the result follows.
\end{proof}

\begin{remark}
\label{remCC}
If $\rootSystem(\complexLieAlgebra, \complexLieAlgebra[h])$ contains a short root that is not complex, we can use Proposition \ref{propCCReal} or Proposition \ref{propCCNci} to determine the central character of $\genTripleInf$. If this root is imaginary, the central character is determined by the genuine pair $\genPairInf$ and all genuine triples extending $\genPairInf$ have the same central character.
\end{remark}

It is interesting to observe there are (at most) two possibilities for the imaginary grading $\grading$. In particular, if $\simpleRoota, \simpleRootb \in \imaginaryRoots(\complexLieAlgebra, \complexLieAlgebra[h])$ are short then 
\[
\grading(\simpleRoota) = \grading(\simpleRootb) \iff (\diff, \simpleRoot) \equiv (\diff, \simpleRootb) \text{ mod } \mathbb{Z}.
\]
Since the structure of Definition \ref{defPC} is formally the same, a similar result holds for the real grading $\realGrading$ as well. The following theorem describes the relationship between $\realGrading$ and $\grading$ in terms of these possibilities.

\begin{theorem}
\label{theoremCompareGradings}
Let $\genTripleInf$ be a genuine triple with corresponding real and imaginary gradings $\realGrading$ and $\grading$. Then $\realGrading$ and $\grading$ are gradings of the opposite kind. In particular, if $\simpleRoota \in \realRoots[\cinv](\complexLieAlgebra, \complexLieAlgebra[h])$ and $\simpleRootb \in \imaginaryRoots[\cinv](\complexLieAlgebra, \complexLieAlgebra[h])$ are short 
\[
\realGrading(\simpleRoota) = \grading(\simpleRootb) \iff (\diff, \simpleRoota) \not\equiv (\diff, \simpleRootb) \text{ mod } \mathbb{Z}.
\]
\end{theorem}

\begin{proof} 
Suppose $\realGrading(\simpleRoota) = 1$. Proposition \ref{propCCReal} and Proposition \ref{propCCNci} give
\begin{eqnarray*}
\genTorusChar(\mac) & = & \left\{\begin{array}{rl}
1\cdot\identity & (\diff, \simpleRoot) \in \mathbb{Z} + \frac{1}{2} \\
-1\cdot\identity & (\diff, \simpleRoot) \in \mathbb{Z}
\end{array}
\right.\\
\genTorusChar(\mac[\simpleRootb]) & = & \left\{\begin{array}{rl}
1\cdot\identity & (\diff, \simpleRootb) \in \mathbb{Z} + \frac{1}{2} \\
-1\cdot\identity & (\diff, \simpleRootb) \in \mathbb{Z}
\end{array}
\right.
\end{eqnarray*}
and it remains to show
\[
\mac[\simpleRoota] = \mac[\simpleRootb] \iff \grading(\simpleRootb) = 0.
\]

To begin let
\begin{eqnarray*}
\simpleRootai[1],\ldots,\simpleRootai[k_{1}] & \in & \realRoots[\cinv](\complexLieAlgebra,\complexLieAlgebra[h]) \\
\simpleRootbi[1],\ldots,\simpleRootbi[k_{2}] & \in & \imaginaryRoots[\cinv](\complexLieAlgebra,\complexLieAlgebra[h]) \\
\simpleRootci[1],\ldots,\simpleRootci[k_{3}] & \in & \imaginaryRoots[\cinv](\complexLieAlgebra,\complexLieAlgebra[h])
\end{eqnarray*}
denote the short real, noncompact, and compact roots in $\posRootSystem(\complexLieAlgebra, \complexLieAlgebra[h])$ respectively. Then Propositions \ref{propSameMacReal} and \ref{propSameMacNci} imply
\begin{eqnarray*}
\mac[{\simpleRootai}] = \mac[{\simpleRootai[j]}] & \text{for} & 1 \le i,j \le k_{1} \\
\mac[{\simpleRootbi}] = \mac[{\simpleRootbi[j]}] & \text{for} & 1 \le i,j \le k_{2} \\
\mac[{\simpleRootci}] = \mac[{\simpleRootci[j]}] & \text{for} & 1 \le i,j \le k_{3}.
\end{eqnarray*}
If $\ctOp[{\simpleRootai}]$ denotes the Cayley transform with respect to $\simpleRootai$, then $\mac[{\simpleRootai}] = \mac[{\ctOp[\simpleRootai](\simpleRootai)}]$ by Proposition \ref{propSameMa}. However
\[
\mac[{\simpleRootai}] = \mac[{\ctOp[\simpleRootai](\simpleRootai)}] = \mac[{\ctOp[\simpleRootai](\simpleRootci[j])}] = \mac[{\simpleRootci[j]}]
\]
by Proposition \ref{propSameMacNci} and Lemma \ref{lemCtCoroot} (see also \cite{IC4}, Lemma 5.1). But then
\[
\mac[{\simpleRootbi}] = -\mac[{\simpleRootci[j]}]
\]
by Proposition \ref{propSameMacNci} and the result follows.
\end{proof}

\subsection{Numerical Duality in Even Rank}
\label{ssNumDuality}

Let $\absDiff \in \absLieAlgebraDual$ be a half-integral infinitesimal character. In this section we attempt to extend the map $\weylGroupInvMap$ (Definition \ref{defWeylGroupInvMap}) to the level of genuine parameters for the nonlinear groups $\realSpinGroupCover{p}{q}$. To simplify things, we first fix a genuine central character (Definition \ref{defCC}).

\begin{definition}
\label{defGenParamsWithCC}
Let $\realGroupCover = \realSpinGroupCover{p}{q}$ and suppose $\chi$ is a genuine character of $\text{Z}(\realGroupCover)$. Let $\hcBasisGenCentSpin{p}{q} \subset \hcBasisGenSpin{p}{q}$ denote the collection of genuine parameters whose central character is the same as $\chi$ (Remark \ref{remCompareCC}).
\end{definition}

If $\realGroupCover$ is \emph{not} compact, we have the following analog of Theorem \ref{theoremRepFiber}.

\begin{corollary} \label{corRepFiberCent}
Let $\inv \in \weylGroup(\complexLieAlgebra, \absLieAlgebra)$ be an involution with $\genFiberOrder \ne 0$ (Definition \ref{defGenFiber}) and suppose $\chi$ is a genuine character of $\text{Z}(\realGroupCover)$. Then if $\hcBasisGenInvOrderCent \ne 0$ we have
\begin{eqnarray*}
\hcBasisGenInvOrderCentSpin{p}{q} & = & 2^{1-\realBitsBit}2^{\realBitsBit(1-\realBitsParityBit)} \\
& = & 2^{1  - \realBitsBit\realBitsParityBit}.
\end{eqnarray*}
In particular
\[
\hcBasisGenInvOrderCentSpin{p}{q} \in \left\{1,2\right\}.
\]
\end{corollary}

\begin{proof}
Fixing the genuine character $\chi$ eliminates the $2^{\symBit}$ factor in Corollary \ref{corGenFiber} (Proposition \ref{propCCNci}) and the $2^{1-\imaginaryBitsBit}$ factor of Corollary \ref{corNumGenRepsFormula}.
\end{proof}

Let $\realGroupCover = \realSpinGroupCover{p}{q}$ and suppose $\realTorusCover \subset \realGroupCover$ is a Cartan subgroup. Fix a genuine triple $\genTripleInf$ with corresponding abstract bigrading $\absBg$ (Definition \ref{defAbsBigrading}). The following definition is a dual version of Definition \ref{defNumCptRoots}.

\begin{definition}
\label{defNumParRoots}
Set
\begin{eqnarray*}
\numNonParBits(\realGrading) & = & \left|\left\{\simpleRoot \in \posRootSystem(\complexLieAlgebra, \absLieAlgebra) \mid \simpleRoot \text{ short and }\realGrading(\simpleRoot) = 0\right\}\right| \\
\numParBits(\realGrading) & = & \left|\left\{\simpleRoot \in \posRootSystem(\complexLieAlgebra, \absLieAlgebra) \mid \simpleRoot \text{ short and }\realGrading(\simpleRoot) = 1\right\}\right|
\end{eqnarray*}
and observe $\numRealBits = \numNonParBits(\realGrading) + \numParBits(\realGrading)$ (Definition \ref{defInvParams}).
\end{definition}

\begin{definition}
\label{defDualGroup}
Let $\realGroupCover = \realSpinGroupCover{p}{q}$ with $p+q=2n+1$ and $p>q$. Suppose $\genTripleInf$ is a genuine triple for $\realGroupCover$ with corresponding abstract bigrading $\absBg$. The \emph{dual bigrading} for $\absBg$ is
\begin{eqnarray*}
\weylGroupInvMap\absBg & = & \absBgDualInv[\weylGroupInvMap(\inv)] \\
& = & \absBgDual.
\end{eqnarray*}
The \emph{dual group} for $\realGroupCover$ \emph{and} $\genTripleInf$ is $\realGroupCoverDual = \realSpinGroupCover{p^{\vee}}{q^{\vee}}$, where
\begin{eqnarray*}
p^{\vee} & = & 2\cdot\text{max}(\numParBits(\realGrading),\numNonParBits(\realGrading)) + \numImaginaryBits + \numComplexBits + \left\{\begin{array}{rl} 1 & \numParBits(\realGrading) \le \numNonParBits(\realGrading) \\ 0 & \text{else}\end{array}\right. \\
q^{\vee} & = & 2\cdot\text{min}(\numParBits(\realGrading),\numNonParBits(\realGrading)) + \numImaginaryBits + \numComplexBits + \left\{\begin{array}{rl} 1 & \numParBits(\realGrading) > \numNonParBits(\realGrading) \\ 0 & \text{else}\end{array}\right..
\end{eqnarray*}
\end{definition}

\begin{remark}
In Definition \ref{defDualGroup}, the dual group $\realGroupCoverDual$ depends on both the group $\realGroupCover$ and the genuine triple $\genTripleInf$. In particular, the infinitesimal and central characters of $\genTripleInf$ play an important role in determining $\realGroupCoverDual$.
\end{remark}

\begin{proposition}
\label{propSuppRepsDual}
In the setting of Definition \ref{defDualGroup}, there exists a genuine triple $\genTripleInfDual$ for $\realGroupCoverDual$ whose corresponding abstract bigrading is $\absBgDual$.
\end{proposition}

\begin{proof}
Since $\realGrading$ is a grading of $\realRoots(\complexLieAlgebra, \absLieAlgebra)$ satisfying the same formal properties as $\grading$ (Proposition \ref{propRealGrading}), the abstract triple $\absTripleDual$ is supportable (Definition \ref{defSuppTriple}). The existence of $\genTripleInfDual$ now follows from Proposition \ref{propIffSupp}. The fact that $\genTripleInfDual$ is a genuine triple for $\realGroupCoverDual$ follows from Proposition \ref{propRealForm} applied to $\absPairDual$.
\end{proof}

\begin{proposition}
\label{propCCDual}
Let $\genTripleInf$ be a genuine triple for $\realGroupCover$ with corresponding abstract bigrading $\absBg$ and choose $\genTripleInfDual$ for $\realGroupCoverDual$ as in Proposition \ref{propSuppRepsDual}. Suppose $\chi$ is the central character of $\genTripleInf$ and write $\chi^{\vee}$ for the opposite (genuine) central character (Remark \ref{remTwoCC}). If $\rootSystem(\complexLieAlgebra, \absLieAlgebra)$ contains a short root that is not complex for $\inv$, then the central character of $\genTripleInfDual$ is $\chi^{\vee}$.
\end{proposition}

\begin{proof}
This follows immediately from Remark \ref{remCC} and Theorem \ref{theoremCompareGradings}.
\end{proof}

We now come to the main theorem of this section (note the important restriction at the beginning). As usual we will assume $\realGroupCover$ is not compact.

\begin{theorem}
\label{theoremNumericalDuality}
Let $\realGroupCover = \realSpinGroupCover{p}{q}$ with $p+q = 2n+1$ and assume the rank of $\realGroupCover$ is even. Suppose $\genTripleInf$ is a genuine triple for $\realGroupCover$ with abstract bigrading $\absBg$ and central character $\chi$. If $\realGroupCoverDual = \realSpinGroupCover{p^{\vee}}{q^{\vee}}$ is the dual group  for $\realGroupCover$ and $\genTripleInf$ (Definition \ref{defDualGroup}) then 
\[
\hcBasisGenInvOrderCentSpin{p}{q} = \hcBasisGenInvOrderCentSpinDual{p^{\vee}}{q^{\vee}}.
\]
\end{theorem}

\begin{proof}
First suppose $\rootSystem(\complexLieAlgebra, \absLieAlgebra)$ contains a short root that is not complex for $\inv$. Since $\hcBasisGenInvOrderCentSpin{p}{q} \ne 0$ by hypothesis, Corollary \ref{corRepFiberCent} implies
\[
\hcBasisGenInvOrderCentSpin{p}{q} = 2^{1  - \realBitsBit\realBitsParityBit}.
\]
By Propositions \ref{propSuppRepsDual} and \ref{propCCDual} we also have $\hcBasisGenInvOrderCentSpinDual{p}{q} \ne 0$ so that
\[
\hcBasisGenInvOrderCentSpinDual{p}{q} = 2^{1  - \imaginaryBitsBit\imaginaryBitsParityBit}
\]
and it remains to show 
\[
2^{1  - \realBitsBit\realBitsParityBit} = 2^{1  - \imaginaryBitsBit\imaginaryBitsParityBit}.
\]
Since $n$ is even, $n - \numComplexBits = \numImaginaryBits + \numRealBits$ is even and $\numImaginaryBits$, $\numRealBits$ must have the same parity (i.e.~$\imaginaryBitsParityBit=\realBitsParityBit$). If $\imaginaryBitsParityBit = \realBitsParityBit = 0$, the result is obvious. If $\imaginaryBitsParityBit = \realBitsParityBit = 1$, then $\imaginaryBitsBit = \realBitsBit = 1$ and the result follows.

Now suppose all short roots in $\rootSystem(\complexLieAlgebra, \absLieAlgebra)$ are complex for $\inv$. This is possible only if $\realGroup$ is split, has even rank, and $p^{\vee} = p$ and $q^{\vee} = q$. Case VII of Theorem \ref{theoremNumGenReps} and Case V of Theorem \ref{theoremKOrbits} now imply
\[
\hcBasisGenInvOrderCentSpin{p}{q} = \hcBasisGenInvOrderCentSpinDual{p}{q} = 2
\]
and the result follows.
\end{proof}

\section{Duality}
\label{secAlmostCentralMa}

Let $\absDiff \in \absLieAlgebraDual$ be a half-integral infinitesimal character and recall $\realGroupCover = \realSpinGroupCover{p}{q}$ denotes the (connected) nonalgebraic double cover of $\realGroup = \realSpinGroup{p}{q}$. If $\realGroupCover$ has even rank, Theorem \ref{theoremNumericalDuality} implies 
\[
\hcBasisGenCentSpinInfOrder{p}{q}{\absDiff} = \hcBasisGenCentSpinInfOrderDual{p}{q}{\absDiff}
\]
for each central character $\chi$. In this section we define a bijection
\[
\weylGroupInvMap : \hcBasisGenCentSpinInf{p}{q}{\absDiff} \to \hcBasisGenCentSpinInfDual{p}{q}{\absDiff}
\]
that commutes with the main operations of the (nonlinear) Kazhdan-Lusztig-Vogan algorithm (Section \ref{secIntro}). We begin by recalling these operations along with some of their basic properties.

\subsection{Integral Cross Actions in $\hcBasisGenSpin{p}{q}$}
\label{ssCaofGT}
The first ingredient in the KLV-algorithm is a version of the cross action (Definition \ref{defCrossAction}) for the \emph{integral} abstract Weyl group on genuine triples. This material will be familiar to most readers and we refer to \cite{VGr} or \cite{VPc} for more details. 

To begin, fix a $\cinv$-stable Cartan subgroup $\realTorusCover \subset \realGroupCover$ and let $\rootSystem(\complexLieAlgebra, \complexLieAlgebra[h])$ be its root system. For each $\simpleRoot \in \rootSystem(\complexLieAlgebra, \complexLieAlgebra[h])$, there is a corresponding \emph{root character} $\rootChar$ of $\realTorusCover$ given by the adjoint action of $\realTorusCover$ on $\rootSpace$. The following lemma implies these characters behave like roots.

\begin{lemma}[\cite{VGr}, Lemma 0.4.5]
\label{lemRootChar}
Suppose we have
\begin{eqnarray*}
\sum_{\simpleRoot \in \rootSystem(\complexLieAlgebra, \complexLieAlgebra[h])} n_{\simpleRoot}\simpleRoot & = & \sum_{\simpleRoot \in \rootSystem(\complexLieAlgebra, \complexLieAlgebra[h])} m_{\simpleRoot}\simpleRoot 
\end{eqnarray*}
in $\complexLieAlgebraDual[h]$ with $n_{\simpleRoota}, m_{\simpleRoota} \in \mathbb{Z}$. Then
\begin{eqnarray*}
\prod_{\simpleRoot \in \rootSystem(\complexLieAlgebra, \complexLieAlgebra[h])} \rootChar^{n_{\simpleRoot}} & = & \prod_{\simpleRoot \in \rootSystem(\complexLieAlgebra, \complexLieAlgebra[h])} \rootChar^{m_{\simpleRoot}}
\end{eqnarray*}
as characters of $\realTorusCover$.
\end{lemma}

In particular, we can use root characters to translate between representations of $\realTorusCoverId$ whose differentials differ by a sum of roots. To this end we have the following lemmas whose proofs are similar.

\begin{lemma}
\label{lemRhoRootSum}
Suppose $\diff \in \complexLieAlgebraDual[h]$ is regular and let $\simpleRoota \in \rootSystem(\complexLieAlgebra,\complexLieAlgebra[h])$. Then in the notation of Section \ref{ssHCModIntro}
\begin{eqnarray*}
\halfSumIm[{\rootReflection[\simpleRoota](\diff)}] & = & \halfSumIm[\diff] - \sum_{\simpleRootb \in\rootSystem(\complexLieAlgebra,\complexLieAlgebra[h])}n_{\simpleRootb}\simpleRootb \\
2\halfSumImCpt[{\rootReflection[\simpleRoota](\diff)}] & = & 2\halfSumImCpt[\diff] - \sum_{\simpleRootb\in\rootSystem(\complexLieAlgebra,\complexLieAlgebra[h])}2\cdot m_{\simpleRootb}\simpleRootb
\end{eqnarray*}
with $n_{\simpleRootb},m_{\simpleRootb} \in \mathbb{Z}$.
\end{lemma}

\begin{lemma}
\label{lemRhoRootSumIm}
In the setting of Lemma \ref{lemRhoRootSum}, suppose $\simpleRoota \in \imaginaryRoots[\cinv](\complexLieAlgebra,\complexLieAlgebra[h])$ is an imaginary root. Then
\begin{eqnarray*}
\halfSumIm[{\rootReflection[\simpleRoota](\diff)}] & = & \left\{\begin{array}{ll}
\halfSumIm[\diff] - \displaystyle\sum_{\underset{\simpleRootb\text{ long}}{\simpleRootb\in\rootSystem(\complexLieAlgebra,\complexLieAlgebra[h])}}n_{\simpleRootb}\simpleRootb & \simpleRoot \text { long} \\
\halfSumIm[\diff] - \displaystyle\sum_{\underset{\simpleRootb\text{ long}}{\simpleRootb\in\rootSystem(\complexLieAlgebra,\complexLieAlgebra[h])}}n_{\simpleRootb}\simpleRootb - \simpleRoot & \simpleRoot \text { short}
\end{array}
\right.
\end{eqnarray*}
with $n_{\simpleRootb},m_{\simpleRootb} \in \mathbb{Z}$.
\end{lemma}

\begin{proof} Let
\begin{eqnarray*}
\mathcal{S} & = & \left\{\simpleRootb \in \imaginaryRoots[\cinv](\complexLieAlgebra,\complexLieAlgebra[h]) \mid (\simpleRootb,\diff) > 0 \text { and } (\simpleRootb, \rootReflection[\simpleRoota](\diff)) < 0 \right\}
\end{eqnarray*}
denote the set of imaginary roots that are positive with respect to $\diff$ and negative with respect to $\rootReflection[\simpleRoota]\diff$. Then
\begin{eqnarray*}
\halfSumIm[{\rootReflection[\simpleRoota](\diff)}] & = & \halfSumIm[\diff] - \frac{1}{2}\sum_{\simpleRootb \in \mathcal{S}}\simpleRootb + \frac{1}{2}\sum_{\simpleRootb \in \mathcal{S}}(-\simpleRootb) \\
& = & \halfSumIm[\diff] - \sum_{\simpleRootb \in \mathcal{S}}\simpleRootb.
\end{eqnarray*}
Suppose first that $\simpleRoota$ is long. If every other root in $\mathcal{S}$ is also long, the result clearly follows. Otherwise there are exactly two short roots in $\mathcal{S}$, say $\simpleRootc_{1}$ and $\simpleRootc_{2}$ and we have
\begin{eqnarray*}
\halfSumIm[{\rootReflection[\simpleRoota](\diff)}] & = & \halfSumIm[\diff] - (\sum_{\underset{\simpleRootb\text{ long}}{\simpleRootb \in \mathcal{S}}}\simpleRootb) - \simpleRootc_{1} - \simpleRootc_{2} \\
& = & \halfSumIm[\diff] - (\sum_{\underset{\simpleRootb\text{ long}}{\simpleRootb \in \mathcal{S}}}\simpleRootb) - \simpleRootc
\end{eqnarray*}
where $\simpleRootc = \simpleRootc_{1} + \simpleRootc_{2}$ is a long root. If $\simpleRoota$ is short, then $\simpleRoota$ is the only such root in $\mathcal{S}$ and the result follows.
\end{proof}

\begin{lemma}
\label{lemRhoRootSumCmplx}
In the setting of Lemma \ref{lemRhoRootSum}, suppose $\simpleRoota \in \rootSystem(\complexLieAlgebra,\complexLieAlgebra[h])$ is a complex root. Then
\begin{eqnarray*}
\halfSumIm[{\rootReflection[\simpleRoota](\diff)}] & = & \left\{\begin{array}{ll}
\halfSumIm[\diff] - \displaystyle\sum_{\underset{\simpleRootb\text{ long}}{\simpleRootb\in\rootSystem(\complexLieAlgebra,\complexLieAlgebra[h])}}n_{\simpleRootb}\simpleRootb - \cmplxBit\simpleRootc & \simpleRoota \text { long} \\
\halfSumIm[\diff] - \displaystyle\sum_{\underset{\simpleRootb\text{ long}}{\simpleRootb\in\rootSystem(\complexLieAlgebra,\complexLieAlgebra[h])}}n_{\simpleRootb}\simpleRootb & \simpleRoota \text { short}
\end{array}
\right.
\end{eqnarray*}
where $\cmplxBit \in \left\{0,1\right\}$ and $\simpleRootc \in \rootSystem(\complexLieAlgebra,\complexLieAlgebra[h])$ is a short root.
\end{lemma}

In particular, reflecting $\diff$ by $\rootReflection[\simpleRoot]$ alters the elements $\halfSumIm$ and  $\halfSumImCpt \in \complexLieAlgebraDual[h]$ by integral sums of roots. A similar result holds for $\diff$ itself if we restrict to a certain subgroup of $\weylGroup(\complexLieAlgebra, \complexLieAlgebra[h])$ (Lemma \ref{lemSumOfRoots}).

\begin{definition}[\cite{VGr}, Definition 7.2.16]
Suppose $\diff \in \complexLieAlgebraDual[h]$ is a regular element and let
\begin{eqnarray*}
\rootSystem(\complexLieAlgebra, \complexLieAlgebra[h])(\diff) & = & \left\{\simpleRoot \in \rootSystem(\complexLieAlgebra, \complexLieAlgebra[h]) \mid (\diff, \rootCheck) \in \mathbb{Z}\right\}
\end{eqnarray*}
denote the set of \emph{integral roots} for $\diff$. Then \rootSystem(\complexLieAlgebra, \complexLieAlgebra[h])(\diff) is a subroot system of $\rootSystem(\complexLieAlgebra, \complexLieAlgebra[h])$ and we denote the corresponding \emph{integral Weyl group} by $\weylGroup(\complexLieAlgebra, \complexLieAlgebra[h])(\diff) \subset \weylGroup(\complexLieAlgebra, \complexLieAlgebra[h])$.
\end{definition}

\begin{lemma}[\cite{VGr}, Lemma 7.2.17]
\label{lemSumOfRoots}
In the above setting, $\weylElt$ is an element of $\weylGroup(\complexLieAlgebra, \complexLieAlgebra[h])(\diff)$ if and only if
\begin{eqnarray*}
\weylElt\diff - \diff & = & \sum_{\simpleRootb \in \rootSystem(\complexLieAlgebra,\complexLieAlgebra[h])}n_{\simpleRootb}\simpleRootb~ (n_{\simpleRootb} \in \mathbb{Z}).
\end{eqnarray*}
In other words, $\weylElt \in \weylGroup(\complexLieAlgebra, \complexLieAlgebra[h])(\diff)$ if and only if $\weylElt\diff - \diff$ can be written as an integral sum of roots.
\end{lemma}

We now define the cross action of the abstract \emph{integral} Weyl group on genuine triples for $\realGroupCover$.

\begin{definition}[\cite{VGr}, Definition 8.3.1]
\label{defCrossActionTriples}
Let $\weylGroup = \weylGroup(\complexLieAlgebra, \absLieAlgebra)$ denote the abstract Weyl group and let $\weylElt \in \weylGroup(\absDiff) = \weylGroup(\complexLieAlgebra, \absLieAlgebra)(\absDiff)$. Suppose $\genTripleInf$ is a genuine triple for $\realGroupCover$ and write $\weylEltDiff = \absConj(\weylElt)$ (Definition \ref{defAbsConjMap}) for the image of $\weylElt$ in $\rootSystem(\complexLieAlgebra, \complexLieAlgebra[h])$. Then $\weylEltDiff \in \weylGroup(\complexLieAlgebra, \complexLieAlgebra[h])(\diff)$ and we can write
\[
\weylElt \cross \diff - \diff = \weylEltDiffInv(\diff) - \diff = \sum_{\simpleRoot \in \rootSystem(\complexLieAlgebra, \complexLieAlgebra[h])} n_{\simpleRoot}\simpleRoot ~ (n_{\simpleRoot} \in \mathbb{Z})
\]
by Lemma \ref{lemSumOfRoots}. 
Similarly, Lemma \ref{lemRhoRootSum} implies
\begin{eqnarray*}
(\halfSumIm[{\weylElt \cross \diff}] - 2\halfSumImCpt[{\weylElt \cross \diff)}]) - (\halfSumIm - 2\halfSumImCpt) & = & \sum_{\simpleRoot \in \rootSystem(\complexLieAlgebra, \complexLieAlgebra[h])} m_{\simpleRoot}\simpleRoot ~ (m_{\simpleRoot} \in \mathbb{Z}).
\end{eqnarray*}
Then
\begin{eqnarray*}
\varphi & = & \sum_{\simpleRoot \in \rootSystem(\complexLieAlgebra, \complexLieAlgebra[h])} (n_{\simpleRoot} + m_{\simpleRoot})\simpleRoot
\end{eqnarray*}
gives a well-defined character of $\realTorusCover$ 
\begin{eqnarray*}
\Phi & = & \prod_{\simpleRoot \in \rootSystem(\complexLieAlgebra, \complexLieAlgebra[h])} \rootChar^{(n_{\simpleRoot}+m_{\simpleRoot})}
\end{eqnarray*}
by Lemma \ref{lemRootChar}. We define the \emph{cross action of} $\genTripleInf$ by $\weylElt$ via
\begin{eqnarray*}
\weylElt \cross \torusChar & = & \torusChar \cdot \Phi \\
\weylElt \cross \genTripleInf & = & \genTripleInfDiffChar{\weylElt \cross \diff}{\weylElt \cross \genTorusChar}.
\end{eqnarray*}
\end{definition}

\begin{remark}
\label{remAbsRegCrossAction}
Definition \ref{defCrossActionTriples} describes the (abstract) cross action of $\weylGroup(\diff)$ on genuine triples for $\realGroupCover$. If $\realTorusCover \subset \realGroupCover$ is a Cartan subgroup, we will occasionally need the (regular) cross action of $\weylGroup(\complexLieAlgebra, \complexLieAlgebra[h])(\diff)$ on genuine triples of the form $\genTripleInf$. The definition is obvious from Definition \ref{defCrossActionTriples} and will be denoted the same way.
\end{remark}

The following lemma describes a special case when the cross action of a genuine triple is easy to compute.

\begin{lemma}[\cite{VGr}, Lemma 8.3.2]
\label{lemCaVgr}
Let $\genTripleInf$ be a genuine triple and suppose $\simpleRoota \in \rootSystem(\absDiff) = \rootSystem(\complexLieAlgebra, \absLieAlgebra)(\absDiff)$ is a \emph{simple} abstract root. Write $\simpleRootb = \absConj(\simpleRoota)$ for the image of $\simpleRoota$ in $\rootSystem(\complexLieAlgebra, \complexLieAlgebra[h])(\diff)$ and set $m=(\diff, \rootCheck[\simpleRootb]) \in \mathbb{Z}$. Then if $\simpleRoota$ is compact imaginary
\begin{eqnarray*}
\rootReflection[\simpleRoota] \cross \genTripleInf & = & \genTripleInfDiffChar{\diff - m\simpleRootb}{\genTorusChar \cdot \rootChar[\simpleRootb]^{-(m-1)}},
\end{eqnarray*}
if $\simpleRoot$ is noncompact imaginary
\begin{eqnarray*}
\rootReflection[\simpleRoota] \cross \genTripleInf & = & \genTripleInfDiffChar{\diff - m\simpleRootb}{\genTorusChar \cdot \rootChar[\simpleRootb]^{-(m+1)}},
\end{eqnarray*}
and if $\simpleRoot$ is real or complex
\begin{eqnarray*}
\rootReflection[\simpleRoota] \cross \genTripleInf & = & \genTripleInfDiffChar{\diff - m\simpleRootb}{\genTorusChar \cdot \rootChar[\simpleRootb]^{-m}}.
\end{eqnarray*}
\end{lemma}

\begin{proof}
We have
\begin{eqnarray*}
\rootReflection[\simpleRootb]^{-1}(\diff) & = & \rootReflection[\simpleRootb](\diff) \\ 
& = & \diff - (\diff, \rootCheck[\simpleRootb])\simpleRootb \\
& = & \diff - m\simpleRootb
\end{eqnarray*}
by the definition of $m$. Clearly $\simpleRootb$ is simple in the positive root system for $\rootSystem(\complexLieAlgebra, \complexLieAlgebra[h])$ determined by $\diff$, and the reflection $\rootReflection[\simpleRootb]$ permutes the positive roots other than $\simpleRootb$. Therefore if $\simpleRoota$ is real or complex, $\halfSumIm[\rootReflection \cross \diff] = \halfSumIm$ and $2\halfSumImCpt[\rootReflection \cross \diff] = 2\halfSumImCpt$. In the notation of Definition \ref{defCrossActionTriples}, $\varphi = -m\simpleRootb$ and the result follows.

If $\simpleRoota$ is compact we have $\halfSumIm[\rootReflection \cross \diff] = \halfSumIm - \simpleRootb$, $2\halfSumImCpt[\rootReflection \cross \diff] = 2\halfSumImCpt - 2\simpleRootb$, and
$(\halfSumIm[{\rootReflection \cross \diff}] - 2\halfSumImCpt[{\rootReflection \cross \diff)}]) - (\halfSumIm - 2\halfSumImCpt) = \simpleRootb$. Then $\varphi = -m\simpleRootb + \simpleRootb$ and the result follows. The case for noncompact imaginary roots is handled similarly.
\end{proof}

\begin{proposition}[\cite{RT3}, Chapter 4]
\label{propCrossGenTriples}
Let $\genTripleInf$ be a genuine triple and suppose $\weylElt \in \weylGroup(\absDiff)$. Then $\weylElt \cross \genTripleInf$ is also genuine triple. Moreover, the cross action descends to a well-defined action on the level of genuine parameters for $\realGroupCover$. 
\end{proposition}

For $\hcRep \in \hcBasisGenSpin{p}{q}$, we will write $\weylElt \cross \hcRep \in \hcBasisGenSpin{p}{q}$ for the cross action of $\weylGroup(\absDiff)$ on genuine parameters for $\realGroupCover$.

\begin{corollary}
In the setting of Proposition \ref{propCrossGenTriples}, suppose $\absTriple$ is the abstract triple corresponding to $\genTripleInf$. Then $\weylElt \cross \absTriple$ is supportable (Definition \ref{defSuppTriple}).
\end{corollary}

\begin{proof}
This follows from Proposition \ref{propCrossGenTriples}, but we can also prove it directly. It suffices to show $\weylElt \cross \absTriple$ satisfies the conditions of Proposition \ref{propSuppRep}. Proposition \ref{propCATriples} implies 
\begin{eqnarray*}
\weylElt \cross \absTriple & = & (\weylElt\cdot\inv\cdot\weylElt^{-1}, \weylElt \cross \grading, \absDiff)
\end{eqnarray*}
and $\simpleRoot$ is imaginary for $\weylElt\cdot\inv\cdot\weylElt^{-1}$ if and only if
\[
\weylElt\cdot\inv\cdot\weylElt^{-1}(\simpleRoot) = \simpleRoot \iff \inv(\weylElt^{-1}\simpleRoot) = \weylElt^{-1}\simpleRoot.
\]
In particular, we must have $\weylElt^{-1}\simpleRoot$ imaginary for $\inv$. However
\begin{eqnarray*}
(\absDiff, \rootCheck[{(\weylElt^{-1}\simpleRoot)}]) & = & (\absDiff, \weylElt^{-1}\rootCheck) \\
& = & (\weylElt\absDiff, \rootCheck) \\
& = & (\absDiff + \sum_{\simpleRootb \in \rootSystem}n_{\simpleRootb}\simpleRootb, ~\rootCheck) \\
& \equiv & (\absDiff, \rootCheck) \text{ mod } \mathbb{Z}.
\end{eqnarray*}
Since $\weylElt \cross \grading(\simpleRoot) = \grading(\weylElt^{-1}\simpleRoot)$, the result holds for imaginary roots. Complex roots are handled similarly.
\end{proof}

\begin{proposition}
\label{propCrossCC}
Let $\hcRep \in \hcBasisGenSpin{p}{q}$ and suppose $\weylElt \in \weylGroup(\absDiff)$. Then $\hcRep$ and $\weylElt \cross \hcRep$ have the same central character (Definition \ref{defCC}).
\end{proposition}

\begin{proof}
This follows from Definition \ref{defCrossActionTriples} and the fact that root characters act trivially on Z($\realGroupCover$).
\end{proof}

Fix $\hcRep \in \hcBasisGenSpin{p}{q}$ and suppose $\simpleRoot \in \rootSystem(\absDiff)$ is an abstract integral root. We conclude this section with a partial description of when $\rootReflection \cross \hcRep = \hcRep$. 

\begin{proposition}[\cite{RT3}, Lemma 6.14(a)]
\label{propCrossCpt}
If $\simpleRoota$ is imaginary and compact for $\hcRep$, then $\rootReflection \cross \hcRep = \hcRep$.
\end{proposition}

\begin{proposition}
\label{propCrossCmplx}
If $\simpleRoot$ is complex for $\hcRep$, then $\rootReflection \cross \hcRep \ne \hcRep$.
\end{proposition}

\begin{proof}
The corresponding $\maxRealCompactCover$-orbits not conjugate by Propositions \ref{propKOGrading} and \ref{propCATriples}.
\end{proof}

\begin{proposition}
\label{propCrossNciTypeI}
If $\simpleRoot$ is imaginary, noncompact, and of type I for $\hcRep$ (Definition \ref{defTypeITypeII}), then $\rootReflection \cross \hcRep \ne \hcRep$.
\end{proposition}

\begin{proof}
The corresponding $\maxRealCompactCover$-orbits not conjugate (by definition).
\end{proof}


\subsection{Extended Cross Actions}
\label{ssEcaofGT}
The nonlinear KLV-algorithm of \cite{RT3} (at half-integral infinitesimal character) requires an extended action of the \emph{full} abstract Weyl group on genuine triples. In this section we briefly examine the definition from \cite{RT3}, Chapters 3 and 4.

To begin, recall the abstract root system $\rootSystem = \rootSystem(\complexLieAlgebra,\absLieAlgebra)$ and let $\mathcal{R} = \rootLattice \subset \absLieAlgebraDual$ be the root lattice. Consider the quotient
\[
\mathcal{Q} = \absLieAlgebraDual / \mathcal{R}
\]
and observe the natural action of $\weylGroup$ on $\absLieAlgebraDual$ descends to $\mathcal{Q}$. Denote the image of $\absDiff$ in $\mathcal{Q}$ by $\left[\absDiff\right]$.

\begin{definition}
\label{defFamInfChar}
A \emph{family of infinitesimal characters for $\absDiff \in \absLieAlgebraDual$} is a collection of dominant representatives for the $\weylGroup$-orbit of $\left[\absDiff\right]$ in $\mathcal{Q}$.
\end{definition}

\begin{remark}
In \cite{RT3}, the authors define $\mathcal{Q}$ using the weight lattice $\mathcal{P} = \weightLattice$ instead of the root lattice $\mathcal{R}$. Either choice will suffice; however, the root lattice is more convenient for our purposes.
\end{remark}

The abstract Weyl group $\weylGroup$ acts via the extended cross action (Definition \ref{defExtendedCrossActionTriples}) on collections of genuine parameters whose infinitesimal characters live in families of the kind described above. In order to define this action, it is first necessary to specify how the infinitesimal characters in a family are related. To this end, suppose $\kappa \in \famInfChar[\absDiff]$ and $\weylElt \in \weylGroup$. Define the element $\ecaElt{\kappa}{\weylElt} \in \mathcal{R}$ by the requirement
\[
\kappa + \ecaElt{\kappa}{\weylElt} \in \weylElt \cdot \famInfChar[\absDiff]
\]
with the convention that $\ecaElt{\kappa}{\weylElt} = \weylElt\kappa - \kappa$ (Lemma \ref{lemSumOfRoots}) for $\weylElt \in \weylGroup(\kappa)$. We now have the following extension of Definition \ref{defCrossActionTriples}.

\begin{definition}
\label{defExtendedCrossActionTriples} (\cite{RT3}, Definition 4.1) Let $\absDiff \in \absLieAlgebraDual$ be a half-integral infinitesimal character with corresponding family $\famInfChar[\absDiff]$. Suppose $\kappa \in \famInfChar[\absDiff]$ and $\genTripleInf[\kappa]$ is a genuine triple for $\kappa$. If $\weylElt \in \weylGroup$ set
\[
\weylElt \cross \diff = \diff + \absConj(\ecaElt{\kappa}{\weylElt^{-1}})
\]
\[
(\halfSumIm[{\weylElt \cross \diff}] - 2\halfSumImCpt[{\weylElt \cross \diff)}]) - (\halfSumIm - 2\halfSumImCpt) = \sum_{\simpleRoot \in \rootSystem(\complexLieAlgebra, \complexLieAlgebra[h])} m_{\simpleRoot}\simpleRoot ~ (m_{\simpleRoot} \in \mathbb{Z}).
\]
Then
\begin{eqnarray*}
\varphi & = & \absConj(\ecaElt{\kappa}{\weylElt^{-1}})~+~\sum_{\simpleRoot \in \rootSystem(\complexLieAlgebra, \complexLieAlgebra[h])} m_{\simpleRoot}\simpleRoot
\end{eqnarray*}
determines a well-defined character $\Phi$ of $\realTorusCover$ by Lemma \ref{lemRootChar}. We define the \emph{(extended) cross action of} $\genTripleInf[\kappa]$ by $\weylElt$ to be
\begin{eqnarray*}
\weylElt \cross \torusChar & = & \torusChar \cdot \Phi \\
\weylElt \cross \genTripleInf[\kappa] & = & \genTripleInfDiffCharInf{\weylElt \cross \diff}{\weylElt \cross \genTorusChar}{}.
\end{eqnarray*}
\end{definition}

\begin{remark}
As the notation suggests, Definition \ref{defCrossActionTriples} and Definition \ref{defExtendedCrossActionTriples} coincide whenever $\weylElt \in \weylGroup(\kappa)$. This follows immediately from the convention $\ecaElt{\kappa}{\weylElt} = \weylElt\kappa - \kappa$ for $\weylElt$ in $\weylGroup(\kappa)$.
\end{remark}

\begin{remark}
\label{remECAInfChar}
The infinitesimal character of $\weylElt \cross \genTripleInf[\kappa]$ is an element of $\famInfChar[\absDiff]$ that depends on $\kappa$ and the element $\weylElt$.
\end{remark}

\begin{proposition}[\cite{RT3}, Chapter 4]
\label{propExtCrossGenTriples}
In the above setting, let $\genTripleInf$ be a genuine triple and suppose $\weylElt \in \weylGroup$. Then $\weylElt \cross \genTripleInf$ is genuine triple and the extended cross action descends to a well-defined action on the level of genuine parameters for $\realGroupCover$. 
\end{proposition}

\subsection{Even Parity Cartan Subgroups}
\label{ssEPCSG}
Fix a half-integral infinitesimal character $\absDiff \in \absLieAlgebraDual$ and assume the rank of $\realGroupCover = \realSpinGroupCover{p}{q}$ is \emph{even}. Let $\realTorusCover \subset \realGroupCover$ be a $\cinv$-stable Cartan subgroup and suppose $\diff \in \complexLieAlgebraDual[h]$ is a regular element. Write $\inv = \absConjInv(\cinv)$ for the corresponding abstract involution and let $\imaginaryBitsBit, \imaginaryBitsParityBit, \realBitsBit, \realBitsParityBit$ be the indicator bits for $\inv$ (Definition \ref{defInvParams}). Note the indicator bits depend only on the $\weylGroup$-conjugacy class of $\inv$ and therefore only on the $\maxRealCompactCover$-conjugacy class of $\realTorusCover$. Since the rank of $\realGroupCover$ is even, we always have $\imaginaryBitsParityBit = \realBitsParityBit$. 

\begin{definition}
\label{defEvenParity}
A Cartan subgroup for which $\imaginaryBitsParityBit = \realBitsParityBit = 0$ is said to be of \emph{even parity}. In particular, an abstract involution corresponds to an even parity Cartan subgroup if and only if its diagram has an even number of both `$+$' and `$-$' signs. 
\end{definition}

\begin{lemma}
Let $\realTorusCover \subset \realGroupCover$ be a $\cinv$-stable Cartan subgroup and suppose $\simpleRoot \in \realRoots[\cinv](\complexLieAlgebra, \complexLieAlgebra[h])$ is a real root. Then $\realTorusCover$ and $\realTorusCtCover$ (Definition \ref{defCT}) have the same parity if and only if $\simpleRoot$ is long.
\end{lemma}

\begin{proof}
It is easily verified either $\numRealBits[\rootReflection\inv] = \numRealBits$ or $\numRealBits[\rootReflection\inv] = \numRealBits-2$ and the result follows.
\end{proof}

\begin{definition}
\label{defMaxSplitEvenParity}
An even parity Cartan subgroup $\realTorusESplitCover \subset \realGroupCover$ is said to be \emph{evenly split} if the number of `$-$' signs in any diagram for $\realTorusESplitCover$ is maximal among even parity Cartan subgroups.
\end{definition}

\begin{remark}
\label{remMaxSplitEvenParity}
An evenly split Cartan subgroup is maximally split if and only if $q$ is even. For any $\realSpinGroupCover{p}{q}$ there is a unique conjugacy class of evenly split Cartan subgroups.
\end{remark}

Remark \ref{remMaxSplitEvenParity} leads to some slightly messy expressions involving $n$ and $q$. To simplify the notation we make the following definition.

\begin{definition}
\label{defQHat}
For each $\realGroupCover = \realSpinGroupCover{p}{q}$ we set
\[
\qSpec = \left\lceil\frac{q-1}{2}\right\rceil.
\]
\end{definition}

The major difficulties in the Kazhdan-Lusztig algorithm concern the delicate nature of even parity Cartan subgroups. To overcome these (and other) issues, we must construct our even parity Cartan subgroups in a regular way. To begin, fix an evenly split Cartan subgroup $\realTorusESplitCover \subset \realGroupCover$ and choose a positive root system $\posRootSystem(\complexLieAlgebra, \eSplitLieAlgebra)$ such that either
\[
\begin{array}{c}
\underset{1}{-}~-~\cdots~\underset{q}{-}~+~+~\cdots~\underset{n}{+} \\
\text{or} \\
\underset{1}{-}~-~\cdots~\underset{q-1}{-}~\oplus~\oplus~\cdots~\underset{n}{\oplus}
\end{array}
\]
is the diagram for $\realTorusESplitCover$. In the usual coordinates for $\rootSystem(\complexLieAlgebra, \eSplitLieAlgebra)$ set
\begin{eqnarray*}
\simpleRootai & = & \ei[2i-1] - \ei[2i]\\
\simpleRootbi & = & \ei[2i-1] + \ei[2i]\\
\simpleRootc & = & \ei[n].
\end{eqnarray*}
Choose root vectors $\rootVector[\simpleRootai] \in \rootSpace[\simpleRootai]$ and $\rootVector[\simpleRootc] \in \rootSpace[\simpleRootc]$ according to Lemma \ref{lemXaSign} and define
\[
\stCenter = \mac[{\simpleRooti[1]}]\mac[{\simpleRooti[2]}] \cdots \mac[{\simpleRooti[n/2]}].
\]
The element $\stCenter$ is almost central in $\realGroupCover$ and will play a critical role in our definition of $\weylGroupInvMap$. Using Lemma \ref{lemAbelianCompGroup} and Proposition \ref{propMCStruct} it is easy to verify
\[
\text{Z}(\realTorusESplitCover) =  \left\{\begin{array}{cl}
\realTorusESplitCover & q \in \left\{0,1\right\} \\
\left<\realTorusESplitCoverId, \pm\stCenter\right> & 2 \le q < n \\
\left<\pm\realTorusESplitCoverId, \pm\mac[\simpleRootc], \pm\stCenter\right> & q=n
\end{array}\right.
\]
and a genuine character of Z($\realTorusESplitCover$) is determined by its differential and its values on $\mac[\simpleRootc]$ and $\stCenter$. We have seen $\mac[\simpleRootc] \in \text{Z}(\realGroupCover)$ has order two and the order of $\stCenter$ is determined by 
\begin{eqnarray*}
\stCenter^{2} & = & (\mac[{\simpleRooti[1]}]\mac[{\simpleRooti[2]}] \cdots \mac[{\simpleRooti[n/2]}])^{2} \\
& = & \mac[{\simpleRooti[1]}]^{2}\mac[{\simpleRooti[2]}]^{2} \cdots \mac[{\simpleRooti[n/2]}]^{2} \\
& = & (-1)^{\qSpec}\cdot(1)^{\frac{n}{2}-\qSpec}
\end{eqnarray*}
to be either two or four. In particular there are 1, 2, or 4 possible \emph{genuine} characters of Z($\realTorusESplitCover$) with a fixed differential. 

\begin{remark}
\label{remCentralChar}
Suppose $\genTripleESplitInf$ is a genuine triple for $\realGroupCover$ with half-integral infinitesimal character $\absDiff$. Proposition \ref{propGenRepBij} implies the isomorphism class of $\genTorusChar$ is determined by its restriction to Z($\realTorusESplitCover$) and this restriction is given by
\begin{eqnarray*}
\genTorusChar|_{Z(\realTorusESplitCover)} = m\chi
\end{eqnarray*}
where $\chi$ is a genuine character for Z($\realTorusESplitCover$). Therefore we may treat $\genTorusChar$ as either a genuine representation of $\realTorusESplitCover$ or a genuine character of Z($\realTorusESplitCover$).
\end{remark}

\begin{proposition}
\label{propEvenParityCartans}
Up to conjugacy, every even parity Cartan subgroup of $\realGroupCover$ can be obtained from $\realTorusESplitCover$ through an iterative application of the operators $\ctOp[\simpleRootai]$ and $\ctOp[\simpleRootbi]$ (Definition \ref{defCT}).
\end{proposition}

\begin{proof}
This is easily verified on the level of involutions.
\end{proof}

Proposition \ref{propEvenParityCartans} implies we can associate a (nonunique) sequence in $\left\{\simpleRootai,\simpleRootbi\right\}$ to each conjugacy class of even parity Cartan subgroups in $\realGroupCover$. Since the corresponding operators $\ctOp[\simpleRootai]$ and $\ctOp[\simpleRootbi]$ commute (Lemma \ref{lemCaCbMbCommute}), the ordering of the roots is unimportant. To eliminate the ambiguity, we make the following definition.

\begin{definition}
\label{defStdSeq}
A sequence in $\left\{\simpleRootai, \simpleRootbi\right\}$ is called \emph{standard} if it is of the form
\begin{eqnarray*}
\pSeq{j}{k} & = & \simpleRootbi[k]\cdots\simpleRootbi[2]\simpleRootbi[1]\simpleRootai[j]\cdots\simpleRootai[2]\simpleRootai[1]
\end{eqnarray*}
with $0 \le k \le j \le \qSpec$.
\end{definition}

There is a unique standard sequence associated to each conjugacy class of even parity Cartan subgroups in $\realGroupCover$. Clearly there are
\begin{eqnarray*}
1 + 2 + \cdots + \left(\qSpec + 1\right) = \binom{\qSpec+2}{2}
\end{eqnarray*}
possible standard sequences (including the empty sequence) and thus $\binom{\qSpec+2}{2}$ even parity Cartan subgroups of $\realGroupCover$ (up to conjugacy). If $\pSeq{j}{k}$ is a standard sequence we will write
\[
\realTorusESplitCtCover[{\pSeq{j}{k}}] = \ctOpSeq{j}{k}(\realTorusESplitCover) = \ctOp[{\simpleRootbi[k]}]\cdots\ctOp[{\simpleRootbi[2]}]\ctOp[{\simpleRootbi[1]}]\ctOp[{\simpleRootai[j]}]\cdots\ctOp[{\simpleRootai[2]}]\ctOp[{\simpleRootai[1]}](\realTorusESplitCover)
\]
for the corresponding even parity Cartan subgroup.

\begin{proposition}
\label{propMaInHsc}
Let $\pSeq{j}{k}$ be a standard sequence and suppose $1 \le i \le \frac{n}{2}$. Then $\mac[{\simpleRooti}] \in \realTorusESplitCtCover[{\pSeq{j}{k}}]$. In particular, the product 
\[
\stCenter = \mac[{\simpleRooti[1]}]\mac[{\simpleRooti[2]}] \cdots \mac[{\simpleRooti[n/2]}] \in \realTorusESplitCover
\]
is an element of $\realTorusESplitCtCover[{\pSeq{j}{k}}]$.
\end{proposition}

\begin{proof}
This follows from Corollary \ref{corMaRootSpace} and Proposition \ref{propMaIdCompH}.
\end{proof}

\begin{proposition}
\label{propZHjk}
Let $\pSeq{j}{k}$ be a nonempty standard sequence and write $\mac[\simpleRootc]$ for the unique element corresponding to any short real root in $\realTorusESplitCover$ (Proposition \ref{propSameMacReal}). Then the center of $\realTorusESplitCtCover[{\pSeq{j}{k}}]$ is given by
\begin{eqnarray*}
\text{Z}(\realTorusESplitCtCover[{\pSeq{j}{k}}]) & = & \left\{\begin{array}{ll}
\left<\realTorusESplitCtCoverId[{\pSeq{j}{k}}],~\mac[{\simpleRootai[j+1]}]\cdots\mac[{\simpleRootai[\qSpec]}],~ \mac[\simpleRootc]\right> & j<\qSpec, k=0 \\
\left<\realTorusESplitCtCoverId[{\pSeq{j}{k}}],~\mac[{\simpleRootai[j+1]}]\cdots\mac[{\simpleRootai[\qSpec]}]\right> & j<\qSpec, k>0 \\
\left<\realTorusESplitCtCoverId[{\pSeq{j}{k}}],~\mac[\simpleRootc]\right> & j=\qSpec, k=0 \\
\realTorusESplitCtCoverId[{\pSeq{j}{k}}] & j=\qSpec, k>0 \\
\end{array}
\right.
\end{eqnarray*}
In particular, $\stCenter \in \text{Z}(\realTorusESplitCtCover[{\pSeq{j}{k}}])$.
\end{proposition}

\begin{proof}
Since $\pSeq{j}{k}$ is nonempty, $-1 \in \realTorusESplitCtCoverId[{\pSeq{j}{k}}]$. The first result follows from Proposition \ref{propCompH} and Proposition \ref{propMCStruct}. The last statement follows by Proposition \ref{propMaIdCompH}.
\end{proof}

\begin{corollary}
\label{corDiffSignZ}
In the setting of Proposition \ref{propZHjk}, let $\pSeq{j}{k}$ be a standard sequence with $j < \qSpec$. Suppose the genuine triples $\genTripleESplitInfCtDiffChar{\diff}{\genTorusChar_{1}}{\pSeq{j}{k}}$ and
$\genTripleESplitInfCtDiffChar{\diff}{\genTorusChar_{2}}{\pSeq{j}{k}}$ are not $\maxRealCompactCover$-conjugate and have the same central character. Then $\genTorusChar_{2}(\stCenter) = -\genTorusChar_{1}(\stCenter)$.
\end{corollary}

\begin{proof}
We have
\begin{eqnarray*}
\genTorusChar_{1}(\mac[{\simpleRootai[1]}]\cdots\mac[{\simpleRootai[j]}]) & = & \genTorusChar_{2}(\mac[{\simpleRootai[1]}]\cdots\mac[{\simpleRootai[j]}]) \\
\genTorusChar_{1}(\mac[{\simpleRootai[\qSpec+1]}]\cdots\mac[{\simpleRootai[n/2]}]) & = & \genTorusChar_{2}(\mac[{\simpleRootai[\qSpec+1]}]\cdots\mac[{\simpleRootai[n/2]}])
\end{eqnarray*}
since $\mac[{\simpleRootai[1]}]\cdots\mac[{\simpleRootai[j]}]$, $\mac[{\simpleRootai[\qSpec+1]}]\cdots\mac[{\simpleRootai[n/2]}] \in \realTorusESplitCtCoverId[\pSeq{j}{k}]$. However, any genuine character extending $\genPairESplitCtInf[\pSeq{j}{k}]$ is determined by its restriction to Z($\realTorusESplitCtCover[{\pSeq{j}{k}}]$). Since $\genTripleESplitInfCtDiffChar{\diff}{\genTorusChar_{1}}{\pSeq{j}{k}}$ and
$\genTripleESplitInfCtDiffChar{\diff}{\genTorusChar_{2}}{\pSeq{j}{k}}$ have the same central characters we must have
\begin{eqnarray*}
\genTorusChar_{1}(\mac[{\simpleRootai[j+1]}]\cdots\mac[{\simpleRootai[\qSpec]}]) & \ne & \genTorusChar_{2}(\mac[{\simpleRootai[j+1]}]\cdots\mac[{\simpleRootai[\qSpec]}])
\end{eqnarray*}
and the result follows.
\end{proof}

\begin{proposition}
\label{propNonPsCross}
Let $\pSeq{j}{k}$ be a standard sequence and suppose $\simpleRootd \in \rootSystem(\complexLieAlgebra,\eSplitLieAlgebraCt[{\pSeq{j}{k}}])$. Then
\begin{eqnarray*}
\rootChar[\simpleRootd](\stCenter) & = & \left\{\begin{array}{rl}
1 & \simpleRootd \text{ is long} \\
-1 & \simpleRootd \text{ is short} 
\end{array}
\right..
\end{eqnarray*}
\end{proposition}

\begin{proof}
According to Proposition \ref{propMaRootSpace}
\begin{eqnarray*}
\rootChar[\simpleRootd](\stCenter) & = & \rootChar[\simpleRootd](\mac[{\simpleRootai[1]}])\rootChar[\simpleRootd](\mac[{\simpleRootai[2]}])\cdots\rootChar[\simpleRootd](\mac[{\simpleRootai[n/2]}]) \\
& = & (-1)^{(\simpleRootd, \rootCheck[{\simpleRootai[1]}])}(-1)^{(\simpleRootd, \rootCheck[{\simpleRootai[2]}])}\cdots(-1)^{(\simpleRootd, \rootCheck[{\simpleRootai[n/2]}])}.
\end{eqnarray*}
Suppose first that $\simpleRootd$ is long. If $\simpleRootd = \simpleRootai[j]$ or $\simpleRootd = \simpleRootbi[j]$ for some $j$, we have $(\simpleRootd, \rootCheck[{\simpleRootai}]) \in \left\{0,2\right\}$ for all $i$. Otherwise, $(\simpleRootd, \rootCheck[{\simpleRootai}]) = -1$ for exactly two values of $i$ and the result follows. If $\simpleRootd$ is short, $(\simpleRootd, \rootCheck[{\simpleRootai}]) = -1$ for exactly one value of $i$.
\end{proof}

We will need an explicit description of the cross action (Definition \ref{defCrossActionTriples}) in real roots for even parity Cartan subgroups. This essentially requires understanding how conjugation affects the element $\stCenter$. We begin with the following general lemma.

\begin{lemma}
\label{lemVgr}
Let $\complexLieAlgebra[h] \subset \complexLieAlgebra$ be a $\cinv$-stable Cartan subalgebra and suppose $\simpleRoota, \simpleRootb \in \realRoots[\cinv](\complexLieAlgebra, \complexLieAlgebra[h])$. Choose root vectors $\rootVector[\simpleRoota] \in \rootSpace[\simpleRoota]$ and $\rootVector[\simpleRootb] \in \rootSpace[\simpleRootb]$ according to Lemma \ref{lemXaSign}. Then we have
\begin{eqnarray*}
\mac[\simpleRoota]\oac[\simpleRootb]\mac[\simpleRoota]^{-1} & = & \oac[\simpleRootb]^{\simpleRootb(\mac[\simpleRoota])} \\ 
\oac[\simpleRootb]\mac[\simpleRoota]\oac[\simpleRootb]^{-1} & = & \oac[\simpleRootb]^{(1-\simpleRootb(\mac[\simpleRoota]))} \mac[\simpleRoota].
\end{eqnarray*}
\end{lemma}

\begin{proof}
The first statement is proved in exactly the same way as \cite{VGr}, Lemma 4.3.19(c). For the second statement we have
\begin{eqnarray*}
\mac[\simpleRoota]\oac[\simpleRootb]\mac[\simpleRoota]^{-1} & = & \oac[\simpleRootb]^{\simpleRootb(\mac[\simpleRoota])} \\ 
\mac[\simpleRoota]\oac[\simpleRootb]^{-1}\mac[\simpleRoota]^{-1} & = & \oac[\simpleRootb]^{-\simpleRootb(\mac[\simpleRoota])} \\ 
\oac[\simpleRootb]\mac[\simpleRoota]\oac[\simpleRootb]^{-1} & = & \oac[\simpleRootb]\oac[\simpleRootb]^{-\simpleRootb(\mac[\simpleRoota])}\mac[\simpleRoota] \\ 
\oac[\simpleRootb]\mac[\simpleRoota]\oac[\simpleRootb]^{-1} & = & \oac[\simpleRootb]^{(1-\simpleRootb(\mac[\simpleRoota]))} \mac[\simpleRoota]
\end{eqnarray*}
as desired.
\end{proof}

\begin{proposition}
\label{propPsConj}
Fix a standard sequence $\pSeq{j}{k}$ and let $\genTripleESplitCtInf[\pSeq{j}{k}]$ be a genuine triple. For $\simpleRootd \in \realRoots[\cinv](\complexLieAlgebra,\eSplitLieAlgebraCt[\pSeq{j}{k}])$ choose $\rootVector[\simpleRootd] \in \rootSpace[\simpleRootd]$ according to Lemma \ref{lemXaSign} and write $\oac[\simpleRootd]$ for the corresponding root reflection in $\maxRealCompactCover$. Then $\oac[\simpleRootd]$ acts of $\genTorusChar$ by conjugation and we have
\begin{eqnarray*}
(\oac[\simpleRootd] \cdot \genTorusChar)(\stCenter) & = & \left\{\begin{array}{cl}
\genTorusChar(\stCenter) & \simpleRootd \text{ is long} \\
\genTorusChar(\mac[\simpleRootc])\genTorusChar(\stCenter) & \simpleRootd \text{ is short}
\end{array}
\right..
\end{eqnarray*}
\end{proposition}

\begin{proof}
This follows in the same fashion as Proposition \ref{propNonPsCross} using Proposition \ref{lemVgr}.
\end{proof}

\begin{corollary}
\label{corZCrossReal}
Let $\absDiff$ be a half-integral infinitesimal character and fix a standard sequence $\pSeq{j}{k}$. If $\genTripleESplitCtInf[\pSeq{j}{k}]$ is a genuine triple and $\simpleRootd \in \realRoots[\cinv](\complexLieAlgebra, \eSplitLieAlgebraCt[\pSeq{j}{k}])$ is a real root, then $\rootReflection[\simpleRootd] \cross \genTripleESplitCtInf[\pSeq{j}{k}]$ and  $\genTripleESplitCtInf[\pSeq{j}{k}]$ are $\maxRealCompactCover$-conjugate if and only if $\simpleRootd$ does \emph{not} satisfy the parity condition (Definition \ref{defPC}).
\end{corollary}

\begin{proof}
First suppose $\simpleRootd$ is long and does not satisfy the parity condition. Then $\simpleRootd$ is integral and Propositions \ref{propPsConj} and \ref{propNonPsCross} imply
\[
(\oac[{{\simpleRootd}}] \cdot \genTorusChar)(\stCenter) = \genTorusChar(\stCenter) = (\rootReflection[\simpleRootd] \cross \genTorusChar)(\stCenter).
\]
In particular, $\rootReflection[\simpleRootd] \cross \genTripleESplitCtInf[\pSeq{j}{k}]$ and $\genTripleESplitCtInf[\pSeq{j}{k}]$ are conjugate by an element of $\maxRealCompactCover$. If $\simpleRootd$ satisfies the parity condition, then $\rootReflection[\simpleRootd] \cross \genTripleESplitCtInf[\pSeq{j}{k}]$ and $\genTripleESplitCtInf[\pSeq{j}{k}]$ have different infinitesimal characters and cannot be $\maxRealCompactCover$-conjugate.

If $\simpleRootd$ is short (and thus integral) we have
\begin{eqnarray*}
(\oac[{{\simpleRootd}}] \cdot \genTorusChar)(\stCenter) & = & \genTorusChar(\mac[\simpleRootd])\genTorusChar(\stCenter) \\
(\rootReflection[\simpleRootd] \cross \genTorusChar)(\stCenter) & = & \simpleRootd(\stCenter)^{(\diff,\rootCheck[\simpleRootd])}\genTorusChar(\stCenter) \\
& = & (-1)^{(\diff,\rootCheck[\simpleRootd])}\genTorusChar(\stCenter)
\end{eqnarray*}
so that
$(\oac[{{\simpleRootd}}] \cdot \genTorusChar)(\stCenter) = (\rootReflection[\simpleRootd] \cross \genTorusChar)(\stCenter)$ if and only if $\genTorusChar(\mac[\simpleRootd]) = (-1)^{(\diff,\rootCheck[\simpleRootd])}$, i.e., $\simpleRootd$ does not satisfy the parity condition.
\end{proof}

\begin{corollary}
\label{corZCrossIm}
Let $\genTripleESplitCtInf[\pSeq{j}{k}]$ be a genuine triple and suppose $\simpleRoot \in \imaginaryRoots[\cinv](\complexLieAlgebra, \eSplitLieAlgebraCt[\pSeq{j}{k}])(\diff)$ is an imaginary integral root. If $m=(\diff,\rootCheck)$, then 
\begin{eqnarray*}
\rootReflection \cross \genTorusChar(\stCenter) & = & \left\{\begin{array}{rl}
\genTorusChar(\stCenter) & \simpleRoot \text{ long} \\
(-1)^{m+1}\genTorusChar(\stCenter) & \simpleRoot \text{ short}
\end{array}
\right..
\end{eqnarray*}
\end{corollary}

\begin{proof}
In the notation of Definition \ref{defCrossActionTriples}, Lemma \ref{lemRhoRootSumIm} implies
\begin{eqnarray*}
\varphi & = & \left\{\begin{array}{ll}
-m\simpleRoot + \displaystyle\sum_{\underset{\simpleRootb\in\rootSystem(\complexLieAlgebra,\splitLieAlgebraCt[\pSeq{j}{k}])}{\simpleRootb\text{ long}}}n_{\simpleRootb}\simpleRootb + \sum_{\simpleRootb\in\rootSystem(\complexLieAlgebra,\splitLieAlgebraCt[\pSeq{j}{k}])}2\cdot m_{\simpleRootb}\simpleRootb & \simpleRoot \text{ long} \\
-m\simpleRoot -\simpleRoot + \displaystyle\sum_{\underset{\simpleRootb\in\rootSystem(\complexLieAlgebra,\splitLieAlgebraCt[\pSeq{j}{k}])}{\simpleRootb\text{ long}}}n_{\simpleRootb}\simpleRootb + \sum_{\simpleRootb\in\rootSystem(\complexLieAlgebra,\splitLieAlgebraCt[\pSeq{j}{k}])}2\cdot m_{\simpleRootb}\simpleRootb & \simpleRoot \text{ short}
\end{array}
\right.
\end{eqnarray*}
and the result follows by Proposition \ref{propNonPsCross}.
\end{proof}

\begin{corollary}
\label{corZCrossCmplx}
Let $\genTripleESplitCtInf[\pSeq{j}{k}]$ be a genuine triple and suppose $\simpleRoot \in \complexRoots[\cinv](\complexLieAlgebra, \eSplitLieAlgebraCt[\pSeq{j}{k}])(\diff)$ is a complex integral root. If $m=(\diff,\rootCheck)$, then in the notation of Proposition \ref{lemRhoRootSumCmplx}
\begin{eqnarray*}
\rootReflection \cross \genTorusChar(\stCenter) & = & \left\{\begin{array}{rl}
(-1)^{\cmplxBit}\genTorusChar(\stCenter) & \simpleRoot \text{ long} \\
(-1)^{m}\genTorusChar(\stCenter) & \simpleRoot \text{ short}
\end{array}
\right..
\end{eqnarray*}
\end{corollary}

\begin{proof}
In the notation of Definition \ref{defCrossActionTriples} and Lemma \ref{lemRhoRootSumCmplx} we have
\begin{eqnarray*}
\varphi & = & \left\{\begin{array}{ll}
-m\simpleRoot - \cmplxBit\simpleRootc + \displaystyle\sum_{\underset{\simpleRootb\in\rootSystem(\complexLieAlgebra,\splitLieAlgebraCt[\pSeq{j}{k}])}{\simpleRootb\text{ long}}}n_{\simpleRootb}\simpleRootb + \sum_{\simpleRootb\in\rootSystem(\complexLieAlgebra,\splitLieAlgebraCt[\pSeq{j}{k}])}2\cdot m_{\simpleRootb}\simpleRootb & \simpleRoot \text{ long} \\
-m\simpleRoot + \displaystyle\sum_{\underset{\simpleRootb\in\rootSystem(\complexLieAlgebra,\splitLieAlgebraCt[\pSeq{j}{k}])}{\simpleRootb\text{ long}}}n_{\simpleRootb}\simpleRootb + \sum_{\simpleRootb\in\rootSystem(\complexLieAlgebra,\splitLieAlgebraCt[\pSeq{j}{k}])}2\cdot m_{\simpleRootb}\simpleRootb & \simpleRoot \text{ short}
\end{array}
\right.
\end{eqnarray*}
and the result follows by Proposition \ref{propNonPsCross}.
\end{proof}

We will occasionally need the following generalization of Corollary \ref{corZCrossIm}.

\begin{corollary}
\label{corZConjIm}
Let $\realTorusESplitCtCover[\pSeq{j}{k}] \subset \realGroupCover$ be an even parity Cartan subgroup with $j=\qSpec$ and let $\simpleRoot \in \imaginaryRoots[\cinv](\complexLieAlgebra, \eSplitLieAlgebraCt[\pSeq{j}{k}])$ be an imaginary root (not necessarily integral). Suppose there exist genuine triples of the form $\genTripleESplitInfCtDiffChar{\diff}{\genTorusChar_{1}}{\pSeq{j}{k}}$ and $\genTripleESplitInfCtDiffChar{\rootReflection\cdot\diff}{\genTorusChar_{2}}{\pSeq{j}{k}}$ with the same central character. If $m = (\diff, \rootCheck)$ then
\begin{eqnarray*}
\genTorusChar_{2}(\stCenter) & = & \left\{\begin{array}{rl}
(-1)^{2m}\genTorusChar_{1}(\stCenter) & \simpleRoot \text{ long} \\
(-1)^{m+1}\genTorusChar_{1}(\stCenter) & \simpleRoot \text{ short}
\end{array}
\right..
\end{eqnarray*}
\end{corollary}

\begin{proof}
Since $\stCenter$ is an element of $\realTorusESplitCtCoverId[\pSeq{j}{k}]$ (Proposition \ref{propZHjk}), it suffices to consider differentials. Write $\ctOpRooti = \ctOpSeq{j}{k}(\simpleRootai)$ for the image of the root $\simpleRootai$ in $\rootSystem(\complexLieAlgebra, \eSplitLieAlgebraCt[\pSeq{j}{k}])$ and let $\coroot[\ctOpRooti]$ denote the corresponding coroot. Suppose first that $\simpleRoot$ is long. Lemma \ref{lemRhoRootSumIm} gives
\begin{eqnarray*}
\text{d}\genTorusChar_{2} & = & \rootReflection(\diff) + \halfSumIm[{\rootReflection}(\diff)] - 2\halfSumImCpt[{\rootReflection(\diff)}] \\
& = & \diff - m\simpleRoot + \halfSumIm[\diff] - \sum_{\underset{\simpleRootb\in\rootSystem(\complexLieAlgebra,\splitLieAlgebraCt[\pSeq{j}{k}])}{\simpleRootb\text{ long}}}n_{\simpleRootb}\simpleRootb - (2\halfSumImCpt[\diff] - \sum_{\simpleRootb\in\rootSystem(\complexLieAlgebra,\splitLieAlgebraCt[\pSeq{j}{k}])}2\cdot m_{\simpleRootb}\simpleRootb) \\
& = & \text{d}\genTorusChar_{1} - m\simpleRoota - \sum_{\underset{\simpleRootb\in\rootSystem(\complexLieAlgebra,\splitLieAlgebraCt[\pSeq{j}{k}])}{\simpleRootb\text{ long}}}n_{\simpleRootb}\simpleRootb + \sum_{\simpleRootb\in\rootSystem(\complexLieAlgebra,\splitLieAlgebraCt[\pSeq{j}{k}])}2\cdot m_{\simpleRootb}\simpleRootb
\end{eqnarray*}
and Proposition \ref{propNonPsCross} implies nontrivial changes come only from the $-m\simpleRoot$ term. From the definition of $\mac[\ctOpRooti]$ (Definition \ref{defZb}) we have
\begin{eqnarray*}
\genTorusChar_{2}(\stCenter) & = & \genTorusChar_{1}(\stCenter)e^{-m\simpleRoot(\pi i\coroot[{\ctOpRooti[1]}])} e^{-m\simpleRoot(\pi i\coroot[{\ctOpRooti[2]}])}\cdots e^{-m\simpleRoot(\pi i\coroot[{\ctOpRooti[k]}])} \\
& = & \genTorusChar_{1}(\stCenter)e^{-m\pi i(\simpleRoot,\rootCheck[{\ctOpRooti[1]}])} e^{-m\pi i(\simpleRoot,\rootCheck[{\ctOpRooti[2]}])} \cdots e^{-m\pi i(\simpleRoot,\rootCheck[{\ctOpRooti[k]}])}
\end{eqnarray*}
with $k = \frac{n}{2}$. Since $\genTripleESplitInfCtDiffChar{\diff}{\genTorusChar_{1}}{\pSeq{j}{k}}$ and $\genTripleESplitInfCtDiffChar{\rootReflection\cdot\diff}{\genTorusChar_{2}}{\pSeq{j}{k}}$ have the same central character, we must have $\simpleRoot = \ctOpRooti$ for some $i$. Then $(\simpleRoot,\rootCheck[\ctOpRooti]) = 2$ and $(\simpleRoot,\rootCheck[{\ctOpRooti[j]}]) = 0$ for $j\ne i$. Therefore 
\begin{eqnarray*}
\genTorusChar_{2}(\stCenter) & = & \genTorusChar_{1}(\stCenter)e^{-2m\pi i}
\end{eqnarray*}
and the result follows. The case for $\simpleRoot$ short is similar.
\end{proof}

\subsection{Cayley Transforms in $\hcBasisGen[\absDiff](p,q)$}
\label{ssCtofGt}
In Section \ref{ssCaofGT}, we extended the cross action of the (integral) abstract Weyl group to $\hcBasisGen[\absDiff](p,q)$. This operation produced new elements in $\hcBasisGen[\absDiff](p,q)$ whose corresponding Cartan subgroups were conjugate. In this section we extend the Cayley transform operation of Section \ref{ssCT} to $\hcBasisGen[\absDiff](p,q)$. Not surprisingly, this operation produces new elements whose corresponding Cartan subgroups are not conjugate. This is the final operation needed in our discussion of the (nonlinear) KLV-algorithm. The material here is well known and we refer the reader to \cite{IC1} and \cite{VGr} for more details.

To begin, let $\realTorusCover = \realTorusCoverT\realTorusCoverA$ be a $\cinv$-stable Cartan subgroup of $\realGroupCover$ and fix a noncompact imaginary root $\simpleRoot \in \imaginaryRoots[\cinv](\complexLieAlgebra, \complexLieAlgebra[h])$. Choose a corresponding Cayley transform operator $\ctOpNci$ (Definition \ref{defCT}) and write 
\begin{eqnarray*}
\ctOpNci(\realTorusCover) & = & \realTorusCtCoverNci = \realTorusCoverTCt\realTorusCoverACt \\
\realTorusCoverTCtInt & = &  \realTorusCoverT \cap \realTorusCoverTCt \\
\realTorusCtCoverInt & = & \realTorusCoverTCtInt\realTorusCoverACt.
\end{eqnarray*}
The following proposition describes the relationship between $\realTorusCoverT$ and $\realTorusCoverTCt$.

\begin{proposition}[\cite{VGr}, Lemma 8.3.5]
\label{propNciCtVgr}
If $\simpleRoot$ is of type I (Definition \ref{defTypeITypeII}) then $\realTorusCoverTCtInt = \realTorusCoverTCt$. If $\simpleRoot$ is of type II, $\rootReflection$ has a representative in $\realTorusCoverTCt\setminus\realTorusCoverTCtInt$ and
\begin{eqnarray*}
\left|\realTorusCoverTCt/\realTorusCoverTCtInt\right| & = & 2.
\end{eqnarray*}
In particular, $\realTorusCtCoverNci = \realTorusCtCoverInt$ if and only if $\simpleRoot$ is of type I.
\end{proposition}

\begin{definition}[\cite{IC1}, before Theorem 4.4]
\label{defCTChar}
Fix a genuine triple $\genTripleInf$ and write $\diff^{\simpleRoot} = \ctOpNci(\diff)$
for the image of $\diff$ in $\complexLieAlgebraCtNci$. Let $\genTorusCharCtInt$ be the irreducible representation of $\realTorusCtCoverInt$ satisfying
\begin{eqnarray*}
\genTorusCharCtInt|_{\realTorusCoverTCtInt} & = & \genTorusChar|_{\realTorusCoverTCtInt} \\
\genTorusCharCtInt|_{\realTorusCoverACt} & = & \expgc(\diff^{\simpleRoot}|_{{\mathfrak{a}}^{\simpleRoot}_{\mathbb{R}}}).
\end{eqnarray*}
According to Proposition \ref{propNciCtVgr}, we define an irreducible representation of $\realTorusCtCoverNci$ via
\begin{eqnarray*}
\genTorusCharCt & = & \left\{\begin{array}{ll}
\genTorusCharCtInt & \simpleRoot \text{ type I} \\
\text{Ind}_{\realTorusCtCoverInt}^{\realTorusCtCoverNci}\genTorusCharCtInt & \simpleRoot \text{ type II}
\end{array}
\right..
\end{eqnarray*}
\end{definition}

\begin{proposition}[\cite{IC1}, before Theorem 4.4]
In the setting of Definition \ref{defCTChar}, suppose $\simpleRoot$ is of type II. Then $\genTorusCharCt = \text{Ind}_{\realTorusCtCoverInt}^{\realTorusCtCoverNci}\genTorusCharCtInt$ is reducible if and only if the element $\rootReflection \in \realTorusCoverTCt$ centralizes $\realTorusCoverTCtInt$. In this case we write 
\begin{eqnarray*}
\genTorusCharCt & = & \genTorusCharCt_{+} \oplus \genTorusCharCt_{-}
\end{eqnarray*}
with $\genTorusCharCt_{\pm}$ irreducible. In particular, $\genTorusCharCt$ is always reducible if $\realTorusCoverTCt$ is abelian.
\end{proposition}

We now define the Cayley transform of a genuine triple through a simple noncompact imaginary root.

\begin{definition}[\cite{IC1}]
\label{defCTTriples}
Let $\gtRep = \genTripleInf$ and suppose $\simpleRoot \in \rootSystem(\complexLieAlgebra,\complexLieAlgebra[h])$ is noncompact and \emph{simple} for $\diff$. Define the \emph{Cayley transform of $\gtRep$ by $\simpleRoot$} to be
\begin{eqnarray*}
\ctOpNci(\gtRep) & = & \left[\begin{array}{ll}
\left\{\genTripleInfCtDiffChar{\diff^{\simpleRoot}}{\genTorusCharCt}{\simpleRoot}\right\} & \genTorusCharCt \text{ irreducible} \\
\left\{\genTripleInfCtDiffChar{\diff^{\simpleRoot}}{\genTorusCharCt_{+}}{\simpleRoot},~ \genTripleInfCtDiffChar{\diff^{\simpleRoot}}{\genTorusCharCt_{-}}{\simpleRoot}\right\} & \genTorusCharCt \text{ reducible}
\end{array}
\right..
\end{eqnarray*}
In particular, $\ctOpNci(\gtRep)$ is double valued if and only if $\simpleRoot$ is of type II and $\genTorusCharCt$ is reducible (Definition \ref{defCTChar}). If $\simpleRootb \in \rootSystem(\complexLieAlgebra, \absLieAlgebra)$ is an abstract noncompact simple root for $\gtRep$, we define the \emph{abstract Cayley transform of $\gtRep$ by $\simpleRootb$} to be $\ctOpNci[{\absConj(\simpleRootb)}](\gtRep)$.
\end{definition}

\begin{proposition}[\cite{RT3}, Lemma 6.14(g)]
\label{propCTLongSingleValued}
In the setting of Definition \ref{defCTTriples}, suppose $\simpleRoot$ is long. Then $\ctOpNci(\gtRep)$ is single valued and we have
\begin{eqnarray*}
\text{\emph{dim}}(\genTorusCharCt) & = & m\cdot\text{\emph{dim}}(\genTorusChar)
\end{eqnarray*}
where 
\begin{eqnarray*}
m & = & \left\{\begin{array}{ll}
1 & \simpleRoot \text{ type I} \\
2 & \simpleRoot \text{ type II}
\end{array}
\right..
\end{eqnarray*}
\end{proposition}

\begin{proof}
The fact that $\ctOpNci(\gtRep)$ is single valued is proven in \cite{RT3}. The statement about dimensions is an easy consequence of Proposition \ref{propNciCtVgr} and Definition \ref{defCTChar}.
\end{proof}

The following proposition implies the Cayley transform is well-behaved and descends to the level of $\hcBasisGen[\absDiff](p,q)$.

\begin{proposition}[\cite{IC1}]
The elements appearing in the definition of $\ctOpNci(\gtRep)$ are genuine triples for $\realGroupCover$. Moreover they have the same infinitesimal character as $\gtRep$ and are well defined up to $\maxRealCompactCover$-conjugacy.
\end{proposition}

\begin{remark}
For $\hcRep \in \hcBasisGen[\absDiff](p,q)$, denote the Cayley transform on the level of genuine parameters for $\realGroupCover$ by $\ctOpNci(\hcRep) \in \hcBasisGen[\absDiff](p,q)$. If $\genTripleInf$ is a genuine triple representing $\hcRep$, we will write
\begin{eqnarray*}
\text{dim}(\hcRep) & = & \text{dim}(\genTorusChar).
\end{eqnarray*}
\end{remark}

\begin{proposition}[\cite{RT3}, Proposition 6.12]
\label{propCrossNciTypeII}
Let $\hcRep \in \hcBasisGen[\absDiff](p,q)$ and suppose $\simpleRoot \in \imaginaryRoots(\complexLieAlgebra,\absLieAlgebra)(\absDiff)$ is an abstract integral noncompact imaginary root that is of type II for $\hcRep$. Then $\rootReflection \cross \hcRep =\hcRep$ if and only 
\begin{eqnarray*}
\text{\emph{dim}}(\ctOpNci(\hcRep)) & = & \text{\emph{dim}}(\hcRep).
\end{eqnarray*}
\end{proposition}

\begin{corollary}
\label{corCrossNciTypeII}
In the setting of Proposition \ref{propCrossNciTypeII}, suppose $\realGroupCover = \realSpinGroupCover{p}{q}$ has even rank and fix an evenly split Cartan subgroup $\realTorusESplitCover \subset \realGroupCover$ (Definition \ref{defMaxSplitEvenParity}). If $\hcRep$ has a representative genuine triple of the form $\genTripleESplitCtInf[\pSeq{j}{k}]$ (Definition \ref{defStdSeq}), then $\rootReflection \cross \hcRep \ne \hcRep$.
\end{corollary}

\begin{proof}
The existence of $\simpleRoot$ implies we must have $j < \qSpec$ (Theorem \ref{theoremFokko}). Propositions \ref{propGenRepBij}, \ref{propCompH}, and \ref{propZHjk} imply
\[
\text{dim}(\hcRep) = \left(\frac{2^{\numRealBits-1}}{2}\right)^{\frac{1}{2}} = \left(2^{\numRealBits-2}\right)^{\frac{1}{2}}.
\]
Similarly we have 
\begin{eqnarray*}
\text{dim}(\ctOpNci(\hcRep)) & = & \left(2^{\numRealBits}\right)^{\frac{1}{2}} \\
\frac{\text{dim}(\ctOpNci(\hcRep))}{\text{dim}(\hcRep)} & = & \left(\frac{2^{\numRealBits}}{2^{\numRealBits-2}}\right)^{\frac{1}{2}} = 4^{\frac{1}{2}} = 2
\end{eqnarray*}
and the result follows from Proposition \ref{propCrossNciTypeII}. 
\end{proof}

We will also need an inverse version of Definition \ref{defCTTriples}. This is most easily stated in terms of abstract roots. We refer the reader to \cite{RT3} or \cite{IC1} for more details.

\begin{definition}
\label{defInvCT}
Let $\hcRep \in \hcBasisGen[\absDiff](p,q)$ and suppose $\simpleRoot \in \realRoots(\complexLieAlgebra, \absLieAlgebra)$ is a simple abstract root that is real for $\hcRep$. Define the \emph{inverse Cayley transform of $\hcRep$ by $\simpleRoot$} to be
\begin{eqnarray*}
\ctOp(\hcRep) & = & \left\{\hcRep' \in \hcBasisGen[\absDiff](p,q) \mid \ctOpNci(\hcRep') = \hcRep\right\}.
\end{eqnarray*}
\end{definition}

More explicit definitions of inverse Cayley transforms appear in \cite{IC1} and \cite{VGr}. The following proposition gives a characterization of when the inverse Cayley transform is nontrivial.

\begin{proposition}[\cite{RT3}]
\label{propInvCTNonempty}
In the setting of Definition \ref{defInvCT}, $\ctOp(\hcRep)$ is nonempty if and only if $\simpleRoot$ satisfies the parity condition (Definition \ref{defPC}).
\end{proposition}


\subsection{Principal Classes}
\label{ssTheMapPs}
Fix a half-integral infinitesimal character $\absDiff \in \absLieAlgebraDual$ and suppose $\realGroupCover = \realSpinGroupCover{p}{q}$ has even rank and is not compact. Choose an evenly split Cartan subgroup $\realTorusESplitCover \subset \realGroupCover$ (Definition \ref{defMaxSplitEvenParity}) and let $\pSeq{j}{k}$ be a standard sequence (Definition \ref{defStdSeq}). If $\genPairESplitCtInf[\pSeq{j}{k}]$ is a genuine pair for which the corresponding abstract triple $\absTriple$ is supportable (Definition \ref{defSuppTriple}), Corollary \ref{corRepFiberCent} implies $\hcBasisGenInvOrderCentSpin{p}{q} = 2$ for some central character $\chi$ (Definition \ref{defCC}). In this section we use the element $\stCenter = \mac[{\simpleRooti[1]}]\mac[{\simpleRooti[2]}] \cdots \mac[{\simpleRooti[n/2]}]$ of Section \ref{ssEPCSG} to distinguish the elements in $\hcBasisGenInvCentSpin{p}{q}$.

To avoid trivialities, we begin by assuming $\realGroupCover = \realSpinGroupCover{p}{q}$ with $p > q \ge 2$. Let $\hcRep \in \hcBasisGenSpin{p}{q}$ and suppose $\genTripleESplitInfCtDiffChar{\varphi}{\genTorusChar}{}$ is a genuine triple representing $\hcRep$.  Since every short root for $\hcRep$ is either real, compact imaginary, or noncompact imaginary and of type II (Theorem \ref{theoremFokko}), $\hcRep$ has a representative genuine triple $\genTripleESplitCtInf[]$ with $(\diff, \ei) > 0$ for all $i$ (recall our choice of coordinates in Section \ref{ssEPCSG}). Such a representative will be called \emph{essentially dominant}.

\begin{definition}
\label{defPrincipalClass}
Two essentially dominant genuine triples $\genTripleESplitInfCtDiffChar{\diff_{1}}{\genTorusChar_{1}}{}$ and $\genTripleESplitInfCtDiffChar{\diff_{2}}{\genTorusChar_{2}}{}$ are said to be in the same \emph{principal class} if $\genTorusChar_{1}(\stCenter) = \genTorusChar_{2}(\stCenter)$. We denote the set of principal classes for $\realGroupCover$ with infinitesimal character $\absDiff$ by $\pc{p}{q}$ and with central character $\chi$ by $\pcCent{p}{q}$.
\end{definition}

\begin{proposition}
\label{propPC}
Suppose $\gtRep_{1} = \genTripleESplitInfCtDiffChar{\diff_{1}}{\genTorusChar_{1}}{}$ and $\gtRep_{2} = \genTripleESplitInfCtDiffChar{\diff_{2}}{\genTorusChar_{2}}{}$ are essentially dominant, $\maxRealCompactCover$-conjugate genuine triples. Then $\gtRep_{1}$ and $\gtRep_{2}$ are in the same principal class.
\end{proposition}

\begin{proof}
Essential dominance implies $\diff_{1}$ and $\diff_{2}$ differ only by reflections in long roots. Since $\gtRep_{1}$ and $\gtRep_{2}$ are $\maxRealCompactCover$-conjugate, these roots must be either real or compact imaginary. The result now follows from Proposition \ref{propNonPsCross} and Corollary \ref{corZCrossIm}.
\end{proof}

\begin{remark}
\label{remPC}
Proposition \ref{propPC} implies the notion of a principal class descends to the level of genuine parameters. Since a principal class does not see the corresponding $\maxRealCompactCover$-orbit, Corollary \ref{corRepFiberCent} implies $\pcCentOrder{p}{q} = 2$ whenever $\pcCentOrder{p}{q} \ne 0$. Note if $\realGroupCover$ is split, a principal class is the same as a principal series representation. 	
\end{remark}

In those cases where $\hcBasisGenInvOrderCentSpin{p}{q} = 2$, we will assign each element of $\hcBasisGenInvCentSpin{p}{q}$ a distinct principal class (Definition \ref{defPsMap}). This will complete the process of distinguishing the elements in $\hcBasisGenCentSpin{p}{q}$ and allow us to define the duality map $\weylGroupInvMap$ (Definition \ref{defPsiOnHcBasisGen}). We begin with a definition.

\begin{definition}
\label{defNumFlips}	
Given an abstract involution $\inv \in \weylGroup(\complexLieAlgebra, \absLieAlgebra)$, let $\numFlips$ denote \emph{half the number of abstract coordinates that are both interchanged and negated by $\inv$}. In terms of diagrams, $\numFlips$ is exactly half the number of parentheses appearing in $\diagram$ (Section \ref{ssAbsWeylGroup}).
\end{definition}

\begin{definition}
\label{defPsMap}
In the setting above, let $\pSeq{j}{k}$ be a standard sequence and suppose $\genPairESplitCtInf[\pSeq{j}{k}]$ is a genuine pair for which the corresponding abstract triple $\absTriple$ is supportable. The \emph{principal class map}
\[
\psMap : \hcBasisGenInvCentSpin{p}{q} \to \pcCent{p}{q}
\]
is defined as follows. For $\hcRep \in \hcBasisGenInvCentSpin{p}{q}$, choose a genuine triple $\gtRep = \genTripleESplitInfCtDiffChar{\diff}{\genTorusChar_{1}}{\pSeq{j}{k}}$ beginning with $\realTorusESplitCtCover[\pSeq{j}{k}]$ and representing $\hcRep$. Write $\genTorusChar_{1}^{\text{z}}$ for the unique genuine character of Z($\realTorusESplitCtCover[\pSeq{j}{k}]$) corresponding to $\genTorusChar_{1}$ (Proposition \ref{propGenRepBij}). Then there is a unique genuine character $\genTorusChar_{2}^{\text{z}}$ of Z($\realTorusESplitCover$) with differential $\ctOpSeqInv{j}{k}(\diff)$ and
\begin{eqnarray*}
\genTorusChar_{2}^{\text{z}}(\mac[\simpleRootc]) & = & \genTorusChar_{1}^{\text{z}}(\mac[\simpleRootc]) \\
\genTorusChar_{2}^{\text{z}}(\stCenter) & = & \genTorusChar_{1}^{\text{z}}(\stCenter)\genTorusChar_{1}^{\text{z}}(\mac[\simpleRootc])^{\numFlips}.
\end{eqnarray*}
Here $\mac[\simpleRootc]$ denotes the nontrivial central element of $\realGroupCover$ corresponding to the short real roots in $\rootSystem(\complexLieAlgebra, \eSplitLieAlgebra)$ (Proposition \ref{propSameMacReal}). Let $\genTorusChar_{2}$ be a genuine representation of $\realTorusESplitCover$ corresponding to $\genTorusChar_{2}^{\text{z}}$ and define 
\begin{eqnarray*}
\psMap(\genTorusChar_{1}) & = & \genTorusChar_{2} \\
\psMap(\gtRep) & = & \genTripleESplitInfCtDiffChar{\ctOpSeqInv{j}{k}(\diff)}{\genTorusChar_{2}}{} \\
\psMap(\hcRep) & = & \left[\genTripleESplitInfCtDiffChar{\ctOpSeqInv{j}{k}(\diff)}{\genTorusChar_{2}}{}\right] \in \pc{p}{q}.
\end{eqnarray*}
In other words, $\psMap(\hcRep)$ is defined to be the principal class represented by the genuine triple $\genTripleESplitInfCtDiffChar{\ctOpSeqInv{j}{k}(\diff)}{\genTorusChar_{2}}{}$. In particular, $\hcRep$ and $\psMap(\hcRep)$ have the same central characters (Remark \ref{remCompareCC}).
\end{definition}

\begin{proposition}
\label{propPsMapWellDefined}
The map $\psMap$ is well defined.
\end{proposition}

\begin{proof}
The issue is the choice of genuine triple $\gtRep = \genTripleESplitInfCtDiffChar{\diff}{\genTorusChar_{1}}{\pSeq{j}{k}}$ representing $\hcRep$. It suffices to show $\psMap(\gtRep)$ is equivalent to $\psMap(\weylElt\cdot\gtRep)$ with
\begin{eqnarray*}
\weylElt\cdot\gtRep & = & \genTripleESplitInfCtDiffChar{\weylElt\cdot\diff}{\weylElt\cdot\genTorusChar_{1}}{\pSeq{j}{k}}
\end{eqnarray*}
and $\weylElt \in \realNormalizer[\maxRealCompactCover](\realTorusESplitCtCover[\pSeq{j}{k}]) / Z_{\maxRealCompactCover}(\realTorusESplitCtCover[\pSeq{j}{k}]) = \cartanWeylGroup[\realGroupCover]{\realTorusESplitCtCover[\pSeq{j}{k}]}$. Note $\weylElt\cdot\genTorusChar_{1}$ is defined only up to $\maxRealCompactCover$-conjugacy, however the corresponding character of Z($\realTorusESplitCtCover[\pSeq{j}{k}]$) is well defined. Since it is the character that matters for Definition \ref{defPsMap}, this level of precision is sufficient for our purposes.

As noted above, the differentials of $\psMap(\gtRep)$ and $\psMap(\weylElt\cdot\gtRep)$ will have essentially dominant representatives in $\eSplitLieAlgebraDual$. However any genuine representation of $\realTorusESplitCover$ with a fixed differential is determined by its values on $\mac[\simpleRootc]$ and $\stCenter$. Since conjugation and cross action cannot affect $\mac[\simpleRootc]$, we simply need to understand their effect on $\stCenter$. Moreover, conjugating by $\weylElt$ does not change the value of $\numFlips$, so we may effectively ignore the $\genTorusChar_{1}^{\text{z}}(\mac[\simpleRootc])^{\numFlips}$ term in the definition of $\genTorusChar_{2}^{\text{z}}(\stCenter)$. 

Proposition \ref{propCWGroup} implies
\begin{eqnarray*}
\cartanWeylGroup[\realGroupCover]{\realTorusESplitCtCover[\pSeq{j}{k}]} & \cong & \complexWeylGroup \ltimes ((A \ltimes \cptImaginaryWeylGroup) \times \realWeylGroup).
\end{eqnarray*}
Our proof is by cases based on the element $\weylElt$. 

\begin{steplist}
\item[Case I] Suppose $\weylElt \in \realWeylGroup$. Then conjugation commutes with $\ctOpSeqInv{j}{k}$ and $\psMap(\gtRep)$ is conjugate to $\psMap(\weylElt \cdot \gtRep)$. 

\item[Case II] Suppose $\weylElt \in \cptImaginaryWeylGroup$. We may assume $\weylElt = \rootReflection$ for $\simpleRoot \in \imaginaryRoots[\cinv](\complexLieAlgebra, \eSplitLieAlgebraCt[\pSeq{j}{k}])(\diff)$ imaginary and compact. Write $\ctOpRoot = \ctOpSeqInv{j}{k}(\simpleRoot)$ for the image of $\simpleRoot$ in $\rootSystem(\complexLieAlgebra,\eSplitLieAlgebra)$. If $\ctOpRoot$ remains imaginary and compact in $\rootSystem(\complexLieAlgebra,\eSplitLieAlgebra)$, then conjugation commutes with $\ctOpSeqInv{j}{k}$ and we are done. 

If $\ctOpRoot$ is short and real in $\rootSystem(\complexLieAlgebra,\eSplitLieAlgebra)$, set $m=(\diff,\rootCheck) \in \mathbb{Z}$. Proposition \ref{propPsConj} implies 
\begin{eqnarray*}
(\rootReflection[\ctOpRoot]\cdot\psMap(\genTorusChar_{1}))(\stCenter) & = & \psMap(\genTorusChar_{1})(\mac[\simpleRootc])\psMap(\genTorusChar_{1})(\stCenter)
\end{eqnarray*}
and Corollary \ref{corZCrossIm} gives
\begin{eqnarray*}
\rootReflection\cdot\genTorusChar_{1}(\stCenter) & = & (-1)^{(m+1)}\genTorusChar_{1}(\stCenter).
\end{eqnarray*}
Therefore $\psMap(\gtRep)$ is conjugate to $\psMap(\rootReflection\cdot\gtRep)$ if and only if
\begin{eqnarray*}
\psMap(\genTorusChar_{1})(\mac[\simpleRootc]) & = & (-1)^{(m+1)}\cdot\identity
\end{eqnarray*}
or if and only if $\ctOpRoot$ satisfies the parity condition  for $\psMap(\gtRep)$. Since $\simpleRoot$ is assumed to compact, this follows from Theorem \ref{theoremCompareGradings}.

If $\ctOpRoot$ is long and real in $\rootSystem(\complexLieAlgebra,\eSplitLieAlgebra)$, Proposition \ref{propPsConj} implies $(\rootReflection[\ctOpRoot]\cdot\psMap(\genTorusChar_{1}))(\stCenter) = \psMap(\genTorusChar_{1})(\stCenter)$. In particular, $\psMap(\gtRep)$ is conjugate to $\psMap(\rootReflection\cdot\gtRep)$ if and only if $\genTorusChar_{1}(\stCenter) = \rootReflection\cdot\genTorusChar_{1}(\stCenter)$ and this follows from Proposition \ref{propCrossCpt} and Corollary \ref{corZCrossIm}.

Finally if $\ctOpRoot$ is long and complex in $\rootSystem(\complexLieAlgebra,\eSplitLieAlgebra)$, then the $\maxRealCompactCover$-orbits for $\psMap(\gtRep)$ and $\psMap(\weylElt\cdot\gtRep)$ differ by cross action in $\ctOpRoot$ (Remark \ref{remAbsRegCrossAction}). If $\rootReflection[\ctOpRoot]$ preserves the (fixed) set of short positive roots, then the same sequence of short reflections will produce essentially dominant representatives for $\psMap(\gtRep)$ and $\psMap(\rootReflection\cdot\gtRep)$. Since $\genTorusChar_{1}(\stCenter) = \rootReflection\cdot\genTorusChar_{1}(\stCenter)$, $\psMap(\gtRep)$ and $\psMap(\rootReflection\cdot\gtRep)$ represent the same principal class. Otherwise, $\rootReflection[\ctOpRoot]$ will negate (and interchange) a short imaginary root and a short real root. If the imaginary root is compact, the real root will satisfy the parity condition (Theorem \ref{theoremCompareGradings}) and the composition of the two reflections changes $\psMap(\rootReflection \cdot \genTorusChar_{1})(\stCenter)$ by
\[
(-1)^{m+1}\psMap(\genTorusChar_{1})(\mac[\simpleRootc]) = (-1)^{m+1}(-1)^{m+1} = 1
\]
(for some integer $m$). If the imaginary root is noncompact, then the real root will not satisfy the parity condition and the composition of the two reflections changes $\psMap(\rootReflection \cdot \genTorusChar_{1})(\stCenter)$ by a factor of
\[
(-1)^{m}\psMap(\genTorusChar_{1})(\mac[\simpleRootc]) = (-1)^{m}(-1)^{m} = 1.
\]
Therefore $\psMap(\gtRep)$ and $\psMap(\rootReflection\cdot\gtRep)$ represent the same principal class as in the previous case.

\item[Case III] Suppose $\weylElt \in A$. Then 
\[
\weylElt \in \imaginaryWeylGroup[\cinv](\complexLieAlgebra, \eSplitLieAlgebraCt[\pSeq{j}{k}]) \cong W(B_{\numImaginaryBits[\cinv]}) \times W(A_{1})^{k}
\]
(Proposition \ref{propInvWeylGroups}) and $\weylElt$ can be expressed as a product of orthogonal reflections in noncompact imaginary roots (\cite{IC4}, Corollary 5.14). 

If $j < \qSpec$, the group $A$ is generated by reflections in noncompact imaginary roots of type II. Suppose first that $\simpleRoot$ is long and let $\weylElt = \rootReflection$. Then $\simpleRoot$ is an element of the $(A_{1})^{k}$ factor of $\imaginaryRoots[\cinv](\complexLieAlgebra, \eSplitLieAlgebraCt[\pSeq{j}{k}])$ and we must have $\simpleRoot = \ctOpSeq{j}{k}(\simpleRootai)$ for some $i$. In particular, there exists a Cartan subgroup $\realTorusESplitCtCover[\pSeq{}{}] = \ctOpNci(\realTorusESplitCtCover[\pSeq{j}{k}])$ containing $\stCenter$ and a representative of $\rootReflection$ (Proposition \ref{propNciCtVgr}). Since $\stCenter$ is central in $\realTorusESplitCtCover[\pSeq{}{}]$, $\stCenter$ commutes with $\rootReflection$ and we have $\genTorusChar_{1}(\stCenter) = \rootReflection\cdot\genTorusChar_{1}(\stCenter)$. The result now follows as in Case II with $\ctOpRoot$ long and real.

Suppose $\simpleRoot$ is short and of type II. Corollary \ref{corZCrossIm} gives
\begin{eqnarray*}
\rootReflection \cross \genTorusChar_{1}(\stCenter) & = & (-1)^{(m+1)}\genTorusChar_{1}(\stCenter)
\end{eqnarray*}
and Corollary \ref{corCrossNciTypeII} implies $\rootReflection \cross \hcRep \ne \hcRep$. Therefore we must have
\begin{eqnarray*}
\rootReflection \cdot \genTorusChar_{1}(\stCenter) & = & (-1)^{(m)}\genTorusChar_{1}(\stCenter)
\end{eqnarray*}
by Corollary \ref{corDiffSignZ}. In particular, $\psMap(\gtRep)$ and $\psMap(\rootReflection\cdot\gtRep)$ are conjugate if and only if
$\ctOpRoot$ does not satisfy the parity condition for $\psMap(\gtRep)$. Since $\simpleRoot$ is assumed to be noncompact, this follows from Theorem \ref{theoremCompareGradings}.

Finally let $j = \qSpec$ and $\weylElt = \rootReflection[{\simpleRoote}]\rootReflection[{\simpleRootd}]$, where $\simpleRootd$ and $\simpleRoote$ are orthogonal imaginary roots in $\imaginaryRoots[\cinv](\complexLieAlgebra, \eSplitLieAlgebraCt[\pSeq{j}{k}])$ that are noncompact and of type I. Set $m_{1} = (\diff, \rootCheck[\simpleRootd])$, $m_{2} = (\diff, \rootCheck[\simpleRoote])$, and suppose first that $\simpleRootd$ and $\simpleRoote$ are long. Then $m_{1},m_{2} \in \mathbb{Z} + \frac{1}{2}$ and Corollary \ref{corZConjIm} implies
\begin{eqnarray*}
\rootReflection[\simpleRootd]\cdot\genTorusChar_{1}(\stCenter) & = & (-1) \cdot \genTorusChar_{1}(\stCenter) \\
\rootReflection[\simpleRoote]\cdot\genTorusChar_{1}(\stCenter) & = & (-1) \cdot \genTorusChar_{1}(\stCenter) \\
\rootReflection[\simpleRoote]\rootReflection[\simpleRootd]\cdot\genTorusChar_{1}(\stCenter) & = & (-1)^{2}\cdot\genTorusChar_{1}(\stCenter) = \genTorusChar_{1}(\stCenter).
\end{eqnarray*}
as desired. If $\simpleRootd$ is long and $\simpleRoote$ is short, then $m_{1} \in \mathbb{Z} + \frac{1}{2}$ and $m_{2} \in \mathbb{Z}$. Corollary \ref{corZConjIm} implies
\begin{eqnarray*}
\rootReflection[\simpleRoote]\cdot\genTorusChar_{1}(\stCenter) & = & (-1)^{m_{2}+1}\cdot\genTorusChar_{1}(\stCenter) \\
\rootReflection[\simpleRoote]\rootReflection[\simpleRootd]\cdot\genTorusChar_{1}(\stCenter) & = & (-1)^{m_{2}}\cdot\genTorusChar_{1}(\stCenter).
\end{eqnarray*}
This is the desired result since $\simpleRoote$ is noncompact. Finally if both $\simpleRootd$ and $\simpleRoote$ are short we have
\begin{eqnarray*}
\rootReflection[\simpleRootd]\cdot\genTorusChar_{1}(\stCenter) & = & (-1)^{m_{1}+1}\cdot\genTorusChar_{1}(\stCenter) \\
\rootReflection[\simpleRoote]\cdot\genTorusChar_{1}(\stCenter) & = & (-1)^{m_{2}+1}\cdot\genTorusChar_{1}(\stCenter) \\
\rootReflection[\simpleRoote]\rootReflection[\simpleRootd]\cdot\genTorusChar_{1}(\stCenter) & = & (-1)^{m_{1}+m_{2}}\cdot\genTorusChar_{1}(\stCenter) \\
& = & \genTorusChar_{1}(\stCenter)
\end{eqnarray*}
since $m_{1} + m_{2} \in 2\mathbb{Z}$. The result follows from Proposition \ref{propNonPsCross}.

\item[Case IV]  Suppose $\weylElt \in \complexWeylGroup$. This case is handled in the same fashion as the previous cases. The reader is spared the details. \qedhere
\end{steplist}
\end{proof}

\begin{theorem}
\label{theoremPsMapBij}
In the setting of Definition \ref{defPsMap}, the map
\[
\hcBasisGenInvCentSpin{p}{q} \to \pcCent{p}{q}
\]
is a bijection.
\end{theorem}

\begin{proof}
Corollary \ref{corRepFiberCent} implies $\psMap$ is, at worst, 2 to 1. In particular, if $\omega,\hcRep \in \hcBasisGenInvCentSpin{p}{q}$, it remains to show $\psMap(\omega) \ne \psMap(\hcRep)$.
Recall $\inv$ is a supportable abstract involution for the genuine pair $\genPairESplitCtInf[\pSeq{j}{k}]$ and let $\imaginaryBitsBit$, $\realBitsBit$, $\imaginaryBitsParityBit = \realBitsParityBit=0$ be the corresponding indicator bits (Definition \ref{defInvParams}). The proof is by cases for the standard sequence $\pSeq{j}{k}$.
\begin{steplist}
\item[Case I]  Suppose $j<\qSpec$ and $k=0$ so that $\realBitsBit = 1$. Corollary \ref{corNumGenRepsFormula} implies
\begin{eqnarray*}
\numCorGenTriplesESplitCtInf[{\genPairESplitCtInf[\pSeq{j}{k}]}] & = & 2^{1-\imaginaryBitsBit}2^{\realBitsBit(1-\realBitsParityBit)} \\
& = & 2\cdot2^{1-\imaginaryBitsBit}.
\end{eqnarray*}
Therefore each element of $\hcBasisGenInvCentSpin{p}{q}$ has a representative beginning with $\genPairESplitCtInf[\pSeq{j}{k}]$ and the result follows from Corollary \ref{corDiffSignZ}.

\item[Case II]  Suppose $j<\qSpec$ and $k>0$ so that $\imaginaryBitsBit = \realBitsBit = 1$. Corollary \ref{corNumGenRepsFormula} implies
\begin{eqnarray*}
\numCorGenTriplesESplitCtInf[{\genPairESplitCtInf[\pSeq{j}{k}]}] & = & 2^{1-\imaginaryBitsBit}2^{\realBitsBit(1-\realBitsParityBit)} = 2
\end{eqnarray*}
and there are two $\maxRealCompactCover$-orbits in the genuine fiber $\genFiber$ of $\inv$ (Section \ref{ssNLFibers}). Since $\mac[\simpleRootc] \in \realTorusESplitCtCoverId[\pSeq{j}{k}]$ by Proposition \ref{propZHjk}, both genuine triples extending $\genPairESplitCtInf[\pSeq{j}{k}]$ have the same central character and the result follows by Corollary \ref{corDiffSignZ}.

\item[Case III]  Suppose $j=\qSpec$ and $k>0$ so that $\imaginaryBitsBit=1$ and $\realBitsBit = 0$. Corollary \ref{corNumGenRepsFormula} implies
\begin{eqnarray*}
\numCorGenTriplesESplitCtInf[{\genPairESplitCtInf[\pSeq{j}{k}]}] & = & 2^{1-\imaginaryBitsBit}2^{\realBitsBit(1-\realBitsParityBit)} = 1.
\end{eqnarray*}
Then $\mac[\simpleRootc] \in \realTorusESplitCtCoverId[\pSeq{j}{k}] = \realTorusESplitCtCover[\pSeq{j}{k}]$ by Proposition \ref{propZHjk} and there are two $\maxRealCompactCover$-orbits in $\genFiber$ with a fixed abstract grading and central character (Proposition \ref{propCCNci}). Theorem \ref{theoremKOrbits} implies these orbits differ by cross action in any short noncompact root. Let $\gtRep = \genTripleESplitCtInf[\pSeq{j}{k}]$ be a genuine triple extending $\genPairESplitCtInf[\pSeq{j}{k}]$ and let $\simpleRoot \in \rootSystem(\complexLieAlgebra,\absLieAlgebra)$ be a short noncompact root. Set $m = (\absDiff, \rootCheck)$ and write $\ctOpRoot$ for the image of $\simpleRoot$ in $\rootSystem(\complexLieAlgebra,\eSplitLieAlgebra)$. Corollary \ref{corZCrossIm} and Proposition \ref{propPsConj} imply
\begin{eqnarray*}
\rootReflection \cross \genTorusChar(\stCenter) & = & (-1)^{m+1}\genTorusChar(\stCenter) \\
\rootReflection[\ctOpRoot]\cdot\psMap(\genTorusChar)(\stCenter) & = & \psMap(\genTorusChar)(\mac[{\ctOpRoot}])\psMap(\genTorusChar)(\stCenter).
\end{eqnarray*}
In particular, $\psMap(\gtRep) \ne \psMap(\rootReflection \cross \gtRep)$ if and only if $\simpleRoot$ does not satisfy the parity condition (Definition \ref{defPC}) for $\psMap(\gtRep)$. Since $\simpleRoot$ is noncompact, this follows by Theorem \ref{theoremCompareGradings}.

\item[Case IV]  Suppose $j=\qSpec$ and $k=0$ so that $\realBitsBit=0$. Corollary \ref{corNumGenRepsFormula} implies
\begin{eqnarray*}
\numCorGenTriplesESplitCtInf[{\genPairESplitCtInf[\pSeq{j}{k}]}] & = & 2^{1-\imaginaryBitsBit}2^{\realBitsBit(1-\realBitsParityBit)} \\
& = & 2^{1-\imaginaryBitsBit}
\end{eqnarray*}
and $\stCenter \in \realTorusESplitCtCoverId[\pSeq{j}{k}]$. Let $\gtRep = \genTripleESplitCtInf$ be a genuine triple extending $\genPairESplitCtInf[\pSeq{j}{k}]$ and suppose $\simpleRoot$ is a long noncompact imaginary root in $\rootSystem(\complexLieAlgebra, \eSplitLieAlgebraCt[\pSeq{j}{k}])$. Then $\genTripleESplitCtInf$ and $\genTripleESplitInfCtDiffChar{\rootReflection\cdot\diff}{\rootReflection\cdot\genTorusChar}{\pSeq{j}{k}}$ are \emph{not} conjugate in $\maxRealCompactCover$. If $m = (\diff, \rootCheck)$, then $m \in \mathbb{Z}+\frac{1}{2}$ and Corollary \ref{corZConjIm} implies
\begin{eqnarray*}
\rootReflection\cdot\genTorusChar(\stCenter) & = & (-1)^{2m} \cdot \genTorusChar(\stCenter) \\
& = & (-1) \cdot \genTorusChar(\stCenter).
\end{eqnarray*}
Since $\simpleRoot$ is long, $\psMap(\gtRep)$ and $\psMap(\rootReflection \cdot \gtRep)$ are not conjugate in $\maxRealCompactCover$ by Proposition \ref{propPsConj}. \qedhere
\end{steplist}
\end{proof}


\subsection{Definition of $\weylGroupInvMap$}
Let $\realGroupCover = \realSpinGroup{p}{q}$ with $p + q = 2n+1$, $p > q$, and recall the rank of $\realGroupCover$ is even. Fix a half-integral infinitesimal character $\absDiff \in \absLieAlgebraDual$ and an evenly split Cartan subgroup $\realTorusESplitCover \subset \realGroupCover$. Suppose $\genTripleESplitInf$ is a genuine triple for $\realGroupCover$ with central character $\chi$ (Definition \ref{defCC}) and let $\realGroupCoverDual = \realSpinGroupCover{p^{\vee}}{q^{\vee}}$ be the dual group (Definition \ref{defDualGroup}) for $\realGroupCover$ and $\genTripleESplitInf$. Define maps $\pcMap$, $\pcMapDual$ for $\realGroupCover$ and $\realGroupCoverDual$ respectively (Definition \ref{defPsMap}) and choose a bijection 
\[
\pcInvMap : \pcCent{p}{q} \to \pcCenter{p^{\vee}}{q^{\vee}}{\chi^{\vee}}
\]
on the corresponding principal classes (Definition \ref{defPrincipalClass}). Let $\inv \in \weylGroup(\complexLieAlgebra, \absLieAlgebra)$ be an involution and suppose $\hcBasisGenInvOrderCentSpin{p}{q} \ne 0$. Corollary \ref{corRepFiberCent} and Theorem \ref{theoremNumericalDuality} imply
\[
\hcBasisGenInvOrderCentSpin{p}{q} = \hcBasisGenInvOrderCentSpinDual{p^{\vee}}{q^{\vee}} \in \left\{1,2\right\}.
\]
The following definition extends $\weylGroupInvMap$ to the level of genuine parameters for $\realGroupCover$ and $\realGroupCoverDual$ (see also Remark \ref{remDefPsMap}).

\begin{definition}
\label{defPsiOnHcBasisGen}
In the above setting, let $\hcRep \in \hcBasisGenInvCentSpin{p}{q}$ and suppose $\absBg$ is the corresponding abstract bigrading. If $\hcBasisGenInvOrderCentSpin{p}{q} = 1$, define $\weylGroupInvMap(\hcRep) = \omega$, where $\omega$ is the unique element in $\hcBasisGenInvCentSpinDual{p}{q}$. If $\hcBasisGenInvOrderCentSpin{p}{q} = 2$, define $\weylGroupInvMap(\hcRep) = \omega$, where $\omega$ is the unique element in $\hcBasisGenInvCentSpinDual{p}{q}$ with $\psInvMap(\psMap(\hcRep)) = \pcMapDual(\omega)$ (Theorem \ref{theoremPsMapBij}). In either case the abstract bigrading of $\weylGroupInvMap(\hcRep)$ is $\absBgDual$ (Proposition \ref{propSuppRepsDual}).
\end{definition}

\begin{remark}
\label{remDefPsMap}
Definition \ref{defPsiOnHcBasisGen} is technically incomplete when $q$ or $q^{\vee}$ is less than or equal to one. If $q = 1$, then $\realSpinGroupCover{p}{q}$ has representations corresponding to an even parity Cartan subgroup if and only if $\realSpinGroupCover{p}{q}$ has discrete series representations. In this case it has exactly two of them (Theorem \ref{theoremFokko}) and we define these to be the two principal classes for $\realGroupCover$. If $q = 0$, then $\realSpinGroupCover{p}{q}$ is compact and $\hcBasisGenOrderCentSpin{p}{q} \in \left\{0,1\right\}$. To simplify the following formalism, we'd like to have $\hcBasisGenOrderCentSpin{p}{q} \in \left\{0,2\right\}$ for any half-integral $\absDiff$. Therefore if $\hcBasisGenOrderCentSpin{p}{q} \ne 0$, we create a formal `copy' of the element in $\hcBasisGenCentSpin{p}{q}$ and consider the two corresponding elements to be the principal classes for $\realSpinGroupCover{p}{q}$.
\end{remark}

We now prove $\weylGroupInvMap$ has the desired properties with respect to the operations of the KLV-algorithm.

\begin{theorem}
\label{theoremDualityCommutesWithCA}
In the setting of Definition \ref{defPsiOnHcBasisGen}, let $\simpleRoot \in \rootSystem(\complexLieAlgebra, \absLieAlgebra)(\absDiff)$ be an abstract integral root. If $\hcRep \in \hcBasisGenCentSpin{p}{q}$, then
\begin{eqnarray*}
\weylGroupInvMap(\rootReflection \cross \hcRep) & = & \rootReflection \cross \weylGroupInvMap(\hcRep).
\end{eqnarray*}
\end{theorem}

\begin{proof}
Let $\absBg$ be the abstract bigrading for $\hcRep$. We have
\begin{eqnarray*}
-(\rootReflection\inv\rootReflection^{-1}) & = & \rootReflection(-\inv)\rootReflection^{-1}
\end{eqnarray*}
and the result holds on the level of involutions (Proposition \ref{propCATriples}). If $\hcBasisGenInvOrderCentSpin{p}{q} = 1$, the elements $\hcRep$ and $\weylGroupInvMap(\hcRep)$ are determined by their involutions and we are done. 

If $\hcBasisGenInvOrderCentSpin{p}{q} = 2$, we need to show $\pcInvMap(\pcMap(\rootReflection\cross\hcRep)) = \pcMapDual(\rootReflection \cross \weylGroupInvMap(\hcRep))$. Since $\pcInvMap(\pcMap(\hcRep)) = \pcMapDual(\weylGroupInvMap(\hcRep))$ by definition, it suffices to show
\[
\psMap(\rootReflection\cross\hcRep) = \psMap(\hcRep) \iff \psMap(\rootReflection \cross \weylGroupInvMap(\hcRep)) = \psMap(\weylGroupInvMap(\hcRep)).
\]
The proof is by cases (for a change) based on the type of $\simpleRoot$ for $\hcRep$.
\begin{steplist}
\item[Case I]
Suppose $\simpleRoot$ is imaginary and compact for $\hcRep$. Then $\simpleRoot$ is real for $\weylGroupInvMap(\hcRep)$ and does not satisfy the parity condition. Therefore,
\begin{eqnarray*}
\rootReflection\cross\hcRep & = & \hcRep \\
\rootReflection \cross \weylGroupInvMap(\hcRep) & = & \weylGroupInvMap(\hcRep)
\end{eqnarray*}
and we have 
\begin{eqnarray*}
\psMap(\rootReflection\cross\hcRep) & = & \psMap(\hcRep) \\
\psMap(\rootReflection \cross \weylGroupInvMap(\hcRep)) & = & \psMap(\weylGroupInvMap(\hcRep))
\end{eqnarray*}
by Proposition \ref{propPsMapWellDefined}.

\item[Case II]
Suppose $\simpleRoot$ is imaginary and noncompact for $\hcRep$ . Proposition \ref{propSuppRep} implies $\simpleRoot$ is short and we know $\simpleRoot$ is real for $\weylGroupInvMap(\hcRep)$ and satisfies the parity condition. Then
\begin{eqnarray*}
\rootReflection\cross\hcRep & \ne & \hcRep \\
\rootReflection \cross \weylGroupInvMap(\hcRep) & \ne & \weylGroupInvMap(\hcRep)
\end{eqnarray*}
by Corollaries \ref{corCrossNciTypeII} and \ref{corZCrossReal}. Therefore
\begin{eqnarray*}
\psMap(\rootReflection\cross\hcRep) & \ne & \psMap(\hcRep) \\
\psMap(\rootReflection \cross \weylGroupInvMap(\hcRep)) & \ne & \psMap(\weylGroupInvMap(\hcRep))
\end{eqnarray*}
by Theorem \ref{theoremPsMapBij}.

\item[Case III]
Suppose $\simpleRoot$ is short and complex for $\hcRep$. Choose a genuine triple $\genTripleESplitCtInf[\pSeq{j}{k}]$ representing $\hcRep$ and recall the number $\numFlips$ from Definition \ref{defNumFlips}. It is easy to check 
\begin{eqnarray*}
\numFlips[{\rootReflection \cross \inv}] & = & \numFlips \pm 1.
\end{eqnarray*}
If $m = (\absDiff, \rootCheck)$, Corollary \ref{corZCrossCmplx} implies
\begin{eqnarray*}
\psMap(\rootReflection \cross \hcRep) = \psMap(\hcRep) & \iff & \\
(-1)^{m}\genTorusChar(\stCenter)\genTorusChar(\mac[\simpleRootc])^{\numFlips + 1} = \genTorusChar(\stCenter)\genTorusChar(\mac[\simpleRootc])^{\numFlips}\genTorusChar(\mac[\simpleRootc]) & \iff & \\
(-1)^{m} = 1.
\end{eqnarray*}
By the same argument we also have $\psMap(\rootReflection \cross \weylGroupInvMap(\hcRep)) = \psMap(\weylGroupInvMap(\hcRep)) \iff (-1)^{m} = 1$ and the result follows.

\item[Case IV]
Suppose $\simpleRoot$ is long and complex for $\hcRep$. This is handled in the same fashion as Case III and is left to the reader.
\end{steplist}
\end{proof}

\begin{theorem}
\label{theoremDualityCommutesWithCT}
Let $\hcRep \in \hcBasisGenCentSpin{p}{q}$ and suppose $\simpleRoot \in \rootSystem(\complexLieAlgebra, \absLieAlgebra)$ is an abstract simple root that is imaginary and noncompact. Then $\simpleRoot$ is real for $\weylGroupInvMap(\hcRep)$ and we have
\begin{eqnarray*}
\weylGroupInvMap(\ctOpNci(\hcRep)) & = & \ctOp(\weylGroupInvMap(\hcRep)).
\end{eqnarray*}
Note this is an equality of sets (Definitions \ref{defCTTriples} and \ref{defInvCT}).
\end{theorem}

\begin{proof}
Let $\absBg$ be the abstract bigrading for $\hcRep$. Since $\absBgDual$ is the abstract bigrading for $\weylGroupInvMap(\hcRep)$, $\simpleRoot$ is real for $\weylGroupInvMap(\hcRep)$ and satisfies the parity condition.
\begin{steplist}
\item[Case I]
Suppose $\hcBasisGenInvOrderCentSpin{p}{q} = 1$ and $\simpleRoot$ is long. Then $\ctOpNci(\hcRep)$ must be the unique element of $\hcBasisGenInvCentSpinInv{p}{q}{\rootReflection\inv}$. Similarly, $\ctOp(\weylGroupInvMap(\hcRep))$ is the unique element of $\hcBasisGenInvCentSpinInv{p}{q}{\rootReflection(-\inv)} = \hcBasisGenInvCentSpinInv{p}{q}{-\rootReflection(\inv)}$ and the result follows.

\item[Case II]
Suppose $\hcBasisGenInvOrderCentSpin{p}{q} = 1$ and $\simpleRoot$ is short. Using Theorem \ref{theoremFokko}, it is easy to check $\simpleRoot$ must be of type II. Since $\hcBasisGenInvOrderCentSpin{p}{q} = 1$, we must have $\rootReflection \cross \hcRep = \hcRep$ and Proposition \ref{propCrossNciTypeII} implies 
\begin{eqnarray*}
\ctOpNci(\hcRep) & = & \left\{\omega_{+}, \omega_{-}\right\} \subset \hcBasisGenInvCentSpinInv{p}{q}{\rootReflection(\inv)}
\end{eqnarray*}
is double valued. In particular, $\weylGroupInvMap(\omega_{+})$ and $\weylGroupInvMap(\omega_{-})$ are exactly the elements in $\hcBasisGenInvCentSpinInv{p}{q}{\rootReflection(-\inv)}$ with the same central character as $\weylGroupInvMap(\hcRep)$. Therefore it suffices to show $\ctOp(\weylGroupInvMap(\hcRep))$ is double valued and this follows immediately from Corollary \ref{corCrossNciTypeII}.

\item[Case III]
Suppose $\hcBasisGenInvOrderCentSpin{p}{q} = 2$ and $\simpleRoot$ is short. Then both $\ctOpNci(\hcRep)$ and $\ctOp(\weylGroupInvMap(\hcRep))$ are single valued and the result follows for the same reasons as in Case I.

\item[Case IV]
Suppose $\hcBasisGenInvOrderCentSpin{p}{q} = 2$ and $\simpleRoot$ is long. Choose a representative $\gtRep = \genTripleESplitCtInf[\pSeq{j}{k}]$ for $\hcRep$ with $\realTorusESplitCtCover[\pSeq{j}{k}] \subset \realGroupCover$ an even parity Cartan subgroup. Proposition \ref{propCTLongSingleValued} implies $\ctOpNci(\hcRep)$ is single-valued and we have
\begin{eqnarray*}
\imaginaryWeylGroup & \cong & W(B_{\numImaginaryBits}) \times W(A_{1})^{l} \\
\cartanWeylGroup{\realTorusESplitCtCover[\pSeq{j}{k}]} & \cong & ((A \ltimes \cptImaginaryWeylGroup) \times \realWeylGroup) \rtimes \complexWeylGroup.
\end{eqnarray*}
Suppose first that $\simpleRoot$ is an element of the $A_{1}^{l}$ factor of $\imaginaryRoots(\complexLieAlgebra, \absLieAlgebra)$. Then $j > k$ and conjugation in $\complexWeylGroup$ allows us to choose the representative $\gtRep$ such that
\[
\ctOpRoot = \absConj(\simpleRoot) = \ctOpSeq{j}{k}(\pm\simpleRooti[j]).
\]
In particular, $\ctOpNci[\ctOpRoot](\realTorusESplitCtCover[\pSeq{j}{k}]) = \realTorusESplitCtCover[\pSeq{j-1}{k}]$ is an even parity Cartan subgroup of $\realGroupCover$. Since $\simpleRoot$ is long, Definition \ref{defCTChar} directly implies
\begin{eqnarray*}
\psMap(\gtRep) & = & \psMap(\ctOpNci[\ctOpRoot](\gtRep))
\end{eqnarray*}
so that 
\begin{eqnarray*}
\psMap(\hcRep) & = & \psMap(\ctOpNci(\hcRep)).
\end{eqnarray*}
A similar argument holds for $\weylGroupInvMap(\hcRep)$ and ultimately gives
\begin{eqnarray*}
\psMap(\weylGroupInvMap(\hcRep)) & = & \psMap(\ctOp(\weylGroupInvMap(\hcRep)))
\end{eqnarray*}
and the result follows.

Now suppose $\simpleRoot$ is an element of the $W(B_{\numImaginaryBits})$ factor of $\imaginaryRoots(\complexLieAlgebra, \absLieAlgebra)$. Then $k > 0$ and conjugation in $\cptImaginaryWeylGroup$ allows us to choose the representative $\gtRep$ such that
\[
\ctOpRoot = \absConj(\simpleRoot) = \ctOpSeq{j}{k}(\pm\simpleRooti[k]).
\]
In particular, $\ctOpNci[\ctOpRoot](\realTorusESplitCtCover[\pSeq{j}{k}]) = \realTorusESplitCtCover[\pSeq{j}{k-1}]$ is an even parity Cartan subgroup of $\realGroupCover$ and we proceed as above.
\end{steplist}
\end{proof}

We will need an extension of Theorem \ref{theoremDualityCommutesWithCA} to the full abstract Weyl group $\weylGroup(\complexLieAlgebra, \absLieAlgebra)$. The precise statement requires a bit of setup. Let $\absDiff \in \absLieAlgebraDual$ be a half-integral infinitesimal character and suppose $\famInfChar[\absDiff]$ is a family for $\absDiff$ (Definition \ref{defFamInfChar}). For each $\kappa \in \famInfChar[\absDiff]$ we define a map $\pcMap_{\kappa} : \hcBasisGenCentSpinInf{p}{q}{\kappa} \to \pcCentInf{p}{q}{\kappa}$ as in Definition \ref{defPsMap} and let
\[
\pcMap' = \coprod_{\kappa \in \famInfChar[\absDiff]}\psMap_{\kappa} : \coprod_{\kappa \in \famInfChar[\absDiff]}\hcBasisGenCentSpinInf{p}{q}{\kappa} \to \coprod_{\kappa \in \famInfChar[\absDiff]}\pcCentInf{p}{q}{\kappa}.
\]
Identifying $\pcCentInf{p}{q}{\kappa}$ with $\pcCent{p}{q}$ in the obvious way gives a map 
\[
\pcMap : \coprod_{\kappa \in \famInfChar[\absDiff]}\hcBasisGenCentSpinInf{p}{q}{\kappa} \to \pcCent{p}{q}.
\]
Construct a map $\pcMapDual$ for $\realGroupCoverDual$ in a similar fashion and fix a bijection 
\[
\pcInvMap : \pcCent{p}{q} \to \pcCenter{p^{\vee}}{q^{\vee}}{\chi^{\vee}}
\]
as above. Define maps $\weylGroupInvMap_{\kappa} : \hcBasisGenCentSpinInf{p}{q}{\kappa} \to \hcBasisGenCentSpinInfDual{p}{q}{\kappa}$ for $\kappa \in \famInfChar$ (Definition \ref{defPsiOnHcBasisGen}) and set
\[
\weylGroupInvMap = \coprod_{\kappa \in \famInfChar}\weylGroupInvMap_{\kappa} : \coprod_{\kappa \in \famInfChar[\absDiff]}\hcBasisGenCentSpinInf{p}{q}{\kappa} \to \coprod_{\kappa \in \famInfChar[\absDiff]}\hcBasisGenCentSpinInfDual{p}{q}{\kappa}.
\]

\begin{theorem}
\label{theoremDualityCommutesWithECA}
In the notation above, let $\simpleRoot \in \rootSystem(\complexLieAlgebra, \absLieAlgebra)$ be a long abstract root and suppose $\rootReflection \notin \weylGroup(\absDiff)$. If $\kappa \in \famInfChar[\absDiff]$ and $\hcRep \in \hcBasisGenCentSpinInf{p}{q}{\kappa}$ then
\begin{eqnarray*}
\weylGroupInvMap(\rootReflection \cross \hcRep) & = & \rootReflection \cross \weylGroupInvMap(\hcRep).
\end{eqnarray*}
\end{theorem}

\begin{proof}
Let $\absBg[\kappa]$ be the abstract bigrading for $\hcRep$. On the level of involutions, the result follows as in Theorem \ref{theoremDualityCommutesWithCA}.
If $\hcBasisGenInvOrderInf{\inv}{\kappa} = 1$, the elements $\hcRep$ and $\weylGroupInvMap(\hcRep)$ are determined by their involutions and we are done. 

If $\left|\hcBasisGenCentSpinInf{p}{q}{\kappa}\right| = 2$ we again need to show 
\[
\psMap(\rootReflection\cross\hcRep) = \psMap(\hcRep) \iff \psMap(\rootReflection \cross \weylGroupInvMap(\hcRep)) = \psMap(\weylGroupInvMap(\hcRep))
\]
and Proposition \ref{propSuppRep} implies it suffices to check this for complex and noncompact imaginary roots. Combined with Remark \ref{remECAInfChar}, the details are as in Theorem \ref{theoremDualityCommutesWithCA} and are omitted.
\end{proof}


\subsection{Character Multiplicity Duality}
\label{ssCharMultDuality}
Let $\absDiff \in \absLieAlgebraDual$ be a half-integral infinitesimal character. We begin with one final definition.

\begin{definition}[\cite{RT3}, Definition 6.7]
\label{defLength}
Let $\inv$ be an involution in $\weylGroup(\complexLieAlgebra, \absLieAlgebra)$ and recall $\posRootSystem$ denotes the set of abstract roots that are positive for $\absDiff$.
The \emph{length} of $\inv$ is defined to be 
\[
\len[\inv] = \frac{1}{2}\left|\left\{\simpleRoot \in \posRootSystem \mid \inv(\simpleRoot) \notin \posRootSystem\right\}\right| + \frac{1}{2}\text{dim}(\inv_{-1})
\]
where $\inv_{-1}$ is the negative eigenspace for $\inv$. If $\hcRep \in \hcBasisGenSpin{p}{q}$ and $\absTriple$ is the corresponding abstract triple, we define $\len = \len[\inv]$.
\end{definition}

\begin{proposition}
In the setting of Definition \ref{defLength}, $\len \in \mathbb{N}^{+}$.
\end{proposition}

\begin{proof}
It suffices to consider the case where $\realGroupCover$ is split. Suppose $\hcRep$ is a principal series for $\realGroupCover$ so that $\inv = -1$ and let $n>0$ denote the rank of $\realGroupCover$. One easily verifies
\begin{eqnarray*}
\len & = & \frac{1}{2}n^{2} + \frac{1}{2}n \\
& = & \frac{n(n+1)}{2} = \binom{n+1}{2}
\end{eqnarray*}
and the result follows for $\hcRep$. We now proceed as in \cite{VGr}, Lemma 8.6.13.
\end{proof}

Let $\realGroupCover = \realSpinGroup{p}{q}$ with $p + q = 2n+1$, $p > q$, and recall the rank of $\realGroupCover$ is even. Suppose $\genTripleESplitInf$ is a genuine triple for $\realGroupCover$ with central character $\chi$ (Definition \ref{defCC}) and let $\realGroupCoverDual = \realSpinGroupCover{p^{\vee}}{q^{\vee}}$ be the dual group (Definition \ref{defDualGroup}) for $\realGroupCover$ and $\genTripleESplitInf$. Define the map
\[
\weylGroupInvMap : \hcBasisGenCentSpin{p}{q} \to \hcBasisGenCentSpinInfDual{p}{q}{\absDiff}
\]
as in Definition \ref{defPsiOnHcBasisGen}. The following proposition describes the effect of $\weylGroupInvMap$ on length.

\begin{proposition}
Let $\hcRep \in \hcBasisGenCentSpin{p}{q}$ and suppose $\absTriple$ is the abstract triple for $\hcRep$. Then $\len[\weylGroupInvMap(\hcRep)] = \binom{n+1}{2} - \len$.
\end{proposition}

\begin{proof}
Set
\begin{eqnarray*}
m_{1} & = & \left|\left\{\simpleRoot \in \posRootSystem \mid \inv(\simpleRoot) \notin \posRootSystem \right\}\right| \\
m_{2} & = & \text{dim}(\inv_{-1}).
\end{eqnarray*}
Clearly we have
\begin{eqnarray*}
\len[\weylGroupInvMap(\hcRep)] = \len[\weylGroupInvMap(\inv)] & = & \len[-\inv] \\
& = & \frac{n^{2}-m_{1}}{2} + \frac{n-m_{2}}{2} \\
& = & \frac{n(n+1)}{2} - \frac{m_{1}+m_{2}}{2} \\
& = & \binom{n+1}{2} - \len
\end{eqnarray*}
as desired.
\end{proof}

Fix a family $\famInfChar[\absDiff]$ of infinitesimal characters for $\absDiff$ and let $\mathcal{B} = \left\{\gamma_{1},\ldots,\gamma_{r}\right\} = \hcBasisGenCentSpin{p}{q}$. If  $\delta_{i} = \weylGroupInvMap(\gamma_{i})$, write $\mathcal{B}' = \left\{\delta_{1},\ldots,\delta_{r}\right\} = \hcBasisGenCentSpinInfDual{p}{q}{\absDiff}$ and extend the map $\weylGroupInvMap$ (and thus the sets $\mathcal{B}$ and $\mathcal{B}'$) as in Theorem \ref{theoremDualityCommutesWithECA}. Let $\mathcal{M}$ (respectively $\mathcal{M}'$) denote the free $\mathbb{Z}[q,q^{-1}]$ module with basis $\mathcal{B}$ (respectively $\mathcal{B}'$). As in \cite{RT3}, we view $\mathcal{M}$ and $\mathcal{M}'$ as Hecke modules for the extended action of the Hecke algebra $\mathcal{H}(\weylGroup)$ (\cite{RT3}, Definition 9.4). The integer matrix $\MultMatrix$ (Definition \ref{defCMDuality}) for $\mathcal{B}$ (respectively $\mathcal{B}'$) is then determined from the combinatorics of $\mathcal{M}$ (respectively $\mathcal{M}'$). The interested reader is referred to \cite{VPc} for a reasonably concise account of this process.

For our purposes, only the following formalism is important. Define the dual $\mathbb{Z}[q,q^{-1}]$ module
\[
\mathcal{M}^{*} = \text{Hom}_{\mathbb{Z}[q,q^{-1}]}(\mathcal{M},\mathbb{Z}[q,q^{-1}])
\]
and extend $\mathcal{M}^{*}$ to an $\mathcal{H}(\weylGroup)$-module as in \cite{RT3}, Theorem 11.1. Write $\check{\mathcal{B}} = \left\{\check{\gamma}_{1},\ldots,\check{\gamma}_{r}\right\}$ for the dual basis of $\mathcal{M}^{*}$ and define the $\mathbb{Z}[q,q^{-1}]$-linear map
\begin{eqnarray*}
\Phi: \mathcal{M}^{*} & \longrightarrow & \mathcal{M}' \\
\check{\gamma}_{i} & \longmapsto & (-1)^{\len[{\gamma_{i}}]}\delta_{i}.
\end{eqnarray*}

\begin{theorem}
\label{theoremHeckeModuleDuality}
In the above setting, $\Phi$ is an isomorphism of $\mathcal{H}(W)$-modules.
\end{theorem}

\begin{proof}
It suffices to check the equivariance of the operators in $\mathcal{H}(\weylGroup)$ corresponding to simple roots. Depending on the root type and length, there are many cases to consider. For integral roots, the details are as in \cite{IC4}, Proposition 13.10. For strictly half integral roots, the details are as in \cite{RT3}, Theorem 11.1. In each case the result is a formal consequence of Theorems \ref{theoremDualityCommutesWithCA}, \ref{theoremDualityCommutesWithCT}, and \ref{theoremDualityCommutesWithECA}.
\end{proof}

\begin{theorem}
\label{theoremMainSecond}
In the above setting
\begin{equation}
\label{eqnCMDuality}
\Mult{\gamma_{i}}{\gamma_{j}} = (-1)^{\len[\gamma_{j}] -\len[\gamma_{i}]} \mult{\delta_{j}}{\delta_{i}}.
\end{equation}
\end{theorem}

\begin{proof}
This follows immediately from Theorem \ref{theoremHeckeModuleDuality} and Lemma 13.7 of \cite{IC4}.
\end{proof}

\begin{example}
Let $\realGroupCover = \realSpinGroupCover{3}{2}$ and fix $\absDiff = \left(\frac{3}{2},1\right)$. Write $\simpleRoot$ (respectively $\simpleRootb$) for the unique short (respectively long) abstract simple root in $\posRootSystem$. Suppose $\hcRep \in \hcBasisGenSpin{p}{q}$ is a principal series representation for which $\simpleRoot$ does not satisfy the parity condition and write $\chi$ for the corresponding central character. From Definition \ref{defDualGroup} we see
\begin{eqnarray*}
p^{\vee} & = & 2 \cdot 1 + 0 + 0 + 1  = 3\\
q^{\vee} & = & 2 \cdot 1 + 0 + 0 + 0 = 2
\end{eqnarray*}
and $\realGroupCoverDual = \realSpinGroupCover{3}{2} = \realGroupCover$. The structure of 
\[
\mathcal{B} = \left\{\gamma_{0},\ldots,\gamma_{8}\right\} = \hcBasisGenCentSpin{3}{2}
\]
is given in the following table
\[
\begin{array}{|c|c|c|c|c|}
\hline
\mathcal{B} & \text{length} & \rootReflection \times \gamma_i & \simpleRootb:\ctOp[](\gamma_i) & \simpleRoot:\ctOp[](\gamma_i) \\
\hline
\gamma_0 & 0 & \gamma_1 & \gamma_2 & \gamma_4 \\
\gamma_1 & 0 & \gamma_0 & \gamma_3 & \gamma_4 \\
\gamma_2 & 1 & \gamma_5 & \gamma_0 & * \\
\gamma_3 & 1 & \gamma_6 & \gamma_1 & * \\
\gamma_4 & 1 & \gamma_4 & * & \left\{\gamma_1,\gamma_{0}\right\} \\
\gamma_5 & 2 & \gamma_2 & \gamma_7 & * \\
\gamma_6 & 2 & \gamma_3 & \gamma_8 & * \\
\gamma_7 & 3 & \gamma_7 & \gamma_5 & * \\
\gamma_8 & 3 & \gamma_8 & \gamma_6 & * \\
\hline
\end{array}.
\]

Each row in the table corresponds to the element $\gamma_{i} \in \mathcal{B}$ listed in the first column. The second column gives the length of $\gamma_{i}$ and the third column gives the image of the (integral) cross action for $\rootReflection$. The final two columns give the images of the Cayley transforms (when defined) for the simple roots $\simpleRootb$ and $\simpleRoot$, respectively. 

If $\delta_{i} = \weylGroupInvMap(\gamma_{i}) \in \hcBasisGenCentSpinCent{3}{2}{\chi^{\vee}}$, the structure of $\mathcal{B}' = \left\{\delta_{8},\ldots,\delta_{0}\right\}$ is given by the following table
\[
\begin{array}{|c|c|c|c|c|}
\hline
\mathcal{B}' & \text{length} & \rootReflection \times \delta_i & \simpleRootb:\ctOp[](\gamma_i) & \simpleRoot:\ctOp[](\gamma_i) \\
\hline
\delta_8 & 0 & \delta_8 & \delta_6 & * \\
\delta_7 & 0 & \delta_7 & \delta_5 & * \\
\delta_6 & 1 & \delta_3 & \delta_8 & * \\
\delta_5 & 1 & \delta_2 & \delta_7 & * \\
\delta_4 & 2 & \delta_4 & * & \left\{\gamma_1,\gamma_{0}\right\} \\
\delta_3 & 2 & \delta_6 & \delta_1 & * \\
\delta_2 & 2 & \delta_5 & \delta_0 & * \\
\delta_1 & 3 & \delta_0 & \delta_3 & \delta_4 \\
\delta_0 & 3 & \delta_1 & \delta_2 & \delta_4 \\
\hline
\end{array}.
\]
Using the methods of \cite{RT3} one verifies the matrix $\MultMatrix$ for $\mathcal{B}$ is given by
\[
\MultMatrix = \left(\begin{array}{rrrrrrrrr}
1 & 0 & -1 & 0 & -1 & 1 & 1 & 0 & -1 \\
& 1 & 0 & -1 & -1 & 1 & 1 & -1 & 0 \\
& & 1 & 0 & 0 & -1 & 0 & 0 & 0 \\
& & & 1 & 0 & 0 & -1 & 0 & 0 \\
& & & & 1 & -1 & -1 & 0 & 0 \\
& & & & & 1 & 0 & -1 & 0 \\
& & & & & & 1 & 0 & -1 \\
& & & & & & & 1 & 0 \\
& & & & & & & & 1
\end{array}\right)
\]
with respect to the ordering above. Similarly, the matrix $\multMatrix$ for $\mathcal{B}'$ is given by
\[
\multMatrix = \left(\begin{array}{rrrrrrrrr}
1 & 0 & 1 & 0 & 0 & 0 & 0 & 0 & 1 \\
& 1 & 0 & 1 & 0 & 0 & 0 & 1 & 0 \\
& & 1 & 0 & 1 & 1 & 0 & 1 & 1 \\
& & & 1 & 1 & 0 & 1 & 1 & 1 \\
& & & & 1 & 0 & 0 & 1 & 1 \\
& & & & & 1 & 0 & 1 & 0 \\
& & & & & & 1 & 0 & 1 \\
& & & & & & & 1 & 0 \\
& & & & & & & & 1
\end{array}\right)
\]
with respect to the opposite order. Theorem \ref{theoremMainSecond} implies the matrix $\MultMatrix$ equals the antitranspose (i.e., reflection about the opposite diagonal) of the matrix $\multMatrix$ up to sign. The reader is invited to verify this for the above matrices.
\end{example}


\bibliographystyle{amsplain}
\bibliography{main}
\end{document}